\documentclass[a4paper,11pt]{article} 
 
\usepackage{fancybox} 
\usepackage{pifont} 

\usepackage[symbol]{footmisc}

\usepackage{dsfont}
\usepackage{amsmath} 
\usepackage[applemac]{inputenc}
\usepackage{amsfonts}
\usepackage{amssymb}
\usepackage{amsthm}
\usepackage{stmaryrd}
\usepackage{color}
\usepackage{mathrsfs} 

\usepackage{alphalph}

\usepackage[allcolors={blue}]{hyperref}

\newtheorem{lemma}{Lemma}

\newcommand{\mysection}{\setcounter{equation}{0} \section}

\newcommand{\N}{\mathbb{N}}

\newcommand{\R}{\mathbb{R}}

\newtheorem{defi}{Definition}
\newtheorem{THM}{Theorem}   
\newtheorem{prop}{Proposition} 
\newtheorem{remark}{Remark}   

\newtheorem{cor}[THM]{Corollary} 
\newtheorem{PROP}[prop]{Proposition}

\def\1{\mbox{1\hspace{-0.25em}l}}

\newcommand \A[1]{{\bf (#1)}}

\def\leftB{[\![}
\def\rightB{]\!]}
\def\btheta{{\boldsymbol{\theta}}}

\def\0{{\mathbf{0}}}

\title{\textbf{
		Transport equations in H\"older space by vanishing viscosity and applications}
}
\author{\textbf{Igor Honor\'e}\footnote{Univ Lyon, CNRS, Université Claude Bernard Lyon 1, UMR5208, Institut Camille Jordan, F-69622 Villeurbanne, France. E-mail: honore@math.univ-lyon1.fr}}

\begin{document}
\maketitle

\begin{abstract}
	
We obtain 
a sharp \textit{limit H\"older continuity} of the solution 
 for the transport equations thanks to a vanishing viscosity analysis. We also derive the same control for parabolic equations and  for inviscid Burgers' equation.

Eventually,  under a structural hypothesis on the coefficients, we provide existence and uniqueness
	of a H\"older continuous solution. 
	
%
\end{abstract}

{\small{\textbf{Keywords:} 
		Transport equations, Parabolic equations, 
		Besov spaces, 
		Inviscid Burgers' equation.}}

{\small{\textbf{MSC:} Primary: 35K40, 35J75; Secondary: 35A02, 
		46E35}}

\tableofcontents

\mysection{Introduction}

\subsection{Statement of the problem}

For a given $d \in \N$, we consider the following $d$-dimensional Cauchy problem:
\begin{eqnarray}
\label{transport_equation}
\begin{cases}
\partial_t u(t,x)+ \langle b(t,x),  \nabla u(t,x)\rangle =f(t,x),\ (t,x)\in \R_+ \times \R^{d},\\
 u(0,x)=g(x),\ x\in \R^{d}.
 \end{cases}
 \end{eqnarray}
 We suppose that the transport coefficient, $b(t,\cdot)$, 
 $t \in \R_+$, lies in the non-homogeneous Besov H\"older space $  B_{\infty,\infty}^{- \beta }$, for a finite $\beta \in \R$. 

One of our main result in this article, 
we derive a \textit{limit H\"older continuity modulus} of the mollified and viscous version of the solution $u$ of \eqref{transport_equation}, this is a generalisation of the control of \cite{cha:jeo:23}.  Under some structural assumptions on $b$, 
we show that the limit solution is unique if $-\beta\in (0,1)$  is large enough.
The solution of transport equation \eqref{transport_equation} is approached with 
the second order parabolic equations whose second order term $\nu$, called viscosity, goes to $0$,
	\begin{equation}\label{Parabolic_Equation_Moll_Itro}
\begin{cases}
\partial_t u^{m,\nu}(t,x)+  \langle b_m(t,x),  \nabla u^{m,\nu} (t,x)\rangle -\nu \Delta u^{m,\nu}(t,x)=f_m(t,x),\ (t,x)\in (0,T]\times \R^{d},\\
u^{m,\nu}(0,x)=g_m(x),\ x\in \R^{d};
\end{cases}
\end{equation}
the function $b_m$ is a mollified version of the distributional valued $b$, also $g_m$ and $f_m$ stand respectively for a mollified version of $g$ and $f$.
\\


We write the solution as a perturbation of a PDE with constant components. These constants correspond to the first order term $b_m$ taking at a \textit{freezing point} throw the corresponding flow, as done in \cite{chau:hono:meno:18}.
 However, to estimate the \textit{limit H\"older norm}, we have to distinguish two regimes, like for Schauder estimates in a parabolic context, the \textit{diagonal} and the \textit{off-diagonal} ones.
In each regime, the choice of \textit{freezing points} changes in order to get a negligible first order contribution when $\nu \to 0$, or with a time decomposition trick.
\\




One of the crucial consequences of our analysis is 
a  general meaning of a classic product of distributions,
$\langle b(t,x),  \nabla u(t,x)\rangle $, where $b $ and $\nabla u$  have negative regularity.
This is written as a weak limit of a sub-sequence of a smooth parabolic approximation.
 The price to pay in this representation is that we do not have usual uniqueness of the limit in this rough case.
\\

Thanks to the techniques developed for the transport equation \eqref{transport_equation}, we also succeed to extend  our analysis to parabolic equation \eqref{Parabolic_Equation_Moll_Itro} without vanishing viscosity. Our result includes negative regularity for $b$ without use of rough path theory\footnote{see \cite{lyon:98}.}.

The \textit{a priori} estimates 
being independent on $b$, we are able to use a fixed-point argument to handle with the inviscid Burgers' equation
\begin{equation}\label{BURGERS_EQUATION}
\begin{cases}
	\partial_t u(t,x)+  u(t,x)  \partial_x u(t,x) 
	=f(t,x),\ (t,x)\in (0,T]\times \R^{d},\\
	u(0,x)=g(x),\ x\in \R^{d},
\end{cases}
\end{equation}
see \textit{e.g.} \cite{bori:kuks:21} for more details on this equation.
 We deduce, 
 some kind of uniqueness of the solution $u$ of \eqref{BURGERS_EQUATION} in  case of the existence of H\"older continuous solution (true before a blow-up time).
Precisely, on the one hand there is uniqueness of selection whatever the choice of vanishing viscosity, and on the other hand there is uniqueness whatever the way to mollify; nevertheless we do not succeed to get both uniquenesses for the same solution (which should imply usual uniqueness).

\subsection{Existing results}
\label{sec_exist_intro}

\subsubsection{On the transport equations}

The Lipschitz framework is classic \textit{via} the characteristic method. Indeed if $b \in L^\infty([0,T];C^1(\R^d,\R^d))$, considering the ODE $\dot X_t = b(t,X_t)$, thanks to the Cauchy–Lipschitz theorem, there is a unique solution $u \in L^\infty([0,T];C^1(\R^d,\R))$ of the transport equation \eqref{transport_equation}.
Out of this regular context,  the analysis has to be more involved.

For instance, a meaning of the equation \eqref{transport_equation}, when the coefficients are in a suitable Sobolev space, can be given by a renormalisation procedure developed by DiPerna and Lions  \cite{dipe:lion:89}. 
When $b$ is divergence free and Log-Lipschitz continuous, \cite{baho:chem:94} establish that there is a unique solution with a loss of regularity in the Sobolev spaces. 
If $b \in L^1([0,T];W^{1,1}_{loc}(\R^d,\R))$ and $ div (b) \in L^1([0,T];L^\infty(\R^d,\R)) $, they establish that the Cauchy problem \eqref{transport_equation} is well-posed in $L^\infty$.

When, $b$ is only supposed to have bounded variations in space, Ambrosio \cite{ambr:04} extends this result for $b \in L^1([0,T];BV_{loc}(\R^d,\R))$, $div (b)_- \in L^1([0,T];L^\infty(\R^d,\R))$.

For other references on transport equation in the non Lipschitz case, see for instance to \cite{mode:szek:18}, 
\cite{xiao:19}.

If $b$ is only H\"older continuous then the Cauchy problem \eqref{transport_equation} is not well-posed any-longer, 
this is illustrated by the well-known counter-example 
\begin{equation}\label{counter_example}
	b(x)= \frac 1{1-\gamma} sign(x) (|x| \wedge R) ^\gamma,  \gamma \in (0,1).
\end{equation}  
With a multiplicative noise, Flandoli Gubinelli  and Priola  \cite{flan:gubi:prio:10}, see also \cite{flan:gubi:prio:12} and \cite{moll:oliv:17}, establish that the following Stochastic Partial Differential Equation
\begin{equation}\label{EDPS}
\begin{cases}
d_t  u+\langle b, \nabla   u \rangle dt +\nabla   u \circ dW_t=0,\\   u(0,\cdot)=   u_0(\cdot),
\end{cases}
\end{equation}
with $b \in L^\infty([0,T],C_b^\alpha)$ and $ div (b) \in L^p$ is well-posed. Here, the symbol $\circ$ corresponds to the stochastic Stratonovich integral.
This is a typical consequence of the regularisation by the noise, let us mention \cite{fred:flan:13}, \cite{atta:flan:11}, \cite{cate:16}.


 

This stochastic approach seems to be hopeless to get uniqueness by zero limit noise selection, indeed from
\cite{atta:flan:09} : if $b$ is defined as in \eqref{counter_example}, then the following equation
\begin{equation*}
\partial_t u^\varepsilon + b  \cdot \partial_x u^\varepsilon = \varepsilon\nabla u^\varepsilon \circ dW_t,
\end{equation*}
has a weak convergence of the corresponding probability $P^\varepsilon$ towards $\frac{\delta_{u_1}+ \delta_{u_2}}{2}$, when $\varepsilon \to 0$, where $u_1 $ and $u_2$ are two different solutions of the associated transport equation.
Other counter-examples are stated in  \cite{depa:03}.
\\

In \cite{ciam:crip:spir:20}, the authors show that
there is $b \in L_{loc}^p$, $p \in [1,\frac 43]$, such that
\begin{equation*}
\partial_t u^m + b_m  \cdot \nabla u^m =0,
\end{equation*}
where $b_m$ is a regularisation of $b$, for $d=3$,  and such that there is no uniqueness of bounded distributional solutions when $m \to + \infty$.
The conclusion is the same, in \cite{dele:giri:22}, for
a compactly supported divergence-free vector field $ b \in L^\infty$.
To put it another way, there is no smooth selection principle by regularisation.
\\

Finally, let us mention \cite{colo:crip:sore:22}, where the authors build a first order coefficients $b$ H\"older continuous such that the vanishing viscosity procedure from a parabolic approximation yields to several different solutions, uniqueness fails to be true in $L^2$.
In spite of this counter-example, 
there is no contradiction with our uniqueness result, stated in Theorem \ref{THEO_SCHAU_non_bounded_optimal}, as we only consider uniqueness in a H\"older space, defined further. The non-uniqueness seems to occur only in non-smooth functional spaces.
  Then we are able to affirm that  a selection principle by vanishing viscosity may happen, which is a partial positive answer to the question (Q3) in \cite{ciam:crip:spir:20}. 
\\

We propose in this article a new approach to handle with the determinist transport equation \eqref{transport_equation},
 by \textit{vanishing viscous solution}, see e.g. \cite{evan:98}.


For the best author's knowledge, the notion of vanishing viscosity has been already used for several classes of evolution PDEs, e.g. hyperbolic ones in \cite{bian:bres:05},
but it was not developed to establish the regularity control of the solution of a general transport equation.
Finally, with this method, we also deduce existence and uniqueness of a H\"older continuous solution of the inviscid Burgers' equation.

\subsubsection{On the parabolic equations}

Historically, a \textit{full} control of the parabolic equation \eqref{Parabolic_Equation_Moll_Itro} with bounded H\"older continuous 
coefficients, called 
Schauder estimates was first proved by Friedman \cite{frie:64} thanks to a \textit{parametrix} approach.
Let us mention also the major references of the parabolic equation:
 Ladyzenskaja, Solonnikov and  Ural'ceva  \cite{lady:solo:ural:68},
 Krylov 
\cite{kryl:96} and Lieberman \cite{lieb:96}.
\\

The first article handle with unbounded coefficients in \eqref{Parabolic_Equation_Moll_Itro} is due to Krylov and Priola \cite{kryl:prio:10}.
In parabolic and elliptic framework, they 
get the parabolic bootstrap through Schauder estimates.
Let us also mention \cite{kruz:kast:lope:75}, for a first partial result in a unbounded context; 
 when coefficients are ``merely" measurable in time is handled in \cite{lore:11}.
\\

%

In a degenerate framework, 
with some H\"ormander conditions, Lunardi establishes, in her  work \cite{lun:97}, Schauder estimates for a linear $b$.
For a fully non-linear H\"older continuous drift, $b$, Chaudru de Raynal et al.  \cite{chau:hono:meno:18} get the corresponding controls, see also \cite{prio:06}.
To extend our result to a degenerate chain, the principal points would be to control the regularity gain of the flow $\btheta$ associated with $b$ through the chain which is not direct when the non-degenerate components of the drift are distributional valued. 
This regularity control is crucial as the \textit{proxy} density does depend on this flow.
\\

When $b$ is a tempered distribution, in particular when $b \in B_{\infty,\infty}^{- \beta}$ with $\beta < \frac{2}{3}$ some controls of the solution of \eqref{Parabolic_Equation_Moll_Itro}  are established in 
\cite{dela:diel:16} and \cite{cann:chou:18}, in order to build a ``polymer measure", an important object to study some stochastic partial differential equations such as the Kardar-Parisi-Zhang equation.
\\

In \cite{danc:07} and \cite{baho:chem:danc:11}, when $b$ lies in a more general Besov space, $B_{p,q}^{-\beta}$, $p,q \neq + \infty$, some \textit{a priori} estimates in Besov space is established for the solution of some parabolic equations.
In particular, these controls are crucial in some approach to handle with the Navier-Stokes equation.




The paper is organised as following.
The notations and the definitions used are gathered in Section \ref{sec_def}.
A first result about the transport equation under existence assumption
is stated 
in Section \ref{sec_main_transport}.
 We provide a discussion on the consequence in term of distributions product in Section \ref{sec_prod_distri}.
The complete analysis is detailed in Section \ref{sec_Proof_transport}.
In particular, we write an extension on the structural assumption (denoted by \A{A} in the following) in Section \ref{sec_time_depend}. 
Comments and results on the parabolic equation and as well as the full proof are featured  in Section \ref{sec_parab}. 

 Our last result is stated 
 and as well the proof of the regularity of a solution to the inviscid Burgers' equation in Section \ref{sec_Buregers}.

We gather in Section \ref{sec_apriori_global}, some results about regularity controls on the solution of linear parabolic of second order 
and on the non-linear Burgers' case.


Also, in Appendix,  some property of the Besov spaces are developed in Sections \ref{sec_conv_molli_Dpsi}-\ref{sec_free_mol}.
Precisely, in Section \ref{sec_conv_molli_Dpsi}, we establish that the space $C_b^\infty$ functions are dense in the space of multi-differentiated H\"older continuous functions. Some inequalities over the norm of Besov-H\"older distributions with their derivative are established in Section \ref{sec_interpom}.
Also in Section \ref{sec_free_mol}, we detail why the limit of regularised distribution in Besov-H\"older space does not depend on the choice of mollification procedure.

Eventually, some precise and long comments of the analysis are postponed in Section \ref{sec_com}.

\mysection{Notations and Definitions}
\label{sec_def}

	From now on, we denote by $C>0$ and $c>1$ generic constants that may change from line to line but only depends on known parameters such as $\gamma$, $d$. Importantly, these constants do not depend on $\beta$.

For $\varepsilon, \tilde \varepsilon>0$, we use the usual notation of asymptotic domination:
\begin{equation}\label{def_ll}
	\varepsilon \ll \tilde \varepsilon, \textit{ if } \frac{\varepsilon}{\tilde \varepsilon} \longrightarrow 0.
\end{equation}
We also write $	\xrightarrow[(\varepsilon,\tilde \varepsilon)\to (0,0)]{\eqref{def_ll}}
$ or $	\underset{(\varepsilon,\tilde \varepsilon)\to (0,0)}{\overset{\eqref{def_ll}}{\lim}}
$, the  limit, up to some subsequence selection, under the condition \eqref{def_ll}.

\subsection{Tensor and Differential notations} 


For any $z \in  \R^d$, we use the decomposition $z= z_1 e_1+ \hdots z_d e_d$, where $(e_1,\hdots, e_d)$ is the canonical base of $\R^d$.

We usually use the notation $\partial_t$ for the derivative in time $t \in [0,T]$ also $\partial_{z_k}$, $k \in \N$, is the derivative in the variable $z_k$.

The gradient in space is denoted by $\nabla$,
in other words $\nabla=\partial_{z_1} e_1 + \hdots + \partial_{z_d} e_d$,
the divergence by 
$\nabla \cdot=div$ and is defined for any $\R$-function $f: \R^d \mapsto \R$ by $\nabla  \cdot f= \sum_{k=1}^d \partial_{z_k} f$.

The symbol ``$\cdot$" between two tensors is the usual tensor contraction. 
For example, if $M \in \R^d \otimes \R^d \otimes\R^d$ and $N \in \R^d$ then $M \cdot N$  is a $d \times d$ matrix. If the two considered tensors are vectors then ``$\cdot$" matches with the scalar product which is also denoted by $\langle \cdot , \cdot \rangle$.
\\

For any $ \R^d \mapsto \R$, we define the Hessian matrix $D_z^2f=\big (\partial_{z_i} \partial_{z_j} f \big )_{1 \leq i,j\leq d}$, and the usual Laplacian operator $\Delta f=\sum_{1 \leq i \leq d} \partial_{z_i}^2  f$.

More generally, for any $k \in \N$, $D^k_zf$ denotes the  order $k$ tensor $(\partial_{z_{i_1}} \hdots \partial_{z_{i_k}}f)_{(i_1, \hdots, i_k) \in \leftB 1,d \rightB^k}$. For any multi-index $\alpha=(\alpha_1,\hdots,\alpha_d) \in \N_0^d$, we write $D^\alpha_z f =  \partial_{z_1}^{\alpha_1} \hdots \partial_{z_k}^{\alpha_k} f$, in particular if, for $ i \in  \leftB 1,d \rightB$, $\alpha_i=0$, there is no derivative in $z_i$ in the expression of $D^\alpha_z f$. 

We also denote for any $\alpha=(\alpha_1,\cdots,\alpha_m)\in \N^m $, the order of this multi-index by $|\alpha|=\sum_{i=1}^{m} \alpha_i $.

\subsection{Associated H\"older, Besov spaces}
\label{SEC_HOLDER_BESOV_SPACE}

 
\subsubsection{H\"older spaces}
\label{sec_def_Holder}

For any, 
$\tilde \gamma\in (0,1) $, 
$\|\cdot\|_{C^{\delta}(\R^m,\R^\ell)},\ m\in \{1,d\} $, $\ell \in \{1,d, d\otimes d \} $\footnote{we write $\R^{d\otimes d}$ for $\R^d \otimes \R^d$ the space of square matrices of size $d$.} is the usual homogeneous H\"older norm, see e.g. Lunardi \cite{luna:95} or Krylov \cite{kryl:96}. Precisely, for all $\psi \in C^{\delta}(\R^m,\R^\ell) $,   
we set the semi-norm:
\begin{equation}
\phantom{BOUHHH} \| \psi\|_{C^\delta}  
=[\psi]_\delta 
:= \sup_{(x,y)\in (\R^m)^2,x\neq y} \frac{| \psi(x)- \psi(y)|}{|x-y|^\delta},
 \label{USUAL_HOLDER_SPACE}
\end{equation} 
the notation $|\cdot| $ is the Euclidean norm on the considered space. 
We denote by:
$$C_b^{\delta }(\R^m,\R^\ell):=\{\psi \in C^{\delta }(\R^m,\R^\ell): \|\psi\|_{L^{\infty}(\R^m,\R^{\ell})}<+\infty\},$$  
the associated subspace with bounded elements (non-homogeneous H\"older space).
The corresponding H\"older norm is defined by:
\begin{equation}
\label{BD_HOLDER}
\|\psi\|_{C_b^{\delta}(\R^m,\R^\ell)}:=\|\psi\|_{C^{\delta}(\R^m,\R^\ell)}+\|\psi\|_{L^\infty(\R^m,\R^\ell)}.
\end{equation}
For the sake of notational simplicity, from now on we write:
\begin{equation*}
\|\psi\|_{L^\infty}:=\|\psi\|_{L^\infty(\R^{d},\R^\ell)}, \
\|\psi\|_{C^{\delta}}:=\|\psi\|_{C^{\delta}(\R^{d},\R^\ell)}
, \
\|\psi\|_{C_{b}^{\delta}}:= \|\psi\|_{C_{b}^{\delta}(\R^{d},\R^\ell)}.
\end{equation*}
For time dependent functions, $\varphi_1 \in L^{\infty}\big ([0,T ],C_{b}^{\delta}(\R^{m},\R^\ell) \big) $ and $ \varphi_2 \in   L^{\infty}\big ([0,T ],C^{\delta}(\R^{m},\R^\ell) \big)$, we define the norms:
\begin{eqnarray*}
\|\varphi_1\|_{L^\infty(C_{b}^{\delta})}&:=& \sup_{t \in [0,T]}\|\varphi_1(t, \cdot)\|_{C_{b}^{\delta}(\R^{m},\R^\ell)},
\nonumber \\
\|\varphi_2\|_{L^\infty(C^{\delta})}&:=& \sup_{t \in [0,T]} \|\varphi_2(t, \cdot)\|_{C^{\delta}(\R^{m},\R^\ell)},
\nonumber \\
\|\varphi_2\|_{L^\infty}&:=& \sup_{t \in [0,T]} \|\varphi_2(t, \cdot)\|_{L^\infty(\R^{m},\R^\ell)}
.
\end{eqnarray*}

 
 The test functions for some weak formulations of different solutions will be 
 in $C_0^\infty(\R^d,\R)$, which corresponds to the space of smooth functions infinitely differentiable with bounded derivatives and with a compact support. 
 
 
 For any sequence of functions $(\psi^{m,\varepsilon})_{m,\varepsilon>0} $ lying in  $C^\delta(\R^d,\R)$, 
 we denote, for any $\eta>0$ by 
 \begin{equation}
 	[\psi^{m,\varepsilon}]_{\delta,\eta}:=
 	\sup_{(x,y)\in (\R^m)^2,|x-y| \leq \eta} \frac{| \psi^{m,\varepsilon}(x)- \psi^{m,\varepsilon}(y)|}{|x-y|^\delta},
 	\label{Def_modulus_continuity_nonlimit}
 \end{equation}  
 the local H\"older modulus without the limit.
 
 For any function $F : \R_+^2  \rightarrow \R_+$ s.t.
 \begin{equation}\label{condi_F}
 	F(\varepsilon,m) \ll 1,
 \end{equation}
 and
 similar to \cite{gilb:trud:83} and \cite{cha:jeo:23}, we introduce the \textit{limit H\"older continuity} as the modulus of continuity defined by
 \begin{equation}
 	[\psi^{m,\varepsilon}]_{\delta,\eta}^{\eqref{condi_F}}:=
 	 \overset{\eqref{condi_F}}{\lim_{(m,\varepsilon) \to (+ \infty,0)}} \sup_{(x,y)\in (\R^m)^2,|x-y| \leq \eta} \frac{| \psi^{m,\varepsilon}(x)- \psi^{m,\varepsilon}(y)|}{|x-y|^\delta}.
 	\label{Def_modulus_continuity}
 \end{equation}  
 \begin{remark}
This semi-norm reveals a partial regularity of a limit function, when $\varepsilon\to 0$. The condition $|x-y| \leq \varepsilon$ allows to avoid some potential singularities arising; without this criterion, we fail to use the \textit{cut-locus} trick and so to prove the \textit{a priori} control of the H\"older modulus, where obviously 
\begin{equation*}
	\max( 	[\psi^{m,\varepsilon}]_{\delta,0}^{\eqref{condi_F}},  	[\psi^{m,\varepsilon}]_{\delta,\eta}) \leq  	[\psi^{m,\varepsilon}]_{\delta}.
\end{equation*}
 \end{remark}

The above modulus, can be replaced through all the article by a fractional derivative defined, for any $\gamma \in (0,1)$,  
for all $\varphi \in C^\infty_b (\R^d,\R)$ and $x \in \R^d$:
 \begin{equation}\label{def_grunwald_deriv}
 	D^\gamma \varphi (x):= \lim_{h \to 0} \sum_{k=0}^{+ \infty} \binom{\gamma}{k} \frac{\varphi(x)- \varphi(x-kh)}{h^\gamma},
 \end{equation}
 with 
 \begin{equation}\label{def_combi}
 	\binom{\gamma}{k} := \frac{\gamma (\gamma-1)(\gamma-2)\cdots (\gamma-k+1)}{k!},
 \end{equation}
the generalised binomial coefficient.

This operator is called the Gr\"unwald-Letnikov derivative, see for instance \cite{diet:10}.

 \subsubsection{Thermic characterization of the Besov space}
 \label{subsec_besov_trieb_def}

%
We define the Besov spaces thanks to a thermic characteristic, see Triebel \cite{trie:83} Section 2.6.4.
For all $\alpha \in \R$, $q\in (0,+\infty]$, $p \in (0,\infty] $, 
\begin{equation}
\label{def_norm_besov_inhomo}
\|f\|_{ B_{p,q}^\alpha}:= \|\varphi(D) f\|_{L^p(\R^d)}+ \|f\|_{ \ddot B_{p,q}^\alpha}, \text{ with }
\|f\|_{\ddot B_{p,q}^\alpha}:= 
\Big(\int_0^1 \frac {dv}{v} v^{(m-\frac \alpha 2)q}    \|\partial_v^m h_{v}\star f\|_{L^p(\R^d)}^q \Big)^{\frac 1q},
\end{equation}
where we define the heat kernel
\begin{equation}\label{def_hv}
h_{v}(z):=\frac{1}{(2\pi v)^{ d/2}}\exp\left(-\frac{|z|^2}{2v} \right),
\end{equation}
and $\varphi(D)f:= (\varphi \hat f)^{\vee} $ with $\varphi \in C_0^\infty(\R^d)$  such that
$\varphi(0)\neq 0 $,   $\hat f$ and $(\varphi \hat f)^\vee $ respectively denote the Fourier transform of $f$ and the inverse  Fourier transform of $\varphi \hat f $.
Note that, when $\alpha> d(\frac{1}{p}-1)_+=d\max(0,\frac{1}{p}-1)$ then in \eqref{def_norm_besov_inhomo}, it possible to replace $\|\varphi(D) f\|_{L^p(\R^d)}$ by 
$\| f\|_{L^p(\R^d)}$.

 When $p=q=+\infty $, we naturally write:
 \begin{equation*}
\|f\|_{ \ddot B_{p,q}^\alpha}= 
\sup_{v \in [0,1]}  v^{m-\frac \alpha 2}    \|\partial_v^m h_{v}\star f\|_{L^\infty(\R^d)},
\end{equation*}
and if $\alpha > d(\frac 1p -1)_+$,
 \begin{equation*}
\|f\|_{ B_{p,q}^\alpha}= \|f\|_{L^\infty} + \sup_{v \in [0,1]}  v^{m-\frac \alpha 2}    \|\partial_v^m h_{v}\star f\|_{L^\infty(\R^d)}.
\end{equation*}
We carefully point out that 
 the homogeneous term $\|f\|_{ \ddot B_{p,q}^\alpha}$ does not define a norm associated to a Banach space.
To consider the whole  homogeneous Besov space we have to consider $v \in \R_+$ in the definition \eqref{def_norm_besov_inhomo}, for $\alpha <0$, see e.g. Theorem 2.34 in \cite{baho:chem:danc:11}.
Somehow, for the inhomogeneous norm defined in \eqref{def_norm_besov_inhomo}, the contribution of the heat kernel convolution for $v>1$ is ``hidden" in the inhomogeneous term $\|\varphi(D) f\|_{L^p(\R^d)}$.

For $\alpha>0$, the homogeneous and respectively inhomogeneous  H\"older spaces match with  Besov space, namely $C^\alpha= \dot B_{\infty,\infty}^\alpha$ and $C_b^\alpha=B_{\infty,\infty}^\alpha$, see \cite{trie:83} for details.

Our analysis tackles with inhomogeneous Besov spaces, in order to extend our analysis to the homogeneous ones some sophisticated changes should be performed as the homogeneous Besov spaces are \textit{a priori} not Banach spaces;
and we should consider the realisation of the space of homogeneous Besov spaces as a space of distributions defined quotiented by polynomials, see e.g. Proposition 3.8 in \cite{lemar:02} to make it a Banach space.

If $\alpha<0$ then it is known that $\dot B_{p,q}^\alpha \subset B_{p,q}^\alpha$, i.e. there is a constant $C>0$ such that, for any $\alpha<0$,
\begin{equation}\label{ineq_Besov_homo_inhomo}
\| \cdot \|_{ B_{p,q}^\alpha}  \leq - \frac{C}{\alpha} \| \cdot \|_{\dot B_{p,q}^\alpha}.
\end{equation}

We also introduce the distributions that can be approached by a mollification procedure.
We put a \textit{tilde} in order to mean that we consider the closure of $C^\infty_b$ in the considered space\footnote{As in Proposition 3.6 in \cite{lemar:02} for the closure Schwartz space, but we do not need the mollified versions of the considered distributions to be rapidly decreasing functions.}. Namely, for any $(\alpha,p,q) \in \R \times (1, + \infty] \times (1, + \infty]$ we define:
\begin{equation}
\tilde B_{p,q}^\alpha:= \textbf{cl}_{B_{p,q}^\alpha}(C^\infty_b), \
\tilde {\dot B}_{p,q}^\alpha:= \textbf{cl}_{\dot B_{p,q}^\alpha}(C^\infty_b). 
\end{equation}

\begin{remark}\label{remark_density_Cinfty}
If there is $\psi \in C^\gamma$, $ \gamma \in (0,1)$, such that $b= D^\alpha \psi$, $\alpha \in \N_0^d$, so $b \in \dot B_{\infty,\infty}^{-\beta}$  
with $-\beta= -|\alpha| + \gamma$, and in Appendix Section \ref{sec_conv_molli_Dpsi}, we show that $b \in {\tilde {\dot B}}_{\infty,\infty}^{-\beta}$.

The last constraint on $b$, being the derivative of a H\"older function, is quiet natural when we consider the structure theorem of the tempered distributions $\mathcal S'$,
see Theorem 8.3.1 in \cite{frie:98}.
We recall indeed that any $b \in \mathcal S'$ writes $b= D^\alpha \psi$ where $\alpha \in \N_0^d$ and $\psi$ is a continuous function with polynomial growth. 


Out of the Besov-H\"older space, namely if $ 1  \leq p,q < +\infty$, then $\tilde B_{p,q}^\alpha=  B_{p,q}^\alpha$ and $\tilde {\dot B}_{p,q}^\alpha=  \dot B_{p,q}^\alpha$, see Theorem 4.1.3 in \cite{adam:hedb:96}, Proposition 2.27 and Proposition 2.74 in \cite{baho:chem:danc:11}.
For more Besov properties, we also mention \cite{peet:76} and \cite{jawe:77}.

We might consider the non-homogeneous low-frequency cut-off in the Littlewood-Paley characterisation instead of usual mollification by convolution, as performed in the current paper, and adapt Lemma 2.73 in \cite{baho:chem:danc:11} for the space $B_{\infty,\infty}^{-\beta}$. 
\end{remark}
%
%
\subsubsection{Besov duality}

In our analysis, we thoroughly use the Besov duality.
The full proof of the duality of Besov spaces is established for example in Proposition 3.6 in \cite{lemar:02} thanks to a Littlewood-Paley decomposition.
\begin{PROP}\label{prop_dualite}
	For all $1\leq p,q\leq + \infty$ and $\alpha \in \R$, we have for all $ \varphi, \psi \in \mathcal S '$:
	\begin{equation*}
	\Big |\int_{\R^d} \varphi(y) \psi (y) dy \Big | \leq C_{d,p,q,\alpha} \|\varphi\|_{B_{p,q}^{\alpha}} \|\psi\|_{B_{p',q'}^{-\alpha}},
	\end{equation*}
	with $1 \leq p',q' \leq + \infty$ such that $\frac{1}{p}+ \frac{1}{p'}=1$ and $\frac{1}{q}+ \frac{1}{q'}=1$.
	%
\end{PROP}

\begin{proof}[Sketch of the proof]
	Let us suppose that w.l.o.g. that $1<p,q<+\infty$ (the analysis is identical if we suppose that $1<p',q'<+ \infty$).
	It is known that $B_{p,q}^{\alpha} (\R^d,\R)$ and $B_{p',q'}^{-\alpha} (\R^d,\R)$ are in duality (Proposition 3.6 in \cite{lemar:02}). Precisely, $B_{p,q}^{ \alpha}$ is the dual of the closure of the Schwartz class $\mathcal S$ in $B_{p',q'}^{-\alpha}$. But $\mathcal S$ is dense in  $B_{p,q}^{\alpha}$ (see for instance 
		4.1.3.
 in \cite{adam:hedb:96}).
	%
	
\end{proof}
It is possible to adapt this result to homogeneous spaces, 
see for instance Proposition 2.29 in \cite{baho:chem:danc:11}.

\subsubsection{Usual tools for the Gaussian function}
\label{sec_Gaussian_properties}

One of the reason to use the thermic representation of the Besov space comes from a well-known and important result about the Gaussian function: 
for any $\delta >0$, there is $ C_{\delta}=C_{\delta}(\delta)>1$ such that:
\begin{equation}\label{ineq_absorb}
\forall x \in \R^d, \ |x|^\delta e^{-|x|^2} \leq C_{\delta} e^{-C_{\delta}^{-1} |x|^2}.
\end{equation}

Furthermore, we often use the cancellation principle: for all $f \in C^\gamma$, $\gamma \in (0,1)$, $x \in \R^d$ and $\sigma>0$
\begin{equation}\label{eq_cancell}
\nabla _x \int_{\R^d} e^{-\frac{|x-y|^2}{2\sigma}} f(y) dy =  \int_{\R^d} \nabla_x e^{-\frac{|x-y|^2}{2\sigma }} [f(y)-f(x)] dy,
\end{equation}
as the Gaussian function, up to a renormalisation by a multiplicative constant, is a probabilistic distribution, hence $\nabla_x \int_{\R^d} e^{-\frac{|x-y|^2}{2\sigma}} dy =0$.
Hence, we obtain,
\begin{eqnarray*}
	(2\pi \sigma)^{\frac{d}{2}}\Big | \nabla_x \int_{\R^d} e^{-\frac{|x-y|^2}{2 \sigma}} f(y) dy \Big | 
	&\leq& 
	(2\pi \sigma)^{\frac{d}{2}} [f]_\gamma  \int_{\R^d} e^{-\frac{|x-y|^2}{2\sigma}}\frac{|y-x|}{\sigma} |y-x|^\gamma dy
	\nonumber \\
	&\leq& 
	{C}_{\gamma}
	(2\pi \sigma)^{\frac{d}{2}} [f]_\gamma \sigma^{\frac{\gamma-1}{2}} \int_{\R^d} e^{-{C}_{\gamma}^{-1}\frac{|x-y|^2}{2\sigma}} dy
	\nonumber \\
	&=& {C}_{\gamma} [f]_\gamma \sigma^{\frac{\gamma-1}{2}}.
\end{eqnarray*}
The penultimate identity comes from the absorbing property \eqref{ineq_absorb}.

\subsection{Different definitions of a solution to the transport equation problem}
\label{sec_def_type_solu}

As said in the introduction, we need to carefully defined the suitable notion of solution, no strong solution can be in a negative Besov space or even in a H\"older space, implying a product of distributions which is obviously not point-wisely defined.

\begin{defi}[mild vanishing viscous]\label{DEFINITION_MILD}
	A function $u$ is said to be a \textit{mild vanishing viscous} solution in $L^\infty\big ([0,T]; C_b^{\gamma}(\R^d,\R)\big )$ of equation \eqref{transport_equation} if for a sequence $(b_m)_{m \in \N}$ in $L^\infty  ( [0,T];C_b^\infty(\R^d,\R^d) )$ such that there is $\beta \in \R$,
	\begin{equation}\label{converg_b_def_mild_viscous}
	\forall \varepsilon>0, 
	\ 
	\lim_{m \to + \infty}\|b_m-b\|_{ L^\infty  ( [0,T];  B_{\infty,\infty}^{-\beta-\varepsilon}(\R^d,\R^d) )}=0,
	\end{equation}
	for any $t \in [0,T]$,  there exists a sub-sequence of  $(u^{m,\nu}(t,\cdot))_{(m,\nu) \in \R_+^2}$ lying in $
	C_b^{\gamma}(\R^d,\R) $ converging in the space $
	C_b^{\gamma-\tilde \varepsilon}(K,\R) $, $ 0< \tilde \varepsilon<\gamma$, for any compact subset $K \subset \R^d$, when $\nu \to 0$ and $m \to + \infty$  towards $u(t,\cdot) \in  
	C_b^{\gamma}(K,\R) $ and satisfying, for all $m \in \N$ and $\nu \in \R_+$,
	\begin{equation}\label{parabolic_moll_def}
	\begin{cases}
	\partial_t u^{m,\nu}(t,x)+  \langle b_m(t,x),  \nabla u^{m,\nu} (t,x)\rangle -\nu \Delta u^{m,\nu}(t,x)=f_m(t,x),\ (t,x)\in (0,T]\times \R^{d},\\
	u^{m,\nu}(0,x)=g_m(x),\ x\in \R^{d},
	\end{cases}
	\end{equation}
\end{defi}
where $(f_m,g_m) \underset{m \to + \infty}\longrightarrow (f,g)$ in $L^\infty([0,T]; C^\gamma(\R^d,\R))\times C^\gamma(\R^d,\R)$.
\\

We point out that such a sequence $(b_m)_{m \geq 0}$ exists if $b \in L^\infty([0,T];\tilde{ B}_{\infty,\infty}^{-\beta}(\R^d,\R^d))$; or in particular if $b$ is the derivative of a bounded H\"older continuous function but in this former case the limit result \eqref{converg_b_def_mild_viscous} has to be in the homogeneous space $ L^\infty([0,T];\tilde{ \dot  B}_{\infty,\infty}^{-\beta}(\R^d,\R^d))$, see Appendix Section \ref{sec_conv_molli_Dpsi}, and identity \eqref{converg_b_def_mild_viscous} is implied by \eqref{ineq_Besov_homo_inhomo}. 

Moreover, it is important to notice that the choice of sub-sequence may depend on the current time $t$.
We do not succeed in getting uniform continuity in time $t$ (only boundedness in $L^\infty$), thus the impossibility to apply a suitable compact argument in time, we need to consider the problem at a fixed $t$.
Nevertheless, for the sake of simplicity we write for the sub-sequence $u^{m,\nu}(t,\cdot)$ instead of a notation of the kind $u^{m_t,\nu_t}(t,\cdot)$.

\begin{remark}\label{rem_int_u_mild}
We could consider another formulation of \textit{mild vanishing viscous} solution, where the considered function is 
 	\begin{equation}\label{def_w}
 w^{m,\nu}(t,x)= \int_0^t u^{m,\nu}(s,x) ds.
 \end{equation}
 Under structural assumption \A{A} and the result stated in 
 Theorem \ref{THEO_SCHAU_non_bounded} below, we have that $(t,x) \mapsto w^{m,\nu}(t,x)$ lies, uniformly in $(m,\nu)$, in $C^1_b([0,T]; C^\gamma_b(\R^d,\R))$, by the  Arzel\`a-Ascoli theorem, we obtain a convergence in all compacts $[0,T] \times K$ of $[0,T] \times \R^d$ towards a function $w\in  C^1_b ([0,T]; C_b^{\gamma}(K,\R) )$.
\end{remark}

Let us define an alternative form of solution which is a mixed version between \textit{mild} and \textit{weak} notions of solutions.

\begin{defi}[mild-weak solution]\label{DEFINITION_MILD_WEAK}
	A function $u$ is a mild-weak solution in $L^\infty ([0,T]; C_b^{\gamma}(\R^d,\R) )$ of equation \eqref{transport_equation} if $u$ is a \textit{mild vanishing viscous} solution,
and 
	such that 
	for any function $\varphi \in C_0^\infty([0,T]\times \R^d,\R)$, we have, up to a sub-sequence selection, for any $t \in [0,T]$,
	\begin{eqnarray}\label{KOLMO_mild_weak}
 \lim_{m \to  \infty, \nu \to 0}\int_{\R^d} \Big \{   \varphi(t,y) u^{m,\nu}(t,y) 
  + \int_0^t  
 \big \{-\partial_t\varphi(s,y) u^{m,\nu}(s,y)+  \langle b_m(s,y),  \nabla u^{m,\nu}(s,y)\rangle \varphi(s,y) 
\big \} ds \Big \} dy  
	\nonumber \\
=  \int_{\R^d}   \varphi(0,y) g(y) dy+   \int_{\R^d}\int_0^t  \varphi(s,y) f(s,y)  ds \, dy .
	\nonumber \\
	\end{eqnarray}
\end{defi}
	The distributional formulation allows to give a sense to the potential irregularities of $b$ and of $ \nabla u^{\nu,m}$ when $\nu \to 0$, $m \to + \infty$, and to consider the whole space $[0,T] \times \R^d$ (instead of any compact), the cut-off of $\R^d$ required for compact argument
	is included in the test function $\varphi$.
%

\begin{defi}[weak solution]\label{DEFINITION_WEAK}
	A function $u$ is a weak solution in $L^\infty\big ([0,T]; C_b^{\gamma}(\R^d,\R)\big )$ of equation \eqref{transport_equation} if $u$ is a \textit{mild vanishing viscous} solution and satisfies, for any function $\varphi \in C_0^\infty([0,T]\times \R^d,\R)$,
	\begin{eqnarray}\label{KOLMO_weak}
	&&\int_{\R^d} \Big \{   \varphi(t,y) u(t,y) 
	+ \int_0^t  
	\big \{-\partial_t\varphi(s,y) u(s,y)+  \langle b(s,y),  \nabla u(s,y)\rangle \varphi(s,y) 
	\big \} ds \Big \} dy  
	\nonumber \\
	&=& \int_{\R^d}   \varphi(0,y) g(y) dy +\int_{\R^d} \int_0^t    \varphi(s,y) f(s,y) ds \, dy .
	\end{eqnarray}
\end{defi}

\begin{remark}
	We cannot hope to define classic solution in our irregular context.
	Indeed, even if  roughly speaking $\partial_t u +  \langle b ,\nabla u \rangle $ is supposed to lie in $L^\infty(C_b^\gamma)$, 
	we cannot \textit{a priori} define point-wisely the classic scalar product between $b$ and $\nabla u$. If $b$ has a blow up at a point $x_0 \in \R^d$ then, as $u$ is solution of \eqref{parabolic_moll_def}, $\lim_{x \to x_0} \int_0^t \langle b(s,x), \nabla u(s,x)\rangle ds $ is necessary finite for any $t \in [0,T]$ but we cannot give a meaning of $ \langle b(t,x_0), \nabla u(t,x_0) \rangle $ in a point-wise sense. Roughly speaking, to handle distributional drift we have to stay in a distributional formulation of the solution.
\end{remark}

\begin{defi}[Uniqueness]\label{DEFINITION_Unique}
	There is a unique solution of \eqref{transport_equation} if for two solutions $u^{m,\nu}$ and $\bar u^{m,\bar \nu}$ of 
	\begin{equation*}
		\begin{cases}
			\partial_t u^{m,\nu}(t,x)+  \langle b_m(t,x),  \nabla u^{m,\nu} (t,x)\rangle -\nu \Delta u^{m,\nu}(t,x)=f_m(t,x),\ (t,x)\in (0,T]\times \R^{d},\\
			u^{m,\nu}(0,x)=g_m(x),\ x\in \R^{d},
		\end{cases}
	\end{equation*}
	and respectively
	\begin{equation*}
		\begin{cases}
			\partial_t \bar u^{m,\bar \nu}(t,x)+  \langle \bar b_m(t,x),  \nabla \bar u^{m,\bar \nu} (t,x)\rangle -\bar \nu \Delta \bar u^{m,\bar \nu}(t,x)=f_m(t,x),\ (t,x)\in (0,T]\times \R^{d},\\
			\bar u^{m,\bar \nu}(0,x)=g_m(x),\ x\in \R^{d},
		\end{cases}
	\end{equation*}
	for $\nu,\bar \nu>0$,
	with
	\begin{equation}\label{converg_b_def_mild_viscous_bis}
		\forall \varepsilon>0, 
		\ 
		\lim_{m \to + \infty}\|b_m-b\|_{ L^\infty  ( [0,T];  B_{\infty,\infty}^{-\beta-\varepsilon}(\R^d,\R^d) )}+\lim_{n \to + \infty}\|\bar b_n-b\|_{ L^\infty  ( [0,T];  B_{\infty,\infty}^{-\beta-\varepsilon}(\R^d,\R^d) )}=0,
	\end{equation}
	converging, up to sub-sequence selection, towards two H\"older continuous solutions $u$, $\bar u$, when $(m,\nu,\bar \nu ) \to (+ \infty,0,0)$, then $u= \bar u$. 
\end{defi}
The mollification procedure associated with $\bar b_n$ can be different from the one performed for $b_m$.


\mysection{Statement under existence assumption}
\label{sec_main_transport}

When $b$ lies in a H\"older-Besov space, we succeed in obtaining
a \textit{maximum principle} like for the \textit{limit H\"older continuity}, matching
 the regularity of 
 $f$ and $g$.
The type of solution strongly depends on the regularity of $b$, however to be sure that our approximation converges towards a true H\"older continuous solution, we need to introduce an assumption framework.

\A{A} We suppose that there is a function $\psi : \R^d \rightarrow \R^d$, such that 
\begin{equation*}
	\forall (t,x) \in [0,T] \times \R^d, \ b(t,x) = (\nabla \psi (x))^{-1},
\end{equation*}
 satisfying
	\begin{equation*}
	\forall (\tau,x) \in [0,T] \times \R, \	\frac{ \nabla \psi  (x)}{\nabla  \psi \big (\psi^{-1}(c\tau +\psi (x))\big )}<+ \infty,
\end{equation*}
and
\begin{equation*}
	\psi_m ^{-1}\big (C(t-s)+\psi_m (x)\big ) \underset{m \to + \infty}{\longrightarrow} 		\psi^{-1}\big (C(t-s)+\psi (x)\big ) \text{ in } C^1(\R^d, \R^d),
\end{equation*}
with $\psi_m$ a mollified version of $\psi$, e.g. $\psi_m := \rho_m \star \psi$.
\\

For more details, see Section \ref{sec_example_b} further, in particular, we develop the case $\psi(x)= |x|^{1+\beta}$ in Section \ref{sec_example_polynome}.
	\\
	
	We could suppose that $b$ does depend on the time $t$, if $b$ is smooth enough in time, thanks to the time decomposition performed in Section \ref{sec_decoup_temps}.
	The only use of this assumption is to get the convergence of the solution $u^{m,\nu}$ toward $u$ in a suitable H\"older space.
	

%

\begin{THM}[Rough transport equation in H\"older spaces] \label{THEO_SCHAU_non_bounded_optimal}
	For  
	$\beta \in \R^*$  and $0 <\gamma <1$ be given, let $f\in L^\infty([0,T];  C_b^{\gamma}(\R^d,\R))$ and $ g \in C_b^{\gamma}(\R^d,\R)$.
	For a distribution $ b \in L^\infty([0,T], \tilde B_{\infty,\infty}^{-\beta} (\R^d,\R^d))$, 
then for $u^{m,\nu}$ defined in \eqref{parabolic_moll_def}
we have
%
%
	 then, 
	\begin{eqnarray}
		\label{ineq_THEO_SCHAU}
	\lim_{n \to + \infty}		\sup_{t \in [0,T]}	[ u^{m,\nu}(t,\cdot)]_{\gamma,(\nu T/n)^{1/2}}
		&\leq&
		 	\lim_{n \to + \infty}	\int_0^T [ f(t,\cdot)]_{\gamma,(\nu T/n)^{1/2}}+ 		 	\lim_{n \to + \infty} [ g]_{\gamma,(\nu T/n)^{1/2}},
		\nonumber \\
		\|  u^{m,\nu} \|_{L^\infty} &\leq&  T\| f \|_{L^\infty}+ \|  g \|_{L^\infty}  .
	\end{eqnarray}
Moreover under if \A{A} is in force there is 
a
\textit{mild  vanishing viscosity} solution $u \in L^\infty([0,T];C_b^{\gamma}(\R^{d},\R))  $ of \eqref{transport_equation}
and
	\begin{trivlist}
		\item[i) 
		\textbf{Incompressibility.}] If $\beta <\gamma$ and $\nabla \cdot b=0$ then the solution $u$ is also a \textit{weak}
		and  a \textit{mild-weak}
		solution.
		\item[ii) \textbf{Positive regularity.}] If $\beta <-1+\gamma$, namely if $b \in L^\infty([0,T]; C_b^{\beta}(\R^d,\R^d))$, $\beta=-\tilde \gamma>1-\gamma$,  the solution $u$ is also a weak solution 
		and if we also suppose that
		\begin{equation}\label{CONDINU_Holder_bis}
			\nu \|\nabla ^2g_m\|_{L^\infty} 
			+ 
			\nu ^{\frac{\gamma}{2}}	m^{1-\tilde \gamma}	
			\ll 1,
		\end{equation}
		then
		$\partial_t u (t,\cdot) \in B_{\infty,\infty}^{-1+\gamma}(\R^d,\R)$, for any $t \in (0,T]$.
		\item[iii) \textbf{Greater regularity.}] If $\beta = - \tilde \gamma > \frac{1}{1+\gamma}$, and if
		\begin{equation}\label{CONDINU_UNIQ}
			\nu \|\nabla ^2g_m\|_{L^\infty}  +	m^{
				1-2\tilde \gamma} \nu ^{\frac{\gamma-1}{2}} + m^{1-\tilde \gamma} \nu^{\frac \gamma 2}\ll 1,
		\end{equation}
	the solution is   unique\footnote{Namely 
		the limit solution does not depend on the choice of sequence $(m,\nu)$ and on the way to mollify $b$.}.
	
\end{trivlist}
\end{THM}



Except for the control of the time derivative and for uniqueness, 
there is no condition on the vanishing viscosity unlike in Theorem \ref{THEO_SCHAU_non_bounded}.
This is due to the possibility to use the time decomposition trick which allows to get estimates independent on $b$, see Section \ref{sec_time_depend}.

\begin{remark}
	We derive a \textit{limit H\"older modulus} in a more general setting than \cite{cha:jeo:23} and \cite{driv:elgi:la:23} where $f$ is supposed to be null, and the consider modulus is a ${\rm log}$-H\"older and an Osgood modulus, see also
 \cite{driv:elgi:23}.
Let us insist that estimate \eqref{ineq_THEO_SCHAU} does not depend on the condition \A{A} which only allows to consider the limit solution $u^{m,\nu}$ in a H\"older space.
\end{remark}

\begin{remark}
	
	Importantly, the above controls \eqref{ineq_THEO_SCHAU} do not depend on the \textit{drift} $b$; which  seems to be paradoxical with the usual gradient control of the transport equation by characteristic method.
	This independence on $b$ is crucial to consider very rough coefficients as well as non-linear equation such as the inviscid Burgers' equation studied in Section \ref{sec_Buregers}.
	
	While for the usual weak solution, there is no more such infinite regularisation effect. 
	The \textbf{Incompressibility} framework, 
	i.e. $ \nabla \cdot b=0$, allows to still consider a negative regularity of $b$.
	Such divergence free condition for non-smooth distributions already exists for instance for Leray's solution of Navier-Stokes equation \cite{lera:34}.
	
	In the last case, i.e. \textbf{Positive regularity}, the considered \textit{drift} $b$ is supposed to be H\"older continuous in space, in particular lying in $L^\infty([0,T]; C_b^\alpha(\R^d,\R^d))$, $\alpha>1-\gamma$ which is the Bony's para-product assumption, see Section \ref{sec_prod_distri} below for more details. 

	
	We even obtain uniqueness for the \textbf{Greater regularity} case, 
	we use a kind of regularisation by turbulence which makes negligible the second order term. 
	The condition on $g$ means that the initial function has to be smooth enough such that $\nu \|\nabla ^2g_m\|_{L^\infty}  $ is negligible, obvious if $g \in C^2(\R^d,\R)$.
	Again, this a partial positive answer to the question (Q3) in \cite{ciam:crip:spir:20}, but we fail to consider uniqueness in a rough (negative regularity) case,
	see Remark \ref{Rem_uniq_negativ} further.

	In this H\"older case, the product $\langle b, \nabla u \rangle$ falls into the usual case of the para-product by the Bony's microlocal analysis \cite{bony:81}.
	In the negative regularity case, we again obtain a result on the \textit{limit H\"older modulus} of the product of distributions.

	
\end{remark}

%
%


For the sake of completeness, we introduce some examples of coefficient $b$ where the existence of solution $u \in L^\infty([0,T],C^\gamma_b(\R^d,\R^d))$ is granted.
In particular, in the following section, we detail a framework taking into account the Peano counter-example and the Coulombian interactions.


\subsection{The flow $\theta_{ s,t }$ under hypothesis \A{A}}
\label{sec_example_b}

\subsubsection{A structural example in $\R^d$}

If the coefficient of \eqref{transport_equation} is, for any $x \in \R^d$,
\begin{equation*}
	b(x)= c (\nabla \psi (x))^{-1},
\end{equation*}
where $\psi : \R^d \to \R^d$ is a function  such that
\footnote{Which is granted if there is $c>0$ such that $\forall x,y \in \R^d$, $c|y|^2<\langle y, \nabla \psi(x)y\rangle <c^{-1}|y|^2 $ yielding that $b \in L^\infty(\R^d,\R^d)$.}
\begin{equation*}
	\forall (\tau,x) \in [0,T] \times \R, \	\frac{ \nabla \psi  (x)}{\nabla  \psi (\psi^{-1}(c\tau +\psi (x)))}<+ \infty.
\end{equation*}

Then the associated flow is
\begin{equation*}
	\theta_{ s,t } (x)=\psi ^{-1}(c(t-s)+\psi (x)).
\end{equation*}
Indeed, differentiating w.r.t. $s$ gives
\begin{equation*}
	\dot \theta_{ s,t }(x)= c \nabla (\psi ^{-1})(c(t-s)+ \psi (x))= \frac{c}{\nabla \psi \circ \psi ^{-1} (c(t-s)+ \psi (x))}= b(\psi ^{-1}(c(t-s)+\psi (x)))= b(\theta_{ s,t }(x)).
\end{equation*}
Furthermore,
\begin{equation*}
	\nabla \theta_{ s,t } (x)=\frac{ \nabla \psi  (x)}{ \nabla \psi  (\psi^{-1}(c(t-s)+\psi (x)))},
\end{equation*}
which is in $L^\infty$.

Recalling that we obtain the example in Section \ref{sec_example_polynome} below, taking $\psi(x)= |x|^{1+\beta}$.
\\
%
 \subsubsection{A polynomial example}
 \label{sec_example_polynome}
 
 
 For a given constant $c>0$, and $\beta \geq 0$, if the transport coefficient write\footnote{The definition of the norm $|x|$ does not change the analysis. The term $x|x|^{-1}$ stands for a multidimensional version of the ${\rm sign}$ function.}
 \begin{equation}\label{exemple_poly}
 b(x)= c x |x|^{-(1+\beta )} \in B_{\infty,\infty}^{-\beta}(\R^d , \R^d),
 \end{equation}
 then we can see that a flow can be written as
 \begin{equation}\label{theta_exemple}
 \theta_{ s,t }(x)= 
 \frac{x}{|x|}	\Big ( c (1+\beta)(t-s)+|x|^{1+\beta}\Big )^{\frac{1}{1+\beta}}.
 \end{equation}
 Indeed, differentiating in $s$, 
 \begin{equation*}
 \dot \theta_{ s,t }(x)= 
 -	\frac{x}{|x|} c 	\Big ( c (1+\beta)(t-s)+|x|^{1+\beta}\Big )^{-\frac{\beta}{1+\beta}}
 .
 \end{equation*}
 and 
 \begin{eqnarray*}
 	b(\theta_{ s,t }(x))
 	&=& c 	\frac{x}{|x|}	\Big ( c (1+\beta)(t-s)+|x|^{1+\beta}\Big )^{-\frac{\beta }{1+\beta}} 
 	\nonumber \\
 	&=&-\dot \theta_{ s,t }(x).
 \end{eqnarray*}
 Furthermore, taking the gradient of the flow,
 \begin{equation}
 \nabla	\theta_{ s,t }(x)= 
 \frac{(\beta (1-\beta)(t-s)+|x|^{1+\beta})^{-\frac{\beta}{1+\beta}}}{|x|^{-\beta}}.
 \end{equation}
 We have, for $\beta\geq 0$, and 
 \begin{equation}
 \lim_{x \to 0}	\nabla \theta_{ s,t }(x) =
 \begin{cases}
 0 \text{ if } s<t, \\
 1 \text{ if } s=t.
 \end{cases}
\text{ and }
 \lim_{|x| \to +\infty}	\nabla \theta_{ s,t }(x) =  1.
 \end{equation}
 In other words, for any $0\leq s \leq t$, $\nabla \theta_{ s,t } \in L^\infty$.
 \\
 
 In order to illustrate an example of a singular $b$ satisfying all conditions \A{A}, let us consider, for any $n >0$, 
 \begin{equation*}
 b^{n}(x)=   c \frac {x}{|x|} (|x|+n^{-1})^{-\beta},
 \end{equation*}
see Appendix Section \ref{sec_free_mol} for the non importance of the mollification procedure.

We deduce from \eqref{theta_exemple}, the ``almost" associated flow,
 \begin{equation*}
 \tilde \theta_{ s,t }^n(x)= 
 \frac{x}{|x|}	\Big ( c (1+\beta)(t-s)+|x|^{1+\beta}+ n^{-1}\Big )^{\frac{1}{1+\beta}}.
 \end{equation*}
We indeed get, 
\begin{eqnarray*}
	\dot {\tilde \theta}_{ s,t }^n(x) &=&-	\frac{x}{|x|} c 	\Big ( c (1+\beta)(t-s)+|x|^{1+\beta}+ n^{-1}\Big )^{-\frac{\beta}{1+\beta}}
	\nonumber \\
	&=& - b^n( \theta_{ s,t }^n(x)).
\end{eqnarray*}
Next, we have the following gradient identity
\begin{equation}
	\nabla \tilde 	\theta_{ s,t }^n(x)
	=
	\frac{(\beta (1-\beta)(t-s)+|x|^{1+\beta}+n^{-1})^{-\frac{\beta}{1+\beta}}}{|x|^{-\beta}} \in L^\infty(\R^d,\R^d).
\end{equation}
However, we get from \eqref{theta_exemple}, before transportation
 $\tilde \theta_{ t,t }^n(x)=\frac{x}{|x|}(|x|^{1+\beta }+n^{-1})^{\frac{1}{1+\beta}}$
  which goes to 
 $x$ in $C^{1}(\R^d,\R^d)$ as $n \to + \infty$.
 \\
 
 In order to get the right terminal condition of the flow, we define the corrected flow $\theta_{ s,t }^n$ by
  \begin{equation*}
 	\tilde \theta_{ s,t }^n(x) 
 		 \textcolor{black}{=: \theta_{ s,t }^n\big (\frac{x}{|x|}(|x|^{1+\beta }+n^{-1})^{\frac{1}{1+\beta}} \big )} .
 \end{equation*}
We  equivalently\footnote{$y=\frac{x}{|x|}(|x|^{1+\beta }+n^{-1})^{\frac{1}{1+\beta}}$ implies $\frac{y}{|y|}=\frac{x}{|x|}$ and $|y|^{1+\beta}= |x|^{1=\beta}+n^{-1}$ which implies that $x= \frac{y}{|y|}(|y|^{1+\beta}-n^{-1})^{\frac{1}{1+\beta}}$.} define, for any  $y \in \R^d$,
  \begin{equation*}
	\theta_{ s,t }^n(y) 
		:= \tilde \theta_{ s,t }^n\big (\frac{y}{|y|}(|y|^{1+\beta }-n^{-1})^{\frac{1}{1+\beta}} \big ) .
\end{equation*}

The analysis for this example follows exactly the computations below,
 see Section \ref{sec_holder} further.
\subsection{On the product of distributions}
\label{sec_prod_distri}


The sense of some particular products of distributions is very challenging, and is related with many long-standing problems.
For instance, Hairer in \cite{hair:14} introduce a regularity structure theory which after some renormalisation allows to handle with products of distribution, and to give a meaning of stochastic partial differential equation such as KPZ \cite{hair:13}.
However, such renormalisation leads to blowing-up constants which is not the case in Theorem \ref{THEO_SCHAU_non_bounded}; the price that we have to pay is the potential non-uniqueness of the limit.
\\

From the different formulations above, we define different meanings of the product $\langle b,\nabla u\rangle$.
First of all, let us remark that by rough \textit{a priori} controls, see Lemma \ref{lemma_apriori} below,
\begin{eqnarray*}
[\Delta u^{m,\nu}(t,\cdot)]_\gamma &\leq& 2 ^{1-\gamma}\|\nabla^2  	u^{m,\nu}(t,\cdot) \|_{L^\infty}^{1-\gamma} \|\nabla^3 	u^{m,\nu}(t,\cdot) \|_{L^\infty}^\gamma
	\nonumber \\
	&\leq & 
	C	 	m^{2+\gamma}\big ( t \| f\|_{L^\infty(C^\gamma)}+[g]_\gamma\big ) 
	\exp(C  m^{1-\beta}t \|b\|_{L^\infty( B_{\infty,\infty}^{-\beta})}).
\end{eqnarray*}
Hence, if 
\begin{equation}\label{CONDINU_prod_distr}
	\nu \ll 
C	 	m^{-(2+\gamma)}\big ( t \| f\|_{L^\infty(C^\gamma)}+[g]_\gamma\big )^{-1} 
\exp(-C t m^{1+ \beta } \|b\|_{L^\infty( B_{\infty,\infty}^{-\beta})}),
\end{equation}  
then, up to subsequence choice, $\nu  \Delta u^{m,\nu}(t,\cdot) \underset{(m,\nu)\to(+ \infty,0)} \longrightarrow 0$ in  $C_b^\gamma(\R^d,\R)$.

Also, for a given $t \in [0,T]$, we see from the definition of \textit{mild vanishing viscous} solution, up to subsequence choice according to the condition \eqref{CONDINU_prod_distr}, that
\begin{equation}\label{distribution_product_mild}
\lim_{(m,\nu)\to (+\infty,0)}	\int_0^t \langle b_m(s,\cdot ),  \nabla u^{m,\nu}(s,\cdot )\rangle ds= g-u(t,\cdot )-\int_0^t f(s,\cdot) ds \in C_b^\gamma(K,\R),
\end{equation}
for any compact $K \subset \R^d$.
We highly point out that $b$ lies in any arbitrary negative regularity in space $L^\infty([0,T];B_{\infty,\infty}^{-\beta}(\R^d,\R^d))$, $\beta \in \R_+$, and $\nabla u (s,\cdot ) \in B_{\infty,\infty}^{-1+\gamma}(\R^d,\R^d)$.

In other words, thanks to the time averaging, we get a new para-product condition.
Indeed, in general from Bony's microlocal analysis \cite{bony:81}, 
 for all $\varphi \in B_{\infty,\infty}^{\alpha_1}$ and $\psi \in B_{\infty,\infty}^{\alpha_2}$, we have
\begin{equation}\label{condi_Bony}
\phi \psi \in B_{\infty,\infty}^{\alpha_1\wedge\alpha_2} \text{, if } \alpha_1+\alpha_2>0.
\end{equation}
However, the uniqueness of the limit in \eqref{distribution_product_mild} seems to be false in general.

Also in the weak formulation, from Theorem \ref{THEO_SCHAU_non_bounded}, we obtain, if $\nabla \cdot b=0$ and  $-\beta=\tilde \gamma < \gamma$ (regularity condition weaker than \eqref{condi_Bony}), a distributional meaning of $\langle b(s,\cdot ),  \nabla u(s,\cdot )\rangle $, but we still do not know in this case if the limit is unique.
\\

Finally, we point out that in the \textbf{Positive regularity} framework, the hypothesis matches with the Bony's para-product condition \eqref{condi_Bony}.
Indeed, for any $t \in [0,T]$, $b(t,\cdot) \in B_{\infty,\infty}^{\alpha}$, $\alpha=-\beta $, and $ \nabla u(t,\cdot) \in B_{\infty,\infty}^{-1+\gamma}$ with $\alpha-1+\gamma>0$.
We are able quantify the regularity, $\langle b(t,\cdot) , \nabla u(t,\cdot) \rangle \in B_{\infty,\infty}^{-1+\gamma}$ by para-product detailed further in Section \ref{sec_control_partial_t}. Moreover, the time averaging version $\int_0^t \langle b(s,\cdot) , \nabla u(s,\cdot)\rangle ds$ in the sense of \eqref{distribution_product_mild} is $\gamma$-H\"older. We remark, as $ \alpha \wedge (-1+\gamma)= -1+\gamma$ then there  is  a $+1$ gain of regularity comparing with the usual para-product result.

\mysection{Proof of Theorem \ref{THEO_SCHAU_non_bounded}}
\label{sec_Proof_transport}
\subsection{Parabolic approximation procedure}
\label{sec_approx}

	Let us first smoothen the drift and the source functions of the parabolic approximation, 
	\begin{equation}
	\label{KOLMOLLI}
	\begin{cases}
	\partial_t  u^{m, \nu}(t,x)+   \langle b_m, \nabla   u ^{m, \nu} \rangle (t,x) -  \nu \Delta  u ^{m, \nu} (t,x)=
	  f_m(t,x),\ (t,x)\in (0,T]\times \R^{d},\\
	u^{m, \nu}(0,x)=  g_m(x),\ x\in \R^{d},
	\end{cases}
	\end{equation}
	where the mollified functions are defined by 
	\begin{eqnarray}\label{def_b_epsilon}
	b_m (t,x) &:=&\int_{\R^d} \rho_m (x-y) b (t,y) dy,
	\nonumber \\
		f_m (t,x) &:=& \int_{\R^d} \rho_m (x-y) f (t,y) dy,
			\nonumber \\
		g_m (t,x) &:=& \int_{\R^d} \rho_m (x-y) g (y) dy,
	\end{eqnarray}
	for $\rho_m(\cdot ):= m^{d} \rho(m \cdot )$ where $\rho$ is a non-negative smooth function $\rho_m$, such that $\int_{\R^d} \rho_m(x-y) dy =1$.
In particular, we choose $\rho= h_1$, the heat kernel defined in \eqref{def_norm_besov_inhomo}. In Appendix Sections \ref{sec_conv_molli_Dpsi} and \ref{sec_free_mol}, we see that the limit of $b_m$ does not depend on the choice of the mollification procedure, whereas the limit of $u^{ m,\nu}$ potentially does.

In our analysis, we use some point-wise controls of the mollified functions or distributions whose blowing-up in the regularisation parameter $m$ is stated below.

\begin{lemma}\label{lemme_ineq_b_m}
	For all $m >1$, and $\beta<0$, 
	if $b \in L^\infty([0,T], B_{\infty,\infty}^{-\beta}(\R^d,\R^d))$,
	we have for any $(t,x) \in [0,T] \times \R^d$:
	\begin{eqnarray}\label{ineq_b_Db_m_lambda}
	|b_m(t,x)| &\leq& C m^{\beta} \|b\|_{L^\infty(  B_{\infty,\infty}^{-\beta})}, 
	\nonumber \\
	|\nabla  b_m(t,x)| &\leq& C
	m^{1+ \beta }  \| b\|_{L^\infty(   B_{\infty,\infty}^{- \beta})} ,
	\end{eqnarray}
	where $\nabla b_m$ stands for the Jacobian matrix of $b_m$;
	also if $\beta \leq 0$,
	\begin{eqnarray*}
	|b_m(t,x)| &\leq&  \|b\|_{L^\infty}, 
	\nonumber \\
	|\nabla  b_m(t,x)| &\leq& C
	m^{1+\beta}  \| b\|_{L^\infty( C^{- \beta})} .
	\end{eqnarray*}
\end{lemma} 

\begin{proof}[Proof of Lemma \ref{lemme_ineq_b_m}]
	From the mollification definition \eqref{def_b_epsilon}, we see, from \eqref{def_hv}, that $\rho_m=h_{m^{-2}}$, and from our scaling choice, we get for any $m>1$
	\begin{eqnarray*}
	\big | b_m(t,x)\big |  
	& =& 
	\big | \int_{\R^d}   h_{m^{-2}}( x-y) b( t,y) dy \big | 
	\nonumber \\
	&\leq&     m^{\beta}  \sup_{\tilde m^{-2} \in [0,1], \ x \in \R^d} 
	\tilde m^{-\beta}
	\big |  \int_{\R^d}   h_{\tilde m^{-2}}( x- y) b( t,  y) dy \big | .
	\end{eqnarray*}
	We readily get by the thermic definition of the Besov norm \eqref{def_norm_besov_inhomo}:
	\begin{equation*}
	 | b_m(t,x) | 
	\leq 
	m^{\beta}  \| b\|_{L^\infty( \ddot  B_{\infty,\infty}^{-\beta})}  \leq 
	 m^{\beta}  \| b\|_{L^\infty(   B_{\infty,\infty}^{-\beta})}.
	\end{equation*}
	
	For the second inequality, it is known that for any $t \in [0,T]$, $\nabla  b( t,  \cdot) \in   B_{\infty,\infty}^{-1-\beta}$, see Theorem 9 of Chapter 3 in \cite{peet:76},
	in particular
	 if $\beta  \in  (0,-1)$ see Corollary \ref{coro_Besov_ineq} in Appendix Section \ref{sec_interpom}.
	 Hence,
	\begin{eqnarray*}
		\big |  \nabla _x b_m(t,x) \big | 
		&\leq&    m^{1+ \beta }   \sup_{ \tilde m^{-2} \in [0,1], \ \tilde x \in \R^d}
		\tilde m^{-\beta} 
		\big | \int_{\R^d} h_{\tilde m^{-2}}(\tilde x- y) \nabla _y b(t,  y) dy \big | 
		\nonumber \\
		&=&
		 m^{1+ \beta }  \| \nabla  b\|_{L^\infty( \ddot {B}_{\infty,\infty}^{-1-\beta })}
		.
	\end{eqnarray*}
	We deduce that there is a constant $C=C(d)>0$ such that:
	\begin{equation*}
	\big |  \nabla _x b_m(t,x) \big | 
	\leq  C 
	 m^{1+ \beta }  \| b\|_{L^\infty( B_{\infty,\infty}^{-\beta})}
	,
	\end{equation*}
	see Corollary \ref{coro_Besov_ineq} in Appendix Section \ref{sec_interpom}.
	
	The two last inequalities, i.e. for the case $\beta \leq 0$, are standard.
\end{proof}

	\subsubsection{\textit{Proxy} choice}
\label{sec_proxy_choice}

	We approximate the Cauchy problem around the flow associated to the smooth function $b_ m$, which is unique by Cauchy-Lipschitz theorem. 
	Namely, let us consider the unique function defined 
	 for any  \textit{freezing} point $(\tau,\xi) \in [0,T] \times  \R^d$ by,
	\begin{equation}\label{def_theta}
	\theta_{s,\tau}^m(\xi):= \xi+ \int_s^\tau  b_m(\tilde s,\theta_{\tilde s,\tau}^m(\xi)) d \tilde s, \ s \in [0,\tau] .
	\end{equation}
	In other words, for any $t \in [0,\tau]$,
	\begin{equation*}
	\dot \theta_{t,\tau}^m(\xi)=- b_m(t,\theta_{t,\tau}^m(\xi)) , \ 	 \theta_{\tau ,\tau}^m(\xi) =\xi .
	\end{equation*}
	We again rewrite the system of linear parabolic PDEs \eqref{def_b_epsilon},
		\begin{equation}
	\label{KOLMOLLI_xi}
	\partial_t  u^{m,\nu} (t,x)
	+  b_m(t,\theta_{t,\tau}^m(\xi))\cdot  \nabla  u^{m,\nu}(t,x) -   \nu \Delta  u^{m,\nu}(t,x)
	 =
		b_{\Delta}^{m} [\tau,\xi](t,x) \cdot  \nabla  u^{m,\nu} (t,x) 
	+  f_m(t,x),
\end{equation}
where we have 
	\begin{equation}\label{def_Delta_b}
	  b_{\Delta}^{m} [\tau,\xi] (t,x):= b_m(t,\theta_{t,\tau}^m(\xi))-b _m(t,x).
	\end{equation}

	For such a fixed freezing point $(\tau,\xi) \in [0,T] \times \R^d$, we use the corresponding Duhamel formula:
	\begin{equation}\label{Duhamel_u}
		 u^{m, \nu} (t,x) = \hat P^{\tau,\xi}  g_m (t,x) + \hat G^{\tau,\xi} f_m(t,x) 
		+ \hat  G^{\tau,\xi}  \big ( b _{\Delta}^{m} [\tau,\xi]  \cdot  \nabla  u^{m, \nu} \big )(t,x) ,
	\end{equation}
	where we define, for any $   f\in C^{1,2}_0((0,T]\times \R^{d},\R )$, the Green operator associated with the perturbed parabolic equation  with constant coefficients \eqref{KOLMOLLI_xi},
	\begin{equation}\label{def_hat_G}
	\forall (t,x) \in (0,T]\times\R^{d}, \ \hat  G ^{\tau,\xi}   f(t,x):= \int_0^{t}  \int_{\R^{d}}  \hat{p}^{\tau,\xi} (s,t, x,y)   f (s,y) \ dy \ ds,
	\end{equation}
	and for any $   g\in C^{2}_0( \R^{d},\R)$, the associated semi-group
	\begin{equation}\label{def_hat_P}
	\hat  P^{\tau,\xi}    g(t,x):=\int_{\R^{d}}\hat  p^{\tau,\xi}  (0,t,x,y)   g(y) \ dy,
	\end{equation}
	where the perturbed heat kernel is
	\begin{equation}\label{def_hat_p}
	\hat{p}^{\tau,\xi}  (s,t,x,y)
	:= \frac{1}{(4\pi \nu (t-s))^{\frac d 2} } 
	\exp \bigg ( -\frac {\left |x+ \int_s^t  b_m (\tilde s,\theta_{\tilde s,\tau }^m(\xi))d \tilde s-y \right|^2}{4\nu(t-s)} \bigg ). 
	\end{equation}
	We carefully point out that, from definition \eqref{def_theta}, if $(\tau,\xi) =(t,x)$, 
	\begin{equation*}
	\hat p^{t,x}  (s,t,x,y)
	= \frac{1}{(4\pi \nu (t-s))^{\frac d 2} } 
	\exp \bigg ( -\frac {\left |\theta_{s,t }^m(x)-y \right|^2}{4\nu(t-s)} \bigg ). 
	\end{equation*}
	We have for each $\alpha\in \N^d$ that there is a constant $C_{\alpha}>1$ s.t.
	\begin{eqnarray}\label{FIRST_deriv_CTR_DENS_flot}
	|D^\alpha \hat  p^{\tau,\xi}(s,t,x,y)| 
	&\leq&  
	\frac{C_{\alpha}
		[ \nu (t-s)] ^{-\frac {|\alpha|}2}}{(4\pi \nu (t-s))^{\frac d 2} } 
	\exp \Big ( - C_{\alpha}^{-1}\frac {\big  |x+ \int_s^t  b_m (\tilde s,\theta_{\tilde s,\tau }^m(\xi))d \tilde s-y \big |^2}{4\nu(t-s)} \Big )
	\nonumber \\
	&=:& C [ \nu (t-s)]^{-\frac {|\alpha|}2}\bar p^{\tau,\xi}  (s,t,x,y ),
	\end{eqnarray}
	and also, after the derivatives we can choose $(\tau,\xi) =(t,x)$, and $\gamma \in [0,1]$,
	\begin{eqnarray}\label{FIRST_deriv_CTR_DENS_flot_absorb}
	|D^\alpha \hat  p^{t,x}(s,t,x,y)| \times  \big |y-x-\int_s^t b _m (\tilde s,\theta_{\tilde s,\tau }^m(\xi))d \tilde s  \big |^\gamma
	&=&
		|D^\alpha \hat  p^{t,x}(s,t,x,y)| \times  \big |y-\theta_{ s,t }^m(x) \big |^\gamma
		\nonumber \\
	&\leq&  C  [\nu (t-s)] ^{-\frac {|\alpha|}2+ \frac{\gamma}{2}} \bar p^{t,x} (s,t,x,y ),
	\end{eqnarray}
	from absorbing property \eqref{ineq_absorb}.
	It also clear, for any $0\leq s<t$, that
	\begin{equation}\label{eq_hat_p}
		\partial_t \hat  p^{\tau,\xi}(s,t,x,y) = \nu \Delta \hat  p^{\tau,\xi}(s,t,x,y) - \langle b _m (t,\theta_{t,\tau }^m(\xi)), \nabla \hat  p^{\tau,\xi}(s,t,x,y) \rangle,
	\end{equation}
	which naturally implies that the function $u^{m,\nu}$ defined in \eqref{Duhamel_u} is indeed solution to \eqref{KOLMOLLI} and to \eqref{KOLMOLLI_xi}.
	\\
	
	Finally, we will marginally use the ``pure" heat kernel already defined in \eqref{def_tilde_p},
	\begin{equation}\label{def_tilde_p}
	\tilde {p}  (s,t,x,y)= 	\hat{p}^{\tau,\xi} \Big  (s,t,x- \int_s^t  b_m (\tilde s,\theta_{\tilde s,\tau }^m(\xi))d \tilde s,y \Big )
	= \frac{1}{(4\pi \nu (t-s))^{\frac d 2} } 
	\exp \bigg ( -\frac {\left |x-y \right|^2}{4\nu(t-s)} \bigg ),
	\end{equation}
recalling the  corresponding Green operator 
	\begin{equation}\label{def_tilde_G}
	\forall (t,x) \in (0,T]\times\R^{d}, \ \tilde  G   f (t,x):= \int_0^{t}  \int_{\R^{d}}  \tilde {p} (s,t, x,y)   f (s,y) \ dy \ ds,
	\end{equation}
	and the associated semi-group
	\begin{equation}\label{def_tilde_P}
	\tilde P    g(t,x):=\int_{\R^{d}} \tilde p  (0,t,x,y)   g(y) \ dy.
	\end{equation}
	
	\subsection{Convergence of the solution}
	\label{sec_conv}
	
	To prove that $u^{m,\nu}$ indeed converges toward $u$, the analysis is similar to the one perform in Section \ref{sec_uniq} replacing $\bar u^{m,\bar \nu}$ by the solution
	\begin{equation}
		\begin{cases}
			\partial_t  u (t,x)+ \langle  b(t,x), \nabla  u(t,x) \rangle=f(t,x) ,\ t\in [0,T),\\
			\bar  u_n(0,x)=g(x),
		\end{cases}
	\end{equation}
	supposed to exist in $L^\infty([0,T],C^\gamma(\R^d,\R^d))$ because the flow $	\theta_{s,\tau}(\xi):= \xi+ \int_s^\tau  b(\tilde s,\theta_{\tilde s,\tau}(\xi)) d \tilde s, \ s \in [0,\tau] $ is assumed to be Lipschitz in space.
	
	Under assumption \eqref{CONDINU_prod_distr}, we have $\nu  \Delta u^{m,\nu}(t,\cdot) \underset{(m,\nu)\to(+ \infty,0)} \longrightarrow 0$ in  $C_b^\gamma(\R^d,\R)$.
	Also under the assumption on the limit flow in \A{A},  
	$\forall \varepsilon>0$,
	\begin{equation}\label{convergence_um_nu}
		\|u-  u^{m,\nu} \|_{L^\infty(C^{\gamma-\varepsilon})}   	
		 \underset{(m,\nu)\to(+ \infty,0)} \longrightarrow 0.
	\end{equation}

	\subsection{Compactness arguments}
\label{sec_compact}

In this section, we suppose that hypothesis \A{A} is in force, which yields existence of uniform control of $u^{m,\nu}$, because the associated flow $\theta_{ s,t }$ is Lipschitz continuous, see Section \ref{sec_example_b}. 

\subsubsection{Mild vanishing viscous}
\label{sec_mild_conv}
In order to pass to the limit $m \to + \infty$, $\nu \to 0$, according to the vanishing condition \eqref{CONDINU}, we consider a subsequence given by \eqref{convergence_um_nu} or by the usual Arzel\`a-Ascoli theorem.
However, this former result is available for uniform continuous function in a compact space.
From the lack of smoothness, uniformly on $\nu$, in space of $u^{m,\nu}(t,\cdot)$ (only $\gamma$-H\"older continuous, $\gamma<1$), we are stuck at a convergence in a compact subset of $\R^d$. For instance, the analysis performed in \cite{hono:21} to get rid of the compactness convergence criterion for quasi-linear equations does not work here as there is no hope to obtain any strong formulation of the PDE \eqref{transport_equation}. 

We do not succeed to obtain any positive regularity on $t$, which would imply dependency on $b$ 
and so we cannot exploit uniform continuity in time to get a convergence of a sub-sequence of $u^{m,\nu}$ in $[0,T] \times \R^d$. 
Thus the convergence at any given time in all compacts set of the \textit{mild vanishing viscous} solution in Theorem \ref{THEO_SCHAU_non_bounded}.
\\

Nevertheless, we are still able to include a truncation procedure into a weak formulation in order to obtain a convergence in a distributional meaning and not a point-wise one as for the \textit{mild vanishing viscous} solution.

%
%


\subsubsection{Truncation procedure}
%

The method is highly inspired by the one in \cite{hono:21},
we also consider a smooth cut-off $\vartheta_{y,R} \in \mathcal D$ supported in a ball $B_d(y,R)=\{x \in \R^d; |x-y| \leq R\} $, $y \in \R^d$ and defined by
\begin{equation}\label{def_vartheta}
	\vartheta_{y,R}(x)= \vartheta_y(\frac{x}{R}),
\end{equation}
where $\vartheta_y: \R^d \to [0,1]^d$ is function lying in $C_0^\infty(\R^d,\R^d)$ s.t. 
\begin{equation*}
	\vartheta_y(x)=\begin{cases}
		x, \ \text{ if } |x-y| < 1,	\\
		0, \ \text{ if } |x-y| > 2.
	\end{cases}
\end{equation*}
The corresponding truncated function is, for any $(t,x ) \in [0,T] \times \R^d$,
\begin{equation}\label{def_u_R}
	u_{y,R}^{m,\nu}(t,x):=  u^{m,\nu} (t,\vartheta_{y,R}(x)).
\end{equation}
We highlight the particular case
\begin{equation}\label{u_Rxx}
	u_{x,R}^{m,\nu}(t,x)=  u^{m,\nu} (t,x).
\end{equation}

The above truncation solution \eqref{def_u_R} naturally appears when we write a weak formulation of the parabolic equation \eqref{parabolic_moll_def}.

\subsubsection{Weak solution of the parabolic approximating equation}
\label{sec_weak}

For any 
smooth function $\varphi_R$ supported on $B_d(0,R)$, a d-ball of radius $R>0$ and center $(0,\hdots, 0) \in \R^d$. 
We consider a weak formulation of the parabolic solution $u^{m,\nu}$,
for any $(t,x) \in [0,T]\times \R^d$:
\begin{eqnarray*}
	\int_0^t \int_{\R^d} \Big \{ - \partial_t \varphi_R(s,y)  u^{m,\nu}(s,y) 
	+ \varphi_R(s,y) \langle  b_m(s,y) ,\nabla   u^{m,\nu}(s,y) \rangle+ \nu \Delta \varphi_R(s,y) u^{m,\nu}(s,y)  \Big \}dy \, ds
	\nonumber \\
	= \int_{\R^d} \varphi_R(0,y)  g_m(y) dy-\int_{\R^d} \varphi_R(t,y) u^{m,\nu}(t,y) dy + \int_0^t \int_{\R^d} \varphi_R(s,y)   f_m(s,y) dy \, ds ,
\end{eqnarray*}
where in l.h.s. the limit of the first order term, $\langle  b_m(s,y) ,\nabla   u^{m,\nu}(s,y) \rangle$, has \textit{a priori} no point-wise limit neither in term of the usual distributional meaning of Schwartz.
Indeed, as already enunciated in Section \ref{sec_prod_distri}, the usual distribution theory does not provide any interpretation of a product of distributions, to get any limit result we have to thoroughly use the PDE.

%

By the cut-off definition, we equivalently have
\begin{eqnarray}
	\label{eq_weak_Kolmo}
	\int_0^t \int_{\R^d} \Big \{ - \partial_t \varphi_R(s,y)  u_{0,R}^{m,\nu}(s,y) 
	+  \varphi_R(s,y) \langle  b_m(s,y) ,\nabla   u_{0,R}^{m,\nu}(s,y) \rangle+ \nu \Delta \varphi_R(s,y) u_{0,R}^{m,\nu}(s,y)  \Big \}dy \, ds
	\nonumber \\
	=   \int_{\R^d} \varphi_R(0,y) g_{m}(y) dy -\int_{\R^d} \varphi_R(s,y) u_{0,R}^{m,\nu}(s,y) dy + \int_0^t \int_{\R^d} \varphi_R(s,y)  f_{m}(s,y) dy \, ds .
	\nonumber \\
\end{eqnarray}
Now, from compact argument developed in Section \ref{sec_mild_conv}, we have that $u_{0,R}^{m,\nu}(s,\cdot)$ converges in $ C^\gamma_b(K,\R^d)$, $K=B_d(0,R)$, towards a function $u_{0,R}(s,\cdot)$ when $(m,\nu)\to (+ \infty,0)$ and the condition \eqref{CONDINU} is satisfied.
In other words, $u_{0,R}(s,\cdot)$ is a \textit{mild vanishing viscous} solution of \eqref{transport_equation}.
%
%

\subsubsection{Mild-weak solution of the transport equation}
\label{sec_existence_mild_weak}

To get a \textit{mild-weak} solution we have to pass to the limit in the weak formulation \eqref{eq_weak_Kolmo} of the mollified parabolic equation \eqref{parabolic_moll_def}.
In equation \eqref{eq_weak_Kolmo}, up to a sub-sequence selection, except for the first order term $ \varphi_R(s,y) \langle  b(s,y) ,\nabla   u_{0,R}^{m,\nu}(s,y) \rangle  $, each contribution obviously has the good converge property by the Arzel\`a-Ascoli theorem.
In particular, from \eqref{ineq_Linfty}, we have
\begin{equation}\label{lim_int_Laplace_u_weak}
	\nu	\int_0^t \int_{\R^d}  \Delta \varphi_R(s,y) u_{0,R}^{m,\nu}(s,y)  dy \, ds
	\xrightarrow[(m,\nu)\to (+ \infty,0)]{\eqref{CONDINU}}0.
\end{equation}
To deal with the \textit{drift} part, we write by integration by parts
\begin{eqnarray}
	&&\int_{\R^d}   \langle b_m(t,y),   \nabla u_{0,R}^{m,\nu}(t,y) \rangle \varphi_R(t,y)  dy
	\nonumber \\
	&=&
	\int_{\R^d}  \nabla \varphi_R(t,y) \cdot  b_m(t,y) u_{0,R}^{m,\nu}(t,y) dy
	+
	\int_{\R^d}   \varphi_R(t,y) \nabla \cdot  b_m(t,y) u_{0,R}^{m,\nu}(t,y) dy
	\nonumber \\
	&=: & B_1+ B_2. 
\end{eqnarray}
By the Besov duality property, see Proposition \ref{prop_dualite}, we write
\begin{equation}\label{ineq_B1_dual}
	|B_1|
	\leq 
	\|b_m\|_{L^\infty(B_{\infty,\infty}^{-\beta})}
	\|\nabla \varphi_R u_{0,R}^{m,\nu}\|_{L^\infty(B_{1,1}^{-\beta})},
\end{equation}
and
\begin{equation}\label{ineq_B2_dual}
	|B_2|
	\leq 
	\|\nabla \cdot b_m\|_{L^\infty(B_{\infty,\infty}^{-1-\beta})}
	\| \varphi_R u_{0,R}^{m,\nu}\|_{L^\infty(B_{1,1}^{1+\beta })},
\end{equation}
with $B_2=0$ for $b$ incompressible.
\\

\textbf{Control of $\|\nabla  \varphi_R u_{0,R}^{m,\nu}\|_{L^\infty(B_{1,1}^{\beta })}$}
\\

Let us prove that $\|\nabla  \varphi_R u_{0,R}^{m,\nu}\|_{L^\infty(B_{1,1}^{\beta })}$ is controlled uniformly in $(m,\nu)$.
By the thermic representation of the Besov norms \eqref{def_norm_besov_inhomo}, we have
\begin{eqnarray*}
	\|\nabla \varphi_R u_{0,R}^{m,\nu}\|_{L^\infty(B_{1,1}^{\beta })}
	&=&  \|\nabla \varphi_R u_{0,R}^{m,\nu}\|_{L^1}+ \|\nabla \varphi_R u_{0,R}^{m,\nu}\|_{L^\infty(\ddot B_{1,1}^{\beta})}
	\nonumber \\
	&=& \|\nabla \varphi_R u_{0,R}^{m,\nu}\|_{L^1} + \int_0^1 \frac{1}{v}v^{1+\frac{\beta }{2}}\int_{\R^d} \Big | \int_{\R^d} \partial_v h_v(z-y) u_{0,R}^{m,\nu}(t,y) \nabla \varphi_R(t,y) dy \Big | dz \, dv.
\end{eqnarray*}
The first contribution in the r.h.s. above is obviously bounded uniformly in $(m,\nu)$ by
\begin{equation*}
	\|\nabla \varphi_R u_{0,R}^{m,\nu}\|_{L^1} \leq \|\nabla \varphi_R\|_{L^1} \|u_{0,R}^{m,\nu}\|_{L^\infty}
	\leq C  \|\nabla \varphi_R\|_{L^1} \big (T   \|f\|_{L^\infty}+\|g\|_{L^\infty} \big ),
\end{equation*}
by uniform estimate \eqref{ineq_Linfty}.

For the second one, we need to deeply use the already known regularity of $u_{0,R}^{m,\nu}$.
By 
triangular inequality, we obtain
\begin{eqnarray*}
	\|\nabla \varphi_R u_{0,R}^{m,\nu}\|_{L^\infty(\ddot B_{1,1}^{\beta})}	&\leq & 
	\int_0^1 \frac{1}{v}v^{1-\frac{\beta}{2}}\int_{\R^d} \Big | \int_{\R^d}  \partial_v h_v(z-y) \cdot \big \{  [u_{0,R}^{m,\nu}(t,y) -u_{0,R}^{m,\nu}(t,z) ]\nabla \varphi_R(t,y)
	\nonumber \\
	&& +u_{0,R}^{m,\nu}(t,z)[ \nabla  \varphi (t,y)-\nabla  \varphi(t,z)]  \big \} dy \Big | dz \, dv
	\nonumber \\
	&\leq & 
	\int_0^1 \frac{1}{v}v^{1-\frac{\beta}{2}}\int_{\R^d} \Big | \int_{\R^d}  \partial_v h_v(z-y) \cdot \big \{  [u_{0,R}^{m,\nu}(t,y) -u_{0,R}^{m,\nu}(t,z) ]\nabla \varphi_R(t,y)
	\nonumber \\
	&& +u_{0,R}^{m,\nu}(t,z) (y-z) \cdot \int_0^1 D^2 \varphi_R (t,z+\mu(y-z)) d \mu  \big \} dy \Big | dz\, dv,
\end{eqnarray*}
by Taylor expansion.
Next, with the exponential absorbing property \eqref{ineq_absorb},
\begin{eqnarray*}
	&&\|\nabla \varphi u_{0,R}^{m,\nu}\|_{L^\infty(\ddot B_{1,1}^{\beta})}	
	\nonumber \\
	&\leq & 
	C 	\|u_{0,R}^{m,\nu}\|_{L^\infty(C^\gamma)}
	\int_0^1 \frac{1}{v}v^{-\frac{\beta}{2}}\int_{\R^d}  \int_{\R^d}  h_{C^{-1}v}(z-y) | y-z|^\gamma |\nabla \varphi_R(t,y)| dy \, dz \, dv
	\nonumber \\
	&& +
	C	\|u_{0,R}^{m,\nu}\|_{L^\infty}
	\int_0^1 \frac{1}{v}v^{-\frac{\beta}{2}}\int_{\R^d}  \int_{\R^d}  h_{C^{-1}v}(z-y) | y-z|   \int_0^1\big | D^2 \varphi_R (t,z+\mu(y-z)) \big | d \mu   dy \, dz \, dv
	\nonumber \\
	&\leq & 
	C 	\|u_{0,R}^{m,\nu}\|_{L^\infty(C^\gamma)} \|\nabla \varphi_R \|_{L^1}
	\int_0^1 \frac{1}{v}v^{\frac{\gamma-\beta}{2}}  dv
	+
	C \|u_{0,R}^{m,\nu}\|_{L^\infty} \|D^2 \varphi_R \|_{L^1}
	\int_0^1 \frac{1}{v}v^{\frac{1-\beta}{2}} dv
	,
\end{eqnarray*}
which is finite if $\beta <\gamma$.
\\

\textbf{Control of $\| \varphi_R u_{0,R}^{m,\nu}\|_{L^\infty(B_{1,1}^{1+\beta})}$}
\\

The analysis is similar as before, replacing $\beta$ by $1+\beta$ and $\nabla \varphi_R$ by $\varphi_R$:
\begin{equation*}
	\| \varphi_R u_{0,R}^{m,\nu}\|_{L^\infty(B_{1,1}^{1+\beta})}
	= \| \varphi_R u_{0,R}^{m,\nu}\|_{L^1}+ \| \varphi_R u_{0,R}^{m,\nu}\|_{L^\infty(\ddot B_{1,1}^{1+\beta})}.
\end{equation*}
We readily get
\begin{equation*}
	\| \varphi_R u_{0,R}^{m,\nu}\|_{L^1}
	\leq \| \varphi_R \|_{L^1} \|  u_{0,R}^{m,\nu}\|_{L^\infty}
	\leq   \| \varphi_R\|_{L^1} \big (T   \|f\|_{L^\infty}+\|g\|_{L^\infty} \big ),
\end{equation*}
and
\begin{equation*}
	\| \varphi_R u_{0,R}^{m,\nu}\|_{L^\infty(\ddot B_{1,1}^{1+\beta})}	
	\leq  
	C 	\|u_{0,R}^{m,\nu}\|_{L^\infty(C^\gamma)} \| \varphi_R \|_{L^1}
	\int_0^1 \frac{1}{v}v^{\frac{\gamma-\beta-1}{2}}dv
	+
	C \|u_{0,R}^{m,\nu}\|_{L^\infty} \|\nabla  \varphi_R \|_{L^1}
	\int_0^1 \frac{1}{v}v^{\frac{-\beta}{2}} dv
	,
\end{equation*}
this is finite if $1+\beta <\gamma$ $\Longleftrightarrow$ $\beta <-1+\gamma<0$.
Let us carefully notice that if there is the incompressible assumption $\nabla \cdot b=0$, then $B_2=0$ and this former constraint disappears.
Thus the different cases considered in Theorem \ref{THEO_SCHAU_non_bounded}.
\\

The Bolzano-Weierstrass theorem then yields the result.

\subsubsection{Weak solution}
\label{sec_weak}

The difficulty for the usual weak solution, here, is to prove that, up to a subsequence extraction,
\begin{eqnarray}\label{exist_weak_sol}
	\lim_{m \to \infty}\int_{\R^d}\int_{0}^T \langle  b_m(t,y) ,  \nabla u_{0,R}^{m,\nu}(t,y) \rangle \varphi_R(t,y)  dy \ dt
	&=& \int_{0}^T\int_{\R^d}  \langle b(t,y),  \nabla u_{0,R}(t,y) \rangle \varphi_R(t,y)  dy \ dt
	\nonumber \\
	&=& \int_{0}^T\int_{\R^d}  \langle b(t,y),  \nabla u(t,y) \rangle \varphi_R(t,y)  dy \ dt
	, \nonumber \\
\end{eqnarray}
where $b \in L^\infty([0,T];\tilde B_{\infty,\infty}^{-\beta }(\R^d,\R^d))$ is the drift of the initial Cauchy problem \eqref{transport_equation} and coincides with the limit of $b_m$ in $L^\infty([0,T];B_{\infty,\infty}^{-\beta -\varepsilon}(\R^d,\R^d))$, for any $0<\varepsilon$ when $m\to \infty$; also $u(s,\cdot) \in C_b^{\gamma}(K,\R)$, $K= B_d(0,R)$, is the limit of $u^{m,\nu}(s,\cdot)$, up to a subsequence selection possibly depending on the current time $s$, in $C_b^{\gamma-\tilde \varepsilon}(\R^d,\R)$ for any $0<\tilde \varepsilon<\gamma$, see Section \ref{sec_mild_conv}.

Let us recall that $\tilde B_{\infty,\infty}^{-\beta }$ is the closure space of $C^\infty_b$ in $B_{\infty,\infty}^{-\beta }$, from Appendix Section \ref{sec_free_mol}, we still can take the regular sequence $(b_m)_{m \geq 1}$ defined in \eqref{def_b_epsilon} to approximate $b$ in $L^\infty([0,T];B_{\infty,\infty}^{-\beta -\varepsilon}(\R^d,\R^d))$, for any $0<\varepsilon$; whereas the considered solution $u(s,\cdot)$ may depend on the choice of mollification. In other words, we have:
\begin{equation}\label{lim_b_m_b}
	\lim_{m \to \infty}\|b_m -b\|_{L^\infty(B_{\infty,\infty}^{-\beta -\varepsilon})}
	=0. 
\end{equation}
For all $m >0$, $t \in [0,T]$, we write by integration by parts that
\begin{eqnarray}\label{ineq_A1_A2}
	&&\Big | \int_{\R^d} \langle b_m(t,y),   \nabla u_{0,R}^{m,\nu}(t,y) \rangle \varphi_R(t,y)  dy
	- \int_{\R^d} \langle  b(t,y) , \nabla u_{0,R}(t,y) \rangle \varphi_R(t,y)  dy \Big|
	\nonumber \\
	&\leq & 
	\Big |\int_{\R^d}  [b_m-b](t,y)   u^{m,\nu}(t,y)\cdot \nabla \varphi_R(t,y)  dy
	\Big |
	+
	\Big | \int_{\R^d}   b(t,y)  [u-u^{m,\nu}](t,y) \cdot \nabla  \varphi_R(t,y)  dy \Big |
	\nonumber \\
	&&+ 
	\Big |\int_{\R^d}  \nabla \cdot [b_m-b](t,y)   u^{m,\nu}(t,y) \varphi_R(t,y)  dy
	\Big |
	+
	\Big | \int_{\R^d}   \nabla \cdot b(t,y)  [u-u^{m,\nu}](t,y) \  \varphi_R(t,y)  dy \Big |
	\nonumber \\
	&=:& \tilde B_1+\tilde B_2+ \tilde B_3 + \tilde B_4.
\end{eqnarray}
To deal with the first contribution, we aim to use \eqref{lim_b_m_b}. We have by the Besov duality result of Proposition \ref{prop_dualite}:
\begin{equation}\label{ineq_A1_dual}
	\tilde B_1 \leq  
	\|b_m -b\|_{L^\infty(B_{\infty,\infty}^{-\beta -\varepsilon})}
	\|u^{m,\nu} \nabla \varphi_R\|_{L^\infty(B_{1,1}^{\beta +\varepsilon})},
\end{equation}
it then remains to control $\|u^{m,\nu} \nabla \varphi_R\|_{L^\infty(B_{1,1}^{\beta+\varepsilon})}$.
Similar computations as in Section \ref{sec_existence_mild_weak} yields that $\|u^{m,\nu} \nabla \varphi_R\|_{L^\infty(B_{1,1}^{\beta+\varepsilon})}$ is finite if $\beta < \gamma-\varepsilon$.

Also for the third $\tilde B_3$ and the fourth term $\tilde B_4$, which are null if $\nabla \cdot b=0$.

Hence, from \eqref{lim_b_m_b}, we obtain
\begin{equation}\label{lim_A1}
	\tilde B_1	\xrightarrow[(m,\nu)\to (+ \infty,0)]{\eqref{CONDINU}}0.
\end{equation}

Now, let us handle with the second term in \eqref{ineq_A1_A2}.
We aim here to use the convergence of $u^{m,\nu}(s,\cdot)  $ towards $u(s,\cdot)$ in the ball $B_d(0,R)$. 

Again by the Besov duality result of Proposition \ref{prop_dualite}, we have:
\begin{eqnarray}\label{ineq_A2}
	\tilde B_2&=&
	\Big | \int_{\R^d}  b(t,y)   [u-u^{m,\nu}](t,y) \cdot \nabla \varphi_R(t,y)  dy \Big |
	\nonumber \\
	&\leq&
	\|b\|_{L^\infty(B_{\infty,\infty}^{-\beta })}
	\|(u^{m,\nu}-u) \nabla \varphi_R\|_{L^\infty(B_{1,1}^{\beta })}.
\end{eqnarray}
We have, from Section \ref{sec_conv}, 

Arzel\`a-Ascoli theorem,  we have
\begin{equation*}
\|\nabla \varphi_R (u^{m,\nu}-u)\|_{L^\infty} \leq \|\nabla \varphi_R\|_{L^\infty} \|(u^{m,\nu}-u)\|_{L^\infty} 
\xrightarrow[(m,\nu)\to (+ \infty,0)]{\eqref{CONDINU}}
0.
\end{equation*}
For the homogenous part of the Besov norm, we also mimic the analysis in the previous section replacing $u^{m,\nu}$ by $(u^{m,\nu}-u)$, for any $\varepsilon \in (0,\gamma)$:
\begin{eqnarray*}
&&
\|(u^{m,\nu}-u) \nabla \varphi_R\|_{L^\infty(\ddot B_{1,1}^{\beta})}
\nonumber \\
&\leq & 
C 	\|(u^{m,\nu}-u)\|_{L^\infty(C^{\gamma-\varepsilon})}
\int_0^1 \frac{1}{v}v^{-\frac{\beta}{2}}\int_{\R^d}  \int_{\R^d}  h_{C^{-1}v}(z-y) | y-z|^{\gamma-\varepsilon} |\nabla \varphi_R(t,y)| dy \, dz \, dv
\nonumber \\
&& +
C	\|(u^{m,\nu}-u)\|_{L^\infty}
\int_0^1 \frac{1}{v}v^{-\frac{\beta}{2}}\int_{\R^d}  \int_{\R^d}  h_{C^{-1}v}(z-y) | y-z|   \int_0^1 \big | D^2 \varphi_R (t,z+\mu(y-z))\big |  d \mu  dy \, dz \, dv
\nonumber \\
&\leq & 
C 	\|(u^{m,\nu}-u)\|_{L^\infty(C^{\gamma-\varepsilon})} \|\nabla \varphi_R \|_{L^1}
\int_0^1 \frac{1}{v}v^{\frac{\gamma-\varepsilon-\beta}{2}} dv
+
C \|(u^{m,\nu}-u)\|_{L^\infty} \|D^2 \varphi_R \|_{L^1}
\int_0^1 \frac{1}{v}v^{\frac{1-\beta}{2}} dv
,
\end{eqnarray*}
which is finite as soon as $\beta < \gamma-\varepsilon<0$, also 
by Arzel\`a-Ascoli theorem,  we have the converging result $ \|(u^{m,\nu}-u)\|_{L^\infty(C^{\gamma-\varepsilon})} 
\xrightarrow{(m,\nu)\to (+ \infty,0)}
0$.


Therefore, we even get
\begin{equation}\label{lim_A2}
\tilde B_2	\xrightarrow{(m,\nu)\to (+ \infty,0)}
0.
\end{equation}
The last contribution $\tilde B_4$, null if $\nabla \cdot b=0$, is similar replacing $\nabla \varphi_R$ by $\varphi_R$ and $\beta $ by $\beta+1$.
Namely, we have 
\begin{eqnarray}\label{ineq_A4}
\tilde B_4
&\leq&
\|\nabla b\|_{L^\infty(B_{\infty,\infty}^{-1-\beta})}
\|(u^{m,\nu}-u)  \varphi_R\|_{L^\infty(B_{1,1}^{1+\beta})}
\nonumber \\
&=&\|\nabla b\|_{L^\infty(B_{\infty,\infty}^{-1-\beta})}(
\|(u^{m,\nu}-u)  \varphi_R\|_{L^1}+
\|(u^{m,\nu}-u)  \varphi_R\|_{L^\infty(\ddot B_{1,1}^{1+\beta})})
,
\end{eqnarray}
with
\begin{equation*}
\|(u^{m,\nu}-u)  \varphi_R\|_{L^1}\leq  \|u^{m,\nu}-u\|_{L^\infty}\|  \varphi_R\|_{L^1},
\end{equation*}
and
\begin{eqnarray*}
&&\|(u^{m,\nu}-u)  \varphi_R\|_{L^\infty(\ddot B_{1,1}^{1+\beta})}
\nonumber \\
&=&\int_0^1 \frac{1}{v}v^{1-\frac{1+\beta}{2}}\int_{\R^d} \Big | \int_{\R^d} \partial_v h_v(z-y) (u^{m,\nu}-u)(t,y)  \varphi_R(t,y) dy \Big | dz \, dv
\nonumber \\
&\leq & 
C 	\|u^{m,\nu}-u\|_{L^\infty(C^{\gamma-\varepsilon})} \| \varphi_R \|_{L^1}
\int_0^1 \frac{1}{v}v^{\frac{\gamma-\varepsilon-1-\beta}{2}} dv
+
C \|u^{m,\nu}-u\|_{L^\infty} \|D^2 \varphi_R \|_{L^1}
\int_0^1 \frac{1}{v}v^{\frac{-\beta}{2}}dv
,
\end{eqnarray*}
which is finite if $\beta <-1+\gamma-\varepsilon$ and by \eqref{convergence_um_nu}, 
we deduce
\begin{equation}\label{lim_A4}
\tilde B_4	\xrightarrow{
(m,\nu)\to (+ \infty,0)}
0.
\end{equation}

Hence, from \eqref{ineq_A1_A2}, \eqref{lim_A1}, \eqref{lim_A2} and \eqref{lim_A4} we deduce \eqref{exist_weak_sol}.

\subsection{Uniqueness}
\label{sec_uniq}


Let us insist that uniqueness of \textit{vanishing viscous} solution does not mean uniqueness of usual solution.
Indeed, this question arises for the uniqueness of the limits of any sub-sequence of $(u^{m,\nu})_{m,\nu \geq 0}$ and for the non-dependency of the limit on the regularisation procedure; also the smooth selection principle established here does not depend on the choice of the vanishing sequence $(\nu)$.


Let us suppose that there are two \textit{vanishing viscous} solutions $u$ and $\bar u$ of \eqref{transport_equation} satisfying estimates \eqref{ineq_THEO_SCHAU}.
We then consider the associated mollified version $(u^{m,\nu})_{m \geq 0}$ and $(\bar u^ {m,\bar \nu})_{m \geq 0}$ solutions, for any $x \in \R^d$, to
\begin{equation}
	\label{KOLMO_moll_u}
	\begin{cases}
		\partial_t u^{m,\nu}(t,x)+ \langle b_m(t,x) ,  \nabla u^{m,\nu}(t,x) \rangle - \bar \nu \Delta u^{m,\nu}(t,x)=f_m(t,x) ,\ t\in [0,T),\\
		u^{m,\nu}(0,x)=g_m(x),
	\end{cases}
\end{equation}
where $b_m $ is a mollified version of $b$ as in \eqref{def_b_epsilon} by a convolution with the Gaussian mollifier  $\rho_m$,
and 
\begin{equation}
	\label{KOLMO_moll_bar_u}
	\begin{cases}
		\partial_t \bar u^{m,\bar \nu}(t,x)+ \langle \bar b_m(t,x), \nabla \bar u^{m,\bar \nu}(t,x) \rangle- \bar \nu \Delta \bar u^{m,\bar \nu}(t,x)=f_m(t,x) ,\ t\in [0,T),\\
		\bar  u_n(0,x)=g_m(x),
	\end{cases}
\end{equation}
where $\bar b_m$ is a mollified  version of $b$ which is potentially defined differently as in \eqref{def_b_epsilon}, and such that 
\begin{equation}\label{cond_lim_b_b_n_0}
	\forall 0 < \varepsilon < 1, \ 
	\lim_{m \to + \infty} \|\bar b_m-b\|_{L^\infty([0,T];C^{\tilde \gamma-\varepsilon}(\R^d,\R^d))}=0.
\end{equation}
From the linearity of the equations, we then derive that $U_{m}:=u^{m,\nu}-\bar u^{m,\bar \nu}$ solves the following Cauchy problem for any $(t,x)\in [0,T) \times \R^d$:
\begin{equation}
	\label{KOLMO_moll_u1_u2}
	\begin{cases}
		\partial_t U_{m}(t,x)+ \langle \bar b_m ,  \nabla U_{m} \rangle(t,x) -\bar \nu \Delta U_{m}(t,x)
		=  (\nu-\bar \nu ) \Delta u^{m,\nu} (t,x)
		- \langle [b_m-\bar b_m],  \nabla  u^{m,\nu} \rangle(t,x) ,\\
		U_{m}(0,x)=0.
	\end{cases}
\end{equation}
By uniform control \eqref{ineq_Linfty}, we directly derive that
\begin{eqnarray}\label{ineq_Umn_1}
	\|U_{m}(t,\cdot)\|_{L^\infty}
	&\leq &
	\int_0^t \|\langle [b_m-\bar b_m], \nabla u^{m,\nu}\rangle (s,\cdot)\|_{L^\infty}ds
	+ 
	\int_0^t (\nu-\bar \nu ) \| \Delta u^{m,\nu} (s,\cdot)\|_{L^\infty} ds
	\nonumber \\
	&\leq & 
	T  \|b_m-\bar b_m\|_{L^\infty} \| \nabla  u^{m,\nu}\|_{L^\infty}
	+ 
	T |\nu-\bar \nu | \| \Delta u^{m,\nu} \|_{L^\infty} .
\end{eqnarray}
It is clear that if $b$ is $\tilde \gamma$-H\"older continuous then 
$\|b_m-\bar b_m\|_{L^\infty}  \leq C m^{ -\tilde \gamma}$. To take advantage of the convergence of $\|b_m-\bar b_m\|_{L^\infty}$ towards $0$, we need to use other \text{a priori} controls.
\begin{lemma}\label{lemma_apriori_bis}
	If $u^{m,\nu}$ is solution of \eqref{KOLMO_moll_u}, then, for any $(t,x) \in [0,T] \times \R^d$, we get the gradient estimate
	\begin{equation*}
		|\nabla u^{m,\nu}(t,x) |
		\leq
		\int_0^t \|\nabla  f_m(s,\cdot)\|_{L^\infty}ds +\|\nabla  g_m\|_{L^\infty}
		+ C
		m ^{1-\tilde \gamma}	\|b_m\|_{L^\infty(C^{\tilde \gamma})}
	\Big ( t \|  f_m\|_{L^\infty}+\| g_m\|_{L^\infty} \Big ) 
	t^{\frac{1}2}
,
\end{equation*}
and  he Hessian estimate 
\begin{eqnarray*}
|\nabla^2 u^{m,\nu}(t,x) |
&\leq&
\Big ( C \nu ^{\frac{\gamma}{2}-1} t^{\frac \gamma 2}  \| f_m\|_{L^\infty(C^\gamma)}+\|\nabla ^2g_m\|_{L^\infty} \Big ) 
\nonumber \\
&& +C
m^{1-\tilde \gamma}\|b_m\|_{L^\infty(C^{\tilde \gamma})}
\nu ^{-\frac{1}{2}} 
t^{\frac{1}{2}} 
\big ( 
t \|\nabla  f_m\|_{L^\infty}+\|\nabla  g_m\|_{L^\infty}
\big ) 
\nonumber \\
&& + C m^{2(1-\tilde \gamma)}
\|b_m\|_{L^\infty(C^{\tilde \gamma})}^2
\big ( t \|  f_m\|_{L^\infty}+\| g_m\|_{L^\infty} \big ) 
\nu ^{- \frac 12} 
t ^{ \frac 12}.
\end{eqnarray*}
\end{lemma}
The proof is deferred in Section \ref{sec_proof_apriori_bis}.\\

Hence, 
\begin{eqnarray*}
&&	\|U_{m}(t,\cdot)\|_{L^\infty}
\nonumber \\
&\leq & 
T  m^{ -\tilde \gamma} \|b\|_{L^\infty(C^{\tilde \gamma})}
\bigg (
t \|\nabla  f_m\|_{L^\infty}+\|\nabla  g_m\|_{L^\infty}
+ C
\|b_m\|_{L^\infty(C^{\tilde \gamma})}
\Big ( t \|  f_m\|_{L^\infty}+\| g_m\|_{L^\infty} \Big ) 
\nu ^{\frac{\tilde \gamma}{2}-1} 
t^{\frac{\tilde \gamma}{2}}  \bigg ) 
\nonumber \\
&& +
T |\nu-\bar \nu | \Bigg ( 
\Big ( C \nu ^{\frac{\gamma}{2}-1} t^{\frac \gamma 2}  \| f_m\|_{L^\infty(C^\gamma)}+\|\nabla ^2g_m\|_{L^\infty} \Big ) 
\nonumber \\
&& + C
\|b_m\|_{L^\infty(C^{\tilde \gamma})}
\nu ^{\frac{\tilde \gamma}{2}-1} 
t^{\frac{\tilde \gamma}{2}} 
\big ( 
t \|\nabla  f_m\|_{L^\infty}+\|\nabla  g_m\|_{L^\infty}
\big ) 
+ C
\|b_m\|_{L^\infty(C^{\tilde \gamma})}^2
\big ( t \|  f_m\|_{L^\infty}+\| g_m\|_{L^\infty} \big ) 
\nu ^{\tilde \gamma-2} 
t^{\tilde \gamma}
\Bigg )
.
\end{eqnarray*}
From this estimate, to prove uniqueness, we need to consider 
\begin{equation}\label{condi_uniq_proof_transport}
(m^{-\tilde \gamma }+ \nu^{\frac{\tilde \gamma }{2}})( \|\nabla  f_m\|_{L^\infty}+\|\nabla  g_m\|_{L^\infty})+ \textcolor{black}{m^{- \tilde \gamma }\nu^{\frac{\tilde \gamma}{2}-1} } + \nu \|\nabla ^2g_m\|_{L^\infty}  +	
\nu^{\tilde \gamma- 1}  \ll 1,
\end{equation}
the first and the third term are indeed negligible if $\tilde \gamma> \max(1-\gamma,\frac 12)$.

We also write for the two last terms,
\begin{equation*}
\nu ^{\frac{\gamma-1}{2}} \ll	m^{-1+2\tilde \gamma}  \text{, and } m^{1-\tilde \gamma} \ll \nu^{-\frac \gamma 2}.
\end{equation*}
Combining the two terms, we get
\begin{equation*}
m^{1-\tilde \gamma} \ll \nu^{-\frac \gamma 2} \ll \big (m^{\frac{2(-1+2\tilde \gamma)}{1-\gamma}} \big )^{\frac \gamma 2} = m^{\frac{\gamma(-1+2\tilde \gamma)}{1-\gamma}} .
\end{equation*}
Hence, we have to suppose that
\begin{equation*}
1-\tilde \gamma < \frac{\gamma(-1+2\tilde \gamma)}{1-\gamma}.
\end{equation*}
This is equivalent to
\begin{equation*}
1-\tilde \gamma - \gamma + \gamma \tilde \gamma < \gamma(-1+2\tilde \gamma),
\end{equation*}
and
\begin{equation*}
1 < \tilde \gamma\gamma +\tilde \gamma= \tilde \gamma(1+\gamma),
\end{equation*}
thus the condition $\tilde \gamma > \frac{1}{1+ \gamma} > \max ( 1-\gamma, \frac{1}{2})$ of Theorem \ref{THEO_SCHAU_non_bounded_optimal}.
\\

We can have the equality case (useful for the Burgers' equation):
$\tilde \gamma= \gamma$, if $P(\gamma)= \gamma^2+ \gamma-1>0$ which means that $\gamma > \frac{-1+ \sqrt{5}}2 \asymp 0.6180.$

\begin{remark}\label{Rem_uniq_negativ}
We fail to get uniqueness for negative Besov regularity of $b$.
Indeed, from the analysis performed in Section \ref{sec_decoup_temps}, and because 
we do not differentiate $U_{m} $ in order to take unsuccessfully take advantage of Gr\"onwall's lemma,
we  need to consider the term by Besov duality
\begin{eqnarray*}
\|\nabla \tilde p(s,t,x,\cdot)  U_{m} (s,\cdot)\|_{\ddot B_{1,1}^{-\tilde \gamma}}
&=& 
\int_{0}^1 v^{-1} v^{1+\frac{\tilde \gamma}{2}} 
\int_{\R^d} \Big | \int_{\R^d} \partial_v h_v(z-y) \nabla \tilde p(s,t,x,y)  U_{m} (s,y) dy \Big | dz \, dv
\nonumber \\
&\leq &
C \| U_{m} (s,\cdot)\|_{L^\infty}  \int_{0}^1 v^{-1} v^{-\frac{\tilde \gamma}{2}} [\nu (t-s)]^{-\frac{1}{2}} dv,
\end{eqnarray*}
which is finite if only $\tilde \gamma<-1$, that means that $b$ has to be Lipschitz continuous.
\end{remark}

\subsection{Control of $\|\partial_t u(t,\cdot)\|_{B_{\infty,\infty}^{-1+\gamma}}$}
\label{sec_control_partial_t}

If $b \in L^\infty([0,T]; C^{\tilde \gamma}(\R^d,\R^d))$, $0<1-\gamma<\tilde \gamma$, we derive an upper-bound of $\|\partial_t u(t,\cdot)\|_{B_{\infty,\infty}^{-1+\gamma}}$
by the equation \eqref{transport_equation} and by para-product result.
But first of all, let us precise why we have, point-wisely, with the viscous condition \eqref{CONDINU}, 
\begin{equation}\label{convergence_nu_Delta_u}
\lim_{(m,\nu)\to (+ \infty,0)}\nu \Delta u ^{m,\nu}(t,\cdot) =0.
\end{equation}
Recalling that, from Lemma \ref{lemma_apriori}, 
\begin{equation}\label{ineq_D2_u_m_nu_proof}
\|\nabla^2 u^{m,\nu}(t,\cdot ) \|_{L^\infty}
\leq 
\Big ( 	m^{2-\gamma}\big ( t \| f\|_{L^\infty(C^\gamma)}+[g]_\gamma\big )  + C t m^{2-\tilde \gamma }	\|b\|_{L^\infty( B_{\infty,\infty}^{\tilde \gamma})}	O_m(t)\Big )
\exp(t m^{1-\tilde \gamma} \|b\|_{L^\infty( B_{\infty,\infty}^{\tilde \gamma})}).
\end{equation}
Hence, for 
\begin{equation*}
\nu \ll \Big ( 	m^{2-\gamma}\big ( T \| f\|_{L^\infty}+[g]_\gamma\big )  + C T m^{2-\tilde \gamma }	\|b\|_{L^\infty( B_{\infty,\infty}^{\tilde \gamma})}	O_m(t)\Big )^{-1}
\exp(-T m^{1-\tilde \gamma} \|b\|_{L^\infty( B_{\infty,\infty}^{\tilde \gamma})}),
\end{equation*}
we deduce \eqref{convergence_nu_Delta_u}.

We are able to take the limit of equation \eqref{parabolic_moll_def}, up to sub-sequence selection as explained in Section \ref{sec_conv},
for any $t \in (0,T]$,
\begin{equation}\label{lim_b_nabla_u}
\lim_{(m,\nu) \to (+ \infty,0)} \partial_t u^{m,\nu}(t,\cdot)= 
\lim_{(m,\nu) \to (+ \infty,0)} \langle b_m(t,\cdot), \nabla u^{m,\nu}(t,\cdot) \rangle+ f(t,\cdot).
\end{equation}
But from para-product \eqref{condi_Bony}, we know that $\langle b_m(t,\cdot), \nabla u^{m,\nu}(t,\cdot) \rangle \in B_{\infty,\infty}^{-1+\gamma}
(\R^d,\R)$, $\nabla u^{m,\nu}(t,\cdot)$ being in $B_{\infty,\infty}^{-1+\gamma}$,
the result then follows.

\subsection{Control of the \textit{limit H\"older modulus}}

In this section, let us suppose that $b \in L^\infty([0,T],B_{\infty,\infty}^{-\beta}(\R^d,\R^d))$.

\begin{prop}[Partial result for the transport equation]
	\label{THEO_SCHAU_non_bounded}
	For  $\gamma\in (0,1)$, $\beta \in \R ^*$  be given.
	For all $ b \in L^\infty([0,T], \tilde B_{\infty,\infty}^{-\beta} (\R^d,\R^d))$, 
	$f\in L^\infty([0,T];  C_b^{\gamma}(\R^d,\R))$ and $ g \in C^{\gamma}_b(\R^d,\R)$, 
	if the conditions on the vanishing viscosity $ 0 <\nu <T^{-1}$, for a given constant $C>0$ depending only on $(\gamma,d)$, 
	\begin{equation}\label{CONDINU}
		1 \ll \nu^{-1}T^{-1} m^{-\frac{2-\gamma}{1-\gamma}} (t  \| f\|_{L^\infty(C^\gamma)}+ [g]_{\gamma} ) ^{-\frac{2}{1-\gamma}}
		\exp \Big (-c\frac{m^{1+\beta} T \|b\|_{L^\infty( B_{\infty,\infty}^{-\beta})}} {1-\gamma}\Big ),
	\end{equation}
	are satisfied,
	then $u^{m\nu}$ satisfies
	\begin{eqnarray}\label{Schauder_ineq}
		[u^{m,\nu}]_{\gamma,\nu^{1/2}T^{1/2}}^{\eqref{CONDINU}}
		&\leq& 
		T \sup_{t \in [0,T]}[f(t,\cdot)]_{\gamma,\nu^{1/2}T^{1/2}}^{\eqref{CONDINU}}+ [g]_{\gamma,\nu^{1/2}T^{1/2}}^{\eqref{CONDINU}},
		\nonumber \\
		\|  u \|_{L^\infty} &\leq&  T\| f \|_{L^\infty}+ \|  g \|_{L^\infty}  ,
	\end{eqnarray}
	
\end{prop}

\begin{remark}\label{rem_vitesse}
	The exponential criterion in \eqref{CONDINU} relies on an \textit{a priori} control by $\|u^{m,\nu}\|_{L^\infty}$, see Section \ref{sec_control_diag_u_C1} for more details.
	This condition prevents us to hope any balance between $m$ and $\nu$ required to get usual uniqueness. 
	Indeed, when we expand the computations, we see only polynomial dependency on $(m,\nu)$ in the upper-bounds. But the contribution on $\nu$ goes in the wrong way, and cannot be overwhelmed by polynomial converging terms in $m$, because at the best $m \sim |\ln(\nu)|$ from \eqref{CONDINU}.

	Even for $b$ lying in a H\"older space, namely with a positive regularity, we cannot avoid 
	an exponential criterion like in \eqref{CONDINU} by the current analysis, see again Section \ref{sec_control_diag_u_C1}.
\end{remark}

	\subsubsection{Some a priori controls}
\label{sec_apriori}

It is well-known that the unique solution $u^{m,\nu}$ of \eqref{KOLMOLLI} is smooth, see \cite{frie:64}.
Some \textit{a priori} controls, potentially blowing up in $m$ and $\nu$, are required in our analysis.
\begin{lemma}\label{lemma_apriori}
	For $u^{m,\nu}$ strong solution of \eqref{KOLMOLLI}, we have
	\begin{equation}\label{ineq_O}
		\|\nabla u^{m,\nu}(t,\cdot)\|_{L^\infty} \leq \Big ( T \| f_m\|_{L^\infty(C^1)}+\|g_m\|_{C^1} \Big ) \exp \Big ( C m^{1+\beta} t \|b\|_{L^\infty( B_{\infty,\infty}^{-\beta})} \Big )=: O_m(t),
	\end{equation}
	and 
	\begin{equation}\label{ineq_D2_u_m_nu_FACILE}
		\|\nabla^2 u^{m,\nu}(t,\cdot ) \|_{L^\infty}
		\leq 
		C	 	m^{2-\gamma}\big ( t \| f\|_{L^\infty(C^\gamma)}+[g]_\gamma\big )  
		\exp(C  m^{1+\beta}t \|b\|_{L^\infty( B_{\infty,\infty}^{-\beta})})
		=: O^{(2)}_m(t),
	\end{equation}
	also
	\begin{equation}\label{ineq_D3_u_m_nu}
		\|\nabla^3 u^{m,\nu}(t,\cdot ) \|_{L^\infty}
		\leq
		C	 	m^{3-\gamma}\big ( t \| f\|_{L^\infty(C^\gamma)}+[g]_\gamma\big ) 
		\exp(C  m^{1+\beta}t \|b\|_{L^\infty( B_{\infty,\infty}^{-\beta})}).
	\end{equation}
	
\end{lemma}
The proof is postponed in Section \ref{sec_grad}.
	
	\subsection{First part of the control of the limit modulus of continuity: the \textit{cut locus} trick}
	\label{sec_holder}
	
		For any $(t,x,x') \in [0,T] \times \R^d\times \R^d$, 
		we choose the associated \textit{freezing} points $\big ((\tau,\xi),(\tau',\xi')\big ) \in \big ( [0,T]\times \R^d \big )^2$.
		Like in \cite{chau:hono:meno:18}, we take $\tau=\tau'$.
		The previous Duhamel formula  \eqref{Duhamel_u} yields
	\begin{eqnarray}\label{Holder_Duhamel}
	&&	|u^{m, \nu}(t,x) -u^{m, \nu}(t,x')| 
	\nonumber \\
	&\leq&  |\hat G^{\tau,\xi}   f_m(t,x)-\hat G^{\tau,\xi '}   f_m(t,x')|+ |\hat P^{\tau,\xi}   g_m(t,x)-\hat P^{\tau,\xi '}   g_m(t,x')|
		\nonumber \\
		&&+
		\Big | \int_0^t \int_{\R^d}   \hat {p}^{\tau,\xi} (s,t,x,y) [ b_m(s,\theta_{s,\tau}^m(\xi))- b_m(s,y )] \cdot  \nabla  u^{m, \nu}(s,y) dy \, ds
		\nonumber \\
				&&-
	\int_0^t \int_{\R^d}   \hat {p}^{\tau',\xi'} (s,t,x,y) [ b_m(s,\theta_{s,\tau'}^m(\xi'))- b_m(s,y )] \cdot  \nabla  u^{m, \nu}(s,y) dy \, ds \Big |
		\nonumber \\
		&=:& |\hat G^{\tau,\xi}   f_m(t,x)-\hat G^{\tau,\xi '}   f_m(t,x')|+ |\hat P^{\tau,\xi}   g_m(t,x)-\hat P^{\tau,\xi '}   g_m(t,x')|+  |R^{\tau,\xi,\xi'}(t,x,x') |.
	\end{eqnarray}
	However, our analysis need different choices of \textit{freezing} point which yields extra contributions in the above inequality, the final Duhamel like identity is stated in \eqref{identi_Holder_final} further.
	\\

\begin{remark}
	It seems to be useless to proceed with any integration by parts to transfer the gradient from $u^{m, \nu}$ to $\hat {p}^{\tau,\xi} (s,t,x,y) b^m_{\Delta} [\tau,\xi]$
	and  to upper-bound by $\|u^{m,\nu}\|_{L^\infty(C^\gamma_b)}$, a well-controlled norm, instead of $\|\nabla u^{m,\nu}\|_{L^\infty}$, a blowing up term.
	
	Indeed, in diagonal regime, $\hat {p}^{\tau,\xi} (s,t,x,y)$ yields a term of the type $[\nu (t-s)]^{-\frac{1}{2}}$ which has to smoothen by Lipschitz norms $\|u^{m,\nu}(t,\cdot)\|_{C^1}$  in order to get a contribution of $|x-x'|^\gamma$ and a positive contribution of $\nu$.
	For more details, we provide in Section \ref{sec_control_diag_u_C1} a full computation associated with the remainder $R^{\tau,\xi,\xi'}(t,x,x')$ which necessarily implies to upper-bound by a Lipschitz norm of $u^{m,\nu}$.

This last value is finite but increases exponentially with $m$, see Lemma \ref{lemma_apriori}. 
	This exponential blowing-up, without the \textit{time decomposition} trick introduced in Section \ref{sec_parab}, yields the limit criterion \eqref{CONDINU}  of $(m,\nu)$, and  prevents us to get any balance between $m$ and $\nu$ to obtain uniqueness of solution.
\end{remark}

		\subsubsection{Main terms}
	\label{sec_Green_semi_Holder}
	
	For the main contributions associated with $f$ and $g$,
	\textcolor{black}{we have to suppose that the \textit{diagonal} regime specified further is in force. Precisely, we assume that there are $\alpha_1,\alpha_2>0$ such that 
		\begin{equation}\label{Condi_diagonal_Semi_group}
			|x-x'| \leq \nu^{\alpha_1}		T^{\alpha_2}.
		\end{equation}
 In this case,}
	 we choose the \textit{freezing} parameters to be $\tau=t$ and
	$\xi=\xi'=x$.
	
	\textcolor{black}{Let us carefully point out that condition \eqref{Condi_diagonal_Semi_group} leads to the limit modulus of continuity of the type of \eqref{Def_modulus_continuity} when $\nu \to 0$, instead of the usual H\"older modulus. We precise the reverse condition in Section \ref{sec_comment_Holder_modulus}}
	\\
	
	This constraint is crucial, in order to avoid the sensitivity of the flow on different points $x,x'$. 
	However, the assumption \eqref{Condi_diagonal_Semi_group} yields to consider the limit H\"older continuity defined in \eqref{Def_modulus_continuity}.
	\\
	
	\textbf{Semi-group}
	\\
	
	We readily derive by change of variables:
	\begin{eqnarray}\label{ineq_P_g_Holder}
&&		|\hat P^{\tau,\xi}  g_m(t,x) - \hat P^{\tau,\xi'}  g_m (t,x')| \Big |_{\tau=t,\xi=\xi'=x}
\nonumber \\
&=& 	\Big | \int_{\R^d} [ \hat {p}^{\tau,\xi} (0,t,x,y)-\hat {p}^{\tau, \xi '} (0,t,x',y)]	 g_m(y) dy  \Big | \Bigg |_{\tau=t,\xi=\xi'=x}
		\nonumber \\
		&=& 	\Big | \int_{\R^d} [ \hat {p}^{t,x} (0,t,0,y)[ g_m(x+y)- g_m(x'+y)]	 dy  \Big | 
		\nonumber \\
		&\leq & [g]_{\gamma,\nu^{\alpha_1}(t-s)^{\alpha_2}}|x-x'|^\gamma.
	\end{eqnarray}
	

	\textbf{Green operator}
	\\
	
	We also get by change of variables:
	\begin{eqnarray}\label{ineq_Gf_Holder}
		&&	|\hat G^{\tau,\xi}  f_m(t,x) - \hat G^{\tau,\xi'}  f_ m (t,x')| \Big |_{\tau=t,\xi=\xi'=x}
		\nonumber \\
		&=& 	\Big | \int_0^t  \int_{\R^d} [ \hat {p}^{t,x} (s,t,x,y)-\hat {p}^{t,x} (s,t,x',y)]	 f_m(s,y) dy \, ds  \Big | 
		\nonumber \\
		&=& 	\Big | \int_0^t \int_{\R^d} [ \hat {p}^{t,x} (s,t,0,y)[ f_m(s,x-y)- f_m(s,x'-y)] ds \, dy  \Big | 
		\nonumber \\
		&\leq & 
		\int_0^t [f(s,\cdot)]_{\gamma,\nu^{\alpha_1}(t-s)^{\alpha_2}} ds |x-x'|^\gamma.
	\end{eqnarray}

	\subsubsection{Remainder term}
	\label{sec_remainders}
	
	To analyse the H\"older modulus of the remainder term, the core  of the \textit{a priori} controls, we separate the \textit{diagonal} regime from the \textit{off-diagonal} one, as performed in \cite{chau:hono:meno:18}.
	This strategy is natural in view with the vanishing viscous solution selected by a parabolic approximation.
	
	However, in the vanishing viscosity context, we have to carefully track the dependency on $\nu$ which yields, for our first approach, to consider a unusual criterion of  \textit{diagonal} / \textit{off-diagonal} regime, unlike for the standard parabolic scaling.

	Specifically, for any $x,x' \in \R^d$ and for given parameters $(\alpha_1, \alpha_2) \in \R^2$, to be tailored further, we call \textit{off-diagonal} regime the case $|x-x'|> \nu^{\alpha_1} (t-s)^{\alpha_2} \Leftrightarrow s> t_0$
	with 
	\begin{equation}\label{def_t0}
		t_0:=t- \nu^{-\frac{\alpha_1}{\alpha_2}} |x-x'|^{\frac 1 {\alpha_2}},
	\end{equation}
called the \textit{cut locus} time.

On the contrary the \textit{diagonal} regime holds when $|x-x'|\leq  \nu^{\alpha_1} (t-s)^{\alpha_2} \Leftrightarrow s\leq t_0$.
	This regime can be in force only if $|x-x'|$ is small enough, specifically, only if $|x-x'|\leq \nu ^{\alpha_1}T^{\alpha_2}$, thus the definition of the limit modulus of continuity \eqref{Def_modulus_continuity}.
	
	The point $t_0$ can be regarded as a \textit{cut-locus} point where we ``catch" the shortest path from $u^{m,\nu}(t,x)$ to $u^{m,\nu}(t,x')$ if $t \in [0,t_0]$ and we choose another way if $t_0<t$.

This procedure yields an extra contribution in \eqref{Holder_Duhamel}, this is detailed in Sections \ref{sec_discon_freez}, \ref{sec_extra_discon}. 
	
	We specify in Section \ref{sec_justification_freez} below why it is possible to choose different \textit{freezing} parameters for the remainder term according to the current regime; meanwhile the semi-group and the Green operator dealt in Section \ref{sec_Green_semi_Holder} has somehow to stay in the \textit{diagonal} regime.

	\subsubsection{Diagonal regime}
	\label{sec_diag}
	
	If  $|x-x'|\leq \nu^{\alpha_1} (t-s)^{\alpha_2} \Leftrightarrow s\leq t_0=t- \nu^{-\frac{\alpha_1}{\alpha_2}} |x-x'|^{\frac 1 {\alpha_2}}$, the points $x$ and $x'$ are supposed to be closed from each other, then we pick $\xi=\xi'=x$, also $\tau=t$, then we define the associated space 
	\begin{equation}\label{def_A}
		A(x,x',\nu,t)(s):=\{|x-x'|\leq \nu^{\alpha_1} (t-s)^{\alpha_2}\},
	\end{equation}
	with the indicator function
		\begin{equation*}
	\mathds 1_{A(x,x',\nu,t)}(s):=
	\begin{cases}
	1 \text{ if }
	|x-x'|\leq \nu^{\alpha_1} (t-s)^{\alpha_2},
	\\
	0 \text{ if }
	|x-x'|> \nu^{\alpha_1} (t-s)^{\alpha_2},
	\end{cases}
	\end{equation*}
also the associated remainder term
\begin{eqnarray}\label{def_R1_R2_A}
R_{A}^{\tau,\xi,\xi'}(t,x,x')  
&:=&
\int_0^t \mathds 1_{A(x,x',\nu,t)}(s)
\int_{\R^d}  
\Big [
 \hat {p}^{\tau,\xi} (s,t,x,y) [ b_m(s,\theta_{s,\tau}^m(\xi))- b_m(s,y )] \cdot  \nabla  u^{m, \nu}(s,y)
 \nonumber \\
 &&-
   \hat {p}^{\tau',\xi'} (s,t,x,y) [ b_m(s,\theta_{s,\tau'}^m(\xi'))- b_m(s,y )] \cdot  \nabla  u^{m, \nu}(s,y) \Big ] dy \, ds
,
\end{eqnarray}
which equivalently write,
\begin{eqnarray}\label{parab_def_alternative_R1}
	&&
	R_{A}^{\tau,\xi,\xi'}(t,x,x')  
	\nonumber \\
	&:=&
	\int_0^{t_0}  \int_{\R^d} \Big [  \hat {p}^{\tau,\xi} (s,t,x,y)  [ b_m(s,\theta_{s,\tau}^m(\xi))- b_m(s,y )]  \cdot 
	\nabla u^{m, \nu}(s,y)
	\nonumber \\
	&&-  \hat {p}^{\tau',\xi'} (s,t,x',y)  [ b_m(s,\theta_{s,\tau}^m(\xi'))- b_m(s,y )] 
	\cdot \nabla u^{m, \nu}(s,y) \Big ] dy \, ds 
	\\
	&=:&
	\hat G^{\tau,\xi}_{0,t_0} \Big \{ \big ( b_m(\cdot ,\theta_{\cdot,\tau}^m(\xi))- b_m \big )
	\cdot \nabla  u^{m, \nu}  \Big \}(t,x)
	-
	\hat G^{\tau,\xi'}_{0,t_0}  \Big \{   \big ( b_m(\cdot ,\theta_{\cdot,\tau}^m(\xi'))- b_m \big )
	\cdot \nabla  u^{m, \nu}  \Big \}(t,x').
	\nonumber 
\end{eqnarray}	
A simple change of variable gives
\begin{eqnarray*}
	&&
	|R_{A}^{\tau,\xi,\xi'}(t,x,x')| \big  |_{\tau=t,\xi=\xi'=x} 
	\nonumber \\
	&=& 
	\Big | \int_0^{t_0} 
	\int_{\R^d}   \hat {p}^{\tau,\xi} (s,t,0,y)
	\Big \{ \big ( b_m(s,\theta_{s,\tau}^m(\xi))- b_m(s,x+y )\big ) \cdot    \nabla  u^{m,\nu}(s,x+y)
	\nonumber \\
	&& - \big ( b_m(s,\theta_{s,\tau}^m(\xi'))- b_m(s,x'+y ) \big ) \cdot    \nabla  u^{m,\nu}(s,x'+y) \Big \} dy \, ds \Big | \Bigg |_{\tau=t,\xi=\xi'=x},
\end{eqnarray*}
%
%
\begin{eqnarray*}
	&& 
	|R_{A}^{\tau,\xi,\xi'}(t,x,x')| \big  |_{\tau=t,\xi=\xi'=x} 
	\nonumber \\
	&=& 
	\Big | \int_0^{t_0} 
	\int_{\R^d}   \hat {p}^{\tau,\xi} (s,t,0,y) \Big \{ b_m(s,\theta_{s,\tau}^m(x)) \cdot   \big (\nabla u^{m,\nu}(s,x+y)- \nabla u^{m,\nu}(s,x'+y)\big )
	\nonumber \\
	&&
	+  \big (b_m(s,x'+y )-b_m(s,x+y )\big )\cdot \nabla u^{m,\nu}(s,x'+y)
	\nonumber \\
	&&- b_m(s,x+y )\cdot  \big (\nabla u^{m,\nu}(s,x+y)- \nabla u^{m,\nu}(s,x'+y)\big ) 
	\Big \}dy \, ds \Big |\Bigg |_{\tau=t,\xi=\xi'=x}.
\end{eqnarray*}
Hence, we readily derive that
\begin{equation}\label{parab_ineq_R_1_diag_1}
	|R_{A}^{\tau,\xi,\xi'}(t,x,x')| \big  |_{\tau=t,\xi=\xi'=x}
	\leq 
	C 	|x-x'| 
	\|b_m\|_{L^\infty(C_b^{1})}\|\nabla u^{m, \nu}\|_{L^\infty(C_b^{1})}
	\int_0^{t_0} 
	ds .
\end{equation}
Recalling that we have to use the \textit{a priori} controls of the gradient and the Hessian in Lemma \ref{lemma_apriori}, 
\begin{eqnarray}\label{parab_ineq_u_C1_1}
	\|u^{m,\nu}(t,\cdot)\|_{C^1} &\leq& \Big ( T \| f_m\|_{L^\infty(C^1)}+\|g_m\|_{C^1} \Big ) \exp \Big ( m^{1+\beta} t \|b\|_{L^\infty( B_{\infty,\infty}^{-\beta})} \Big )= O_m(t),
	\nonumber \\
	\|\nabla^2 u^{m,\nu}(t,\cdot ) \|_{L^\infty}
	&\leq& 
	C	 	m^{2-\gamma}\big ( t \| f\|_{L^\infty(C^\gamma)}+[g]_\gamma\big )  
	\exp(C  m^{1+\beta}t \|b\|_{L^\infty( B_{\infty,\infty}^{-\beta})})
	= O^{(2)}_m(t).
\end{eqnarray}
From the definition of the \textit{diagonal} regime, $	|x-x'|  \leq 	|x-x'| ^\gamma \nu^{\alpha_1 (1-\gamma)} (t-s)^{(1-\gamma)\alpha_2} $, we deduce
\begin{equation}\label{ineq_RA_parab}
	|R_{A}^{\tau,\xi,\xi'}(t,x,x')| \big  |_{\tau=t,\xi=\xi'=x}
	\leq 
	C 	  |x-x'|^\gamma
	\|b\|_{L^\infty( B_{\infty,\infty}^{-\beta})}
	m^{1+\beta} (O_m(t)+O_m^{(2)}(t))
	\int_0^{t_0} \nu^{\alpha_1 (1-\gamma)} (t-s)^{(1-\gamma)\alpha_2} ds.
\end{equation}
If $(1-\gamma)\alpha_2 >-1 $ 
$\Leftrightarrow $ $\alpha_2 >
-\frac{1}{1-\gamma} >-1$, the above time integral is finite, and after integration,
\begin{equation}\label{ineq_R_Holder_diag}
|R_{A}^{\tau,\xi,\xi'}(t,x,x')| \big  |_{\tau=t,\xi=\xi'=x}
	\leq 
	C 	  |x-x'|^\gamma 
	\|b\|_{L^\infty(B_{\infty,\infty}^{-\beta})}
	m^{1+\beta} O_m(t) 
	\nu^{\alpha_1(1-\gamma)} t^{1+(1-\gamma)\alpha_2} 
	. 
\end{equation}
To consider vanishing viscous, it is necessary to have $\alpha_1(1-\gamma)>0$ 
$\Leftrightarrow$  $\alpha_1 
> 0$ for $\gamma<1$.
\\


%
To sum up, 
we consider the constraints
\begin{eqnarray}\label{Contrainte_1}
\alpha_1 &>& 0,
\nonumber \\
\alpha_2  &>&-\frac{1}{1-\gamma}.
\end{eqnarray}

\subsubsection{Off-diagonal regime}
\label{sec_off_diag}

In this case, $x$ and $x'$ are supposed to be ``far away" from each other, then we pick $\xi=x$ and $\xi'=x'$.
If $|x-x'|> \nu^{\alpha_1} (t-s)^{\alpha_2} \Leftrightarrow s>t- \nu^{-\frac{\alpha_1}{\alpha_2}} |x-x'|^{\frac{1}{\alpha_2}}$, we recall the corresponding \textit{cut locus} time 
\begin{equation*}
	t_0=t- \nu^{-\frac{\alpha_1}{\alpha_2}} |x-x'|^{\frac{1}{\alpha_2}}.
\end{equation*} 
In this case, we choose as \textit{freezing} parameters $\xi=x$ and $\xi'=x'$.
In the \textit{off-diagonal} regime, the associated space is
\begin{equation}\label{def_AC}
	A^c(x,x',\nu,t)(s):=\{|x-x'|> \nu^{\alpha_1} (t-s)^{\alpha_2}\},
\end{equation}
the indicator function
\begin{equation*}
	\mathds 1_{A^c(x,x',\nu,t)}(s)=1-\mathds 1_{A(x,x',\nu,t)}(s)=
	\begin{cases}
		1 \text{ if }
		|x-x'|> \nu^{\alpha_1} (t-s)^{\alpha_2},
		\\
		0 \text{ if }
		|x-x'|\leq  \nu^{\alpha_1} (t-s)^{\alpha_2},
	\end{cases}
\end{equation*}
and the associated remainder terms are
\begin{eqnarray*}
	&&R_{A^c}^{\tau,\xi,\xi'}(t,x,x')  
	\nonumber \\
	&:=& 
	R^{\tau,\xi,\xi'}(t,x,x') -R_{A}^{\tau,\xi,\xi'}(t,x,x') 
	\nonumber \\
	&=&
	\int_{t_0}^{t}  \int_{\R^d}  \Big [   \hat {p}^{\tau,\xi} (s,t,x,y)  \big ( b_m(s,\theta_{s,\tau}^m(\xi))- b_m(s,y )\big ) 
	\cdot \nabla u^{m, \nu}(s,y)
	\nonumber \\
	&&-
	\hat {p}^{\tau',\xi'} (s,t,x',y)  \big ( b_m(s,\theta_{s,\tau}^m(\xi'))- b_m(s,y )\big ) 
	\cdot \nabla u^{m, \nu}(s,y) \Big] dy \, ds 
	\\
	&=:&
	\hat G^{\tau,\xi}_{t_0,t} \Big \{ \big ( b_m(\cdot ,\theta_{\cdot,\tau}^m(\xi))- b_m \big )
	\cdot \nabla u^{m, \nu} \Big \}(t,x)
	-
	\hat G^{\tau,\xi'}_{t_0,t}  \Big \{   \big ( b_m(\cdot ,\theta_{\cdot,\tau}^m(\xi'))- b_m \big )
	\cdot \nabla u^{m, \nu} \Big \}(t,x'). \nonumber
\end{eqnarray*}
The analysis is direct by triangular inequality
\begin{equation*}
	|R_{A^c}^{\tau,\xi,\xi'}(t,x,x') | \big |_{\tau=t,\xi=x,\xi'=x'} 
	\leq 
	2 C  \|b_ m\|_{L^\infty(C^{1})} \|u^{m,\nu}\|_{L^\infty(C^{ 1})}  
	\sup_{x \in \R^d}	\int_{t_0}^t \int_{\R^d} \bar {p}^{t,x} (s,t,x,y) |y-\theta_{s,t}^m(x)| dy \, ds ,
\end{equation*}
the absorbing property of the exponential \eqref{ineq_absorb} gives
\begin{eqnarray}\label{ineq_RAc_parab}
	|R_{A^c}^{\tau,\xi,\xi'}(t,x,x') | \big |_{\tau=t,\xi=x,\xi'=x'} 
	&\leq & 
	C m^{1+\beta}\|b\|_{L^\infty( B_{\infty,\infty}^{-\beta})} O_m(t)   
	\int_{t_0}^t
	\big ( \nu (t-s)\big )^{\frac{  1 }{2}} 
	ds 
	\nonumber \\
	&\leq & 
	C m^{1+\beta}\|b\|_{L^\infty( B_{\infty,\infty}^{-\beta})} O_m(t)   
	\int_{t_0}^t
	\nu^{\frac 12}
	(t-s)^{\frac{  1 -\varepsilon}{2}} 
	\nu^{-\frac{\alpha_1 \varepsilon }{2\alpha_2}} |x-x'|^{\frac{\varepsilon}{2\alpha_2}}
	ds 
	,
	\nonumber \\
\end{eqnarray}
for a given $0<\varepsilon<1$, and because the \textit{off-diagonal} regime is in force, $|x-x'| \geq  \nu^{\alpha_1}(t-s)^{\alpha_2}$.

The solution $u^{m,\nu}$ is supposed to have a $\gamma$ regularity in the limit modulus of continuity \eqref{Def_modulus_continuity_nonlimit}
then
\begin{equation*}
	\gamma= \frac{\varepsilon}{2 \alpha_2} 
	\ \Rightarrow \
	\varepsilon = 2 \alpha_2 \gamma>0.
\end{equation*}
Also, we have to suppose that $\varepsilon<1$ (for positive time contribution purpose) which implies
\begin{equation}\label{condi_alpha2_off}
	\alpha_2 < \frac{1}{2 \gamma}.
\end{equation}
Let us consider equality between the two regimes from \eqref{ineq_RA_parab} and \eqref{ineq_RAc_parab}, i.e.
\begin{equation}\label{parab_CONDI_alpha}
	(1-\gamma) \alpha_2= \frac{1-2 \alpha_2 \gamma}{2}
	\
	\Leftrightarrow
	\
	\alpha_2= \frac{1}{2}<\frac{1}{2\gamma},
\end{equation}
which is exactly the parabolic scale, and also
\begin{equation*}
	\varepsilon = \gamma.
\end{equation*}
This is another way to see the difficulty to get a Lipschitz control, in our framework, because $\varepsilon$ has to be strictly lower than $1$.
\begin{eqnarray}
|R_{A^c}^{\tau,\xi,\xi'}(t,x,x') | \big |_{\tau=t,\xi=x,\xi'=x'} 
&\leq & 
C m^{1+\beta}\|b\|_{L^\infty( B_{\infty,\infty}^{-\beta})} O_m(t)   
\int_{t_0}^t
\nu^{\frac {1-\gamma}2}
(t-s)^{\frac{  1 -\gamma}{2}} |x-x'|^{\gamma}
ds 
,
\nonumber \\
\end{eqnarray}

\subsubsection{Final control of $	R^{\tau,\xi,\xi'}(t,x,x')$}
We then conclude
\begin{eqnarray*}
	&& |R^{\tau,\xi,\xi'}(t,x,x')| \big  |_{\tau=t,\xi=\xi'=x}
	\nonumber \\
	&\leq &
	C 	m^{1+\beta}  \|b\|_{L^\infty( B_{\infty,\infty}^{-\beta})}	  |x-x'|^\gamma
	(O_m(t)+O_m^{(2)}(t))
	\int_0^t  ( \nu^{\alpha_1 (1-\gamma)}+ 	\nu^{\frac{1}{2}-\alpha_1 \gamma }  ) (t-s)^{\frac{1-\gamma}2} ds
	.
\end{eqnarray*}
Let us also consider the equality of the exponent of $\nu$,
\begin{equation*}
	\alpha_1(1-\gamma)= \frac{1}{2}-\alpha_1 \gamma 
	\ 
	\Leftrightarrow
	\
	\alpha_1 = \frac{1}{2},
\end{equation*}
the usual parabolic scale $\alpha_1= \alpha_2= \frac 12$ is finally in force. 
We deduce,
\begin{equation}\label{parab_ineq_R}
	|R^{\tau,\xi,\xi'}(t,x,x')| \big  |_{\tau=t,\xi=\xi'=x}
	\leq 
	C 	m^{1+\beta}  \|b\|_{L^\infty( B_{\infty,\infty}^{-\beta})}	  |x-x'|^\gamma
	(O_m(t)+O_m^{(2)}(t))
	\int_0^t   \nu^{\frac{1-\gamma}2}(t-s)^{\frac{1-\gamma}2} ds
	.
\end{equation}

	\subsubsection{On the discontinuous choice of freezing parameters : consequence for $u^{m,\nu}(t,x')$}
\label{sec_discon_freez}
\label{sec_discontinuous_u_x}

Let us carefully point out that even if the solution $u^{m,\nu}$ \textbf{does not} depend on the corresponding freezing parameter $\xi$, the choice of $\xi$ in this section \textbf{does} depend on the current time variable of integration $s$.


Therefore,  like for the approach developed in \cite{chau:hono:meno:18}, the \textit{cut locus} time yields an additional contribution.
\\

Previously, in the H\"older norm controls, we considered two points $(x,x') \in \R^d\times \R^d$.
Let us specify how to write the solution $u^{m,\nu}(t,x')$ with the different choices of freezing parameter $\xi'$ depending on the time variable of integration $s$. 
To do so, we first rewrite the theoretical representation of the solution 
where the initial time is $r \in [0,T]$, and the initial function is replaced by $u^{m,\nu}(r,x')$,
\begin{equation}\label{representation_sol_u_r}
u^{m,\nu}(t,x')=  P_{r}^{m,\nu} u^{m,\nu}(r,x') + G_{r}^{m,\nu}f_m(t,x') ,
\end{equation}
where $ P_{r,t}^{m,\nu} $ and $ G_{r,t}^{m,\nu}$ stand respectively for the semi-group and the Green operator associated with the Cauchy problem 	
\begin{equation}\label{Kolmo_intial_r}
\begin{cases}
\partial_t u^{m,\nu}(t,x)+  \langle b_m(t,x),  \nabla u^{m,\nu} (t,x)\rangle -\nu \Delta u^{m,\nu}(t,x)=f_m(t,x),\ (t,x)\in (r,T]\times \R^{d},\\
u^{m,\nu}(r,x)=u^{m,\nu}(r,x),\ x\in \R^{d}.
\end{cases}
\end{equation}
We also write
	\begin{equation}\label{Duhamel_u_r}
u^{m, \nu} (t,x) = \hat P_r^{\tau,\xi}  u^{m,\nu}(r,\cdot) 
(t,x) + \hat G_r^{\tau,\xi} f_m(t,x) 
+ \hat  G_r^{\tau,\xi}  \big ( b _{\Delta}^{m} [\tau,\xi]  \cdot  \nabla  u^{m, \nu} \big )(t,x) ,
\end{equation}
where the operators are defined by
	\begin{equation}\label{def_hat_G_r}
\forall (t,x) \in (0,T]\times\R^{d}, \ \hat  G_r ^{\tau,\xi}   f_m(t,x):= \int_r^{t}  \int_{\R^{d}}  \hat{p}^{\tau,\xi} (s,t, x,y)   f (s,y)  dy \, ds,
\end{equation}
and 
\begin{equation}\label{def_hat_P_r}
\hat  P_r^{\tau,\xi}    g_m(t,x):=\int_{\R^{d}}\hat  p^{\tau,\xi}  (r,t,x,y)   g_m(y)  dy.
\end{equation}

Let us recall the definition of the transition time
\begin{equation} \label{def_t_0}
t_0:=t-\nu^{-\frac{\alpha_1}{\alpha_2}}|x-x'|^{\frac{1}{\alpha_2}}=t-\nu^{-1}
|x-x'|^{2}
.
\end{equation} 

	If $t_0 \leq 0$ $\Leftrightarrow$ $t \leq \nu^{-1}|x-x'|^{2}$, the \textit{off-diagonal} regime is in force, then we pick $\xi'=x'$ and there is no intricate choice of the \textit{freezing} parameter. 
	\textcolor{black}{However, we cannot have the same control as \eqref{ineq_P_g_Holder}, we indeed need to handle with the flow regularity by the choice of different \textit{proxy}, see Section \ref{sec_comment_Holder_modulus} further.}
\\

However, if $t_0>0$ $\Leftrightarrow$ $t > \nu^{-1}|x-x'|^{2}$, we need to be more subtle to handle with the dependency on $s$ for the value choice of $\xi' \in \R^d$.
From now on, we suppose that $t_0>0$.
\textcolor{black}{In this case, we can get identity \eqref{ineq_P_g_Holder}, but the condition $|x-x'|^2 < \nu t $ implies that we do consider the \textit{limit H\"older continuity} defined in \eqref{Def_modulus_continuity} instead of the usual H\"older modulus.}
\\

We next differentiate \eqref{Duhamel_u_r}
w.r.t. $r$ 
\begin{equation}
0 =  \partial_r \big (\hat P_{r}^{\tau,\xi'} u^{m,\nu}(r,\cdot) \big )(t,x') + \partial_r \hat G_{r}^{\tau,\xi'}f _m(t,x') +\partial_r \hat  G_r^{\tau,\xi'}  \big ( b _{\Delta}^{m} [\tau,\xi']  \cdot  \nabla  u^{m, \nu} \big )(t,x').
\end{equation}
We integrate the variable $r$ between $[t_0,t]$ with the \textit{proxy} parameter $\xi ' \in \R^d$, 
\begin{equation*}
	0= u^{m,\nu}(t,x') -  [P_{t_0}^{\tau,\xi'} u^{m,\nu}(t_0,\cdot)](t,x')-  G_{t_0}^{\tau,\xi'}f_m(t,x')-  \hat  G_{t_0}^{\tau,\xi'}  \big ( b _{\Delta}^{m} [\tau,\xi']  \cdot  \nabla  u^{m, \nu} \big )(t,x')
	,
\end{equation*}
which yields for $t \in [t_0,T]$
\begin{equation}\label{eq_off_diag_discon}
u^{m,\nu}(t,x')=
  [\hat P_{t_0}^{\tau,\xi'} u^{m,\nu}(t_0,\cdot)](t,x')+  G_{t_0}^{\tau,\xi'}f_m(t,x')+  \hat  G_{t_0}^{\tau,\xi'}  \big ( b _{\Delta}^{m} [\tau,\xi']  \cdot  \nabla  u^{m, \nu} \big )(t,x').
\end{equation}
Next, we integrate in time between $[0,t_0]$ with a different \textit{freezing} parameter $\tilde \xi' \in \R^d$,
\begin{eqnarray*}
0&=&   [\hat P_{t_0}^{\tau,\tilde \xi'} u^{m,\nu}(t_0,\cdot)](t,x')-\hat  P_{0}^{\tau,\tilde \xi'} g_m(t,x')
+  \hat G_{t_0}^{\tau,\tilde \xi'}f_m(t,x') -  \hat G_{0}^{\tau,\tilde \xi'}f_m(t,x')
\nonumber \\
&&
+ \hat  G_{t_0}^{\tau,\tilde \xi'}  \big ( b _{\Delta}^{m} [\tau,\tilde \xi']  \cdot  \nabla  u^{m, \nu} \big )(t,x')
- \hat  G_{0}^{\tau,\tilde \xi'}  \big ( b _{\Delta}^{m} [\tau,\tilde \xi']  \cdot  \nabla  u^{m, \nu} \big )(t,x')
.
\end{eqnarray*}
Hence,
\begin{eqnarray*}
u^{m,\nu}(t,x')&=&
[\hat P_{t_0}^{\tau,\xi'} u^{m,\nu}(t_0,\cdot)](t,x')+  G_{t_0}^{\tau,\xi'}f_m(t,x')+  \hat  G_{t_0}^{\tau,\xi'}  \big ( b _{\Delta}^{m} [\tau,\xi']  \cdot  \nabla  u^{m, \nu} \big )(t,x')
\nonumber \\
 &&- [\hat P_{t_0}^{\tau,\tilde \xi'} u^{m,\nu}(t_0,\cdot)](t,x')+\hat  P_{0}^{\tau,\tilde \xi'} g_m(t,x')
- \hat G_{t_0}^{\tau,\tilde \xi'}f_m(t,x') +  \hat G_{0}^{\tau,\tilde \xi'}f_m(t,x')
\nonumber \\
&&
- \hat  G_{t_0}^{\tau,\tilde \xi'}  \big ( b _{\Delta}^{m} [\tau,\tilde \xi']  \cdot  \nabla  u^{m, \nu} \big )(t,x')
+ \hat  G_{0}^{\tau,\tilde \xi'}  \big ( b _{\Delta}^{m} [\tau,\tilde \xi']  \cdot  \nabla  u^{m, \nu} \big )(t,x').
\end{eqnarray*}
Defining 
	\begin{equation}\label{def_hat_G_rt}
\forall (t',x) \in [0,t]\times\R^{d}, \ \hat  G_{r,t'} ^{\tau,\xi}   f_m(t,x):= \int_r^{t'}  \int_{\R^{d}}  \hat{p}^{\tau,\xi} (s,t, x,y)   f (s,y)  dy \, ds,
\end{equation}
we  write
\begin{eqnarray}\label{eq_u_r_final}
u^{m,\nu}(t,x')
&=&
[\hat P_{t_0}^{\tau,\xi'} u^{m,\nu}(t_0,\cdot)](t,x')
- [\hat P_{t_0}^{\tau,\tilde \xi'} u^{m,\nu}(t_0,\cdot)](t,x')
+  G_{t_0}^{\tau,\xi'}f_m(t,x')
\nonumber 
\\
&&
+  \hat  G_{t_0}^{\tau,\xi'}  \big ( b _{\Delta}^{m} [\tau,\xi']  \cdot  \nabla  u^{m, \nu} \big )(t,x')
+\hat  P_{0}^{\tau,\tilde \xi'} g_m(x')
+  \hat G_{0,t_0}^{\tau,\tilde \xi'}f_m(t,x')
\nonumber 
\\
&&
+ \hat  G_{0,t_0}^{\tau,\tilde \xi'}  \big ( b _{\Delta}^{m} [\tau,\tilde \xi']  \cdot  \nabla  u^{m, \nu} \big )(t,x').
\end{eqnarray}
There is an extra contribution $[\hat P_{t_0}^{\tau,\xi'} u^{m,\nu}(t_0,\cdot)](t,x')
- [\hat P_{t_0}^{\tau,\tilde \xi'} u^{m,\nu}(t_0,\cdot)](t,x')$ due to the discontinuous \textit{freezing} choice, the other terms match with the ones appearing in the above computations.

\subsubsection{Extra contribution $[\hat P_{t_0}^{\tau,\xi'} u^{m,\nu}(t_0,\cdot)](t,x')
	- [\hat P_{t_0}^{\tau,\tilde \xi'} u^{m,\nu}(t_0,\cdot)](t,x')$ }
\label{sec_extra_discon}

Thanks to a change of variables, we readily obtain
\begin{eqnarray*}
&&	[\hat P_{t_0}^{\tau,\xi'} u^{m,\nu}(t_0,\cdot)](t,x')
	- [\hat P_{t_0}^{\tau,\tilde \xi'} u^{m,\nu}(t_0,\cdot)](t,x')
\nonumber \\
	&=&
	 \int_{\R^{d}}  \hat{p}^{\tau,\xi'} (t_0,t, x',y)   u^{m,\nu}(t_0,y)  dy 
	 -	 \int_{\R^{d}}  \hat{p}^{\tau,\tilde \xi'} (t_0,t, x',y)   u^{m,\nu}(t_0,y)  dy 
	 \nonumber \\
	 &=&
	 \int_{\R^{d}}  \tilde {p}(t_0,t, x',y)  
	  \Big [u^{m,\nu}\big (t_0,y +\int_{t_0}^t b_m(\tilde s, \theta_{\tilde s, \tau '}^m (\xi'))d\tilde s\big )  dy 
	- u^{m,\nu}\big (t_0,y +\int_{t_0}^t b_m(\tilde s, \theta_{\tilde s, \tilde \tau '}^m (\tilde \xi')d\tilde s)\big ) \Big ] dy,  
\end{eqnarray*}
recalling that $\tilde {p}(t_0,t, x',y)$ stands for the usual heat kernel defined in \eqref{def_tilde_p}.
Therefore,
\begin{eqnarray*}
&& \Big |	[\hat P_{t_0}^{\tau,\xi'} u^{m,\nu}(t_0,\cdot)](t,x')
- [\hat P_{t_0}^{\tau,\tilde \xi'} u^{m,\nu}(t_0,\cdot)](t,x') \Big |
\nonumber \\
&\leq &
\|u^{m,\nu}\|_{L^\infty(C^1)}
\int_{t_0}^t \big | b_m(\tilde s, \theta_{\tilde s, \tau}^m (\xi'))- b_m(\tilde s, \theta_{\tilde s,  \tau}^m (\tilde \xi')\big | d\tilde s 
\nonumber \\
&\leq &
\|u^{m,\nu}\|_{L^\infty(C^1)}
\|b_m\|_{L^\infty(C^1)}
\int_{t_0}^t | \theta_{\tilde s, \tau '}^m (\xi')- \theta_{\tilde s, \tau }^m (\tilde \xi')|d\tilde s .
\end{eqnarray*}
Additionally, we get the \textit{a priori} control for the flow.
\begin{lemma}\label{lemma_flot}
	For all $(x,x') \in \R^d \times \R^d$ and $0 \leq s \leq \tau \leq T$:
	\begin{eqnarray*}
		\sup_{\tilde s \in [0,\tau ]}	|\theta_{s,\tau }^m(x)-\theta_{s,\tau}^m(x')|
		&\leq &
		|x-x'|
		\exp \big ( \|b_m\|_{L^\infty(C^1)} \tau  \big ).
	\end{eqnarray*}
\end{lemma}
 The proof is postponed in Section \ref{sec_flow}.
 \\
 
We then deduce for $(\tau,\xi',\tilde \xi')=(t,x',x)$ that
\begin{eqnarray*}
&& \Big |	[\hat P_{t_0}^{\tau,\xi'} u^{m,\nu}(t_0,\cdot)](t,x')
- [\hat P_{t_0}^{\tau,\tilde \xi'} u^{m,\nu}(t_0,\cdot)](t,x') \Big | \Bigg |_{(\tau,\xi',\tilde \xi')=(t,x',x)}
\nonumber \\
&\leq &
\|u^{m,\nu}\|_{L^\infty(C^1)}
\|b_m\|_{L^\infty(C^1)}
\int_{t_0}^t 	|x-x'| \exp \big ( \|b_m\|_{L^\infty(C^1)} t  \big ) d\tilde s 
\nonumber \\
&\leq &
\|u^{m,\nu}\|_{L^\infty(C^1)}
\|b_m\|_{L^\infty(C^1)}
	|x-x'| (t-t_0) \exp \big ( \|b_m\|_{L^\infty(C^1)} t  \big ) d\tilde s .
\end{eqnarray*}
Recalling that, see Lemma \ref{lemma_apriori},
\begin{equation*}
	\|\nabla u^{m,\nu}(t,\cdot)\|_{L^\infty} \leq \Big ( T \| f_m\|_{L^\infty(C^1)}+\|g_m\|_{C^1} \Big ) \exp \Big ( C m^{1+\beta} t \|b\|_{L^\infty( B_{\infty,\infty}^{-\beta})} \Big )= :O_m(t).
\end{equation*}
Therefore,
\begin{eqnarray}
&& \Big |	[\hat P_{t_0}^{\tau,\xi'} u^{m,\nu}(t_0,\cdot)](t,x')
- [\hat P_{t_0}^{\tau,\tilde \xi'} u^{m,\nu}(t_0,\cdot)](t,x') \Big | \Bigg |_{(\tau,\xi',\tilde \xi')=(t,x',x)}
\nonumber \\
&\leq &
m^{1+\beta} O_m(t)
\|b\|_{L^\infty(B_{\infty,\infty}^{-\beta})}
|x-x'| (t-t_0) \exp \big ( \|b_m\|_{L^\infty(C^1)} t  \big ) 
 .
\end{eqnarray}
Because $|x-x'|= \nu^{\alpha_1} (t-t_0)^{\alpha_1}$, we get
\begin{eqnarray}\label{ineq_extra_Pu_prelim}
&& \Big |	[\hat P_{t_0}^{\tau,\xi'} u^{m,\nu}(t_0,\cdot)](t,x')
- [\hat P_{t_0}^{\tau,\tilde \xi'} u^{m,\nu}(t_0,\cdot)](t,x') \Big | \Bigg |_{(\tau,\xi',\tilde \xi')=(t,x',x)}
\nonumber \\
&\leq &
|x-x'| ^\gamma
m^{1+\beta} O_m(t)
\|b\|_{L^\infty(B_{\infty,\infty}^{-\beta})} 
\nu^{\frac{1-\gamma}2}
(t-t_0)^{\frac{1-\gamma}2 +1} \exp \big ( \|b_m\|_{L^\infty(C^1)} t  \big ) 
.
\end{eqnarray}
As we have supposed that $t_0\geq 0$, it follows that
\begin{eqnarray}\label{ineq_extra_Pu}
	&& \Big |	[\hat P_{t_0}^{\tau,\xi'} u^{m,\nu}(t_0,\cdot)](t,x')
	- [\hat P_{t_0}^{\tau,\tilde \xi'} u^{m,\nu}(t_0,\cdot)](t,x') \Big | \Bigg |_{(\tau,\xi',\tilde \xi')=(t,x',x)}
	\nonumber \\
	&\leq &
	t^{\frac{1-\gamma}2 +1}
	|x-x'| ^\gamma
	m^{1+\beta} O_m(t)
	\|b\|_{L^\infty(B_{\infty,\infty}^{-\beta})} 
	\nu^{\frac{1-\gamma}2}
	\exp \big (m^{1+\beta}  \|b\|_{L^\infty(B_{\infty,\infty}^{-\beta})} t  \big ) 
	.
\end{eqnarray}
We finally derive, from definition of $O_m(t)$ 
 and 
by exponential absorption \eqref{ineq_absorb}, 
\begin{eqnarray}\label{ineq_extra1}
	&& \Big |	[\hat P_{t_0}^{\tau,\xi'} u^{m,\nu}(t_0,\cdot)](t,x')
	- [\hat P_{t_0}^{\tau,\tilde \xi'} u^{m,\nu}(t_0,\cdot)](t,x') \Big | \Bigg |_{(\tau,\xi',\tilde \xi')=(t,x',x)}
	\nonumber \\
	 &\leq &
	C t^{\frac{1-\gamma}2 }
	 |x-x'| ^\gamma
		 \nu^{\frac{1-\gamma}2}
	\Big ( T \| f_m\|_{L^\infty(C^1)}+\|g_m\|_{C^1} \Big ) \exp \Big ( C m^{1+\beta} t \|b\|_{L^\infty( B_{\infty,\infty}^{-\beta})} \Big )
	.
\end{eqnarray}
The above term goes to $0$ with the vanishing condition 
\begin{equation}\label{nu_contrainte_extra1}
\nu \ll
T^{-1} 	\Big ( T \| f_m\|_{L^\infty(C^1)}+\|g_m\|_{C^1} \Big ) ^{-\frac{2}{1-\gamma}}
\exp \big (- C \frac{m^{1+\beta} \|b\|_{L^\infty(B_{\infty,\infty}^{-\beta})} T}{1-\gamma}  \big ) .
\end{equation}

\subsubsection{Justification of the \textit{freezing} point change}
\label{sec_justification_freez}

Eventually, let us explicitly write the difference of the two final Duhamel formulas associated with $	u^{m,\nu}(t,x)$ and $u^{m,\nu}(t,x') $.

For any $(t,x,x') \in[0,T] \times \R^d \times \R^d$ and $(\tau,\xi,\xi', \tilde \xi') \in [0,T] \times \R^d$, we write from \eqref{eq_u_r_final},
\begin{eqnarray*}
	u^{m,\nu}(t,x)-u^{m,\nu}(t,x') 
	&=& 
\Big [	\hat P^{\tau,\xi}  g_m (t,x)-\hat P^{\tau,\tilde \xi'}  g_m (t,x')\Big ] + \Big [\hat G_{t_0,t}^{\tau,\xi} f_m(t,x) -\hat G_{t_0,t}^{\tau,\xi'} f_m(t,x') \Big ]
\nonumber \\
&&+ \Big [\hat G_{0,t_0}^{\tau,\xi} f_m(t,x) -\hat G_{0,t_0}^{\tau,\tilde \xi'} f_m(t,x') \Big ]
\nonumber \\
&&+ \Big [ \tilde G_{t_0,t}^{\tau,\xi}  \big [ b _{\Delta}^{m} [\tau,\xi] \cdot  \nabla  u^{m,\nu} \big ](t,x)-\tilde G_{t_0,t}^{\tau,\xi'}  \big [ b _{\Delta}^{m'}[\tau,\xi'] \cdot  \nabla  u^{m,\nu} \big ](t,x') \Big ]
\nonumber \\
&&+ \Big [ \tilde G_{0,t_0}^{\tau,\xi}  \big [ b _{\Delta}^{m} [\tau,\xi] \cdot  \nabla  u^{m,\nu} \big ](t,x)-\tilde G_{0,t_0}^{\tau,\xi'}  \big [ b _{\Delta}^{m'}[\tau,\xi'] \cdot  \nabla  u^{m,\nu} \big ](t,x') \Big ]
\nonumber \\
&&+
[\hat P_{t_0}^{\tau,\xi'} u^{m,\nu}(t_0,\cdot)](t,x')
- [\hat P_{t_0}^{\tau,\tilde \xi'} u^{m,\nu}(t_0,\cdot)](t,x')
.
\end{eqnarray*}
Because the l.h.s. of the first equality does not depend on $(\xi,\xi')$, we can get the infimum over these \textit{freezing} points, namely
\begin{eqnarray*}
	&& u^{m,\nu}(t,x)-u^{m,\nu}(t,x') 
	\nonumber \\
	&=& \inf_{\xi,\xi',\tilde \xi ' \in \R } \Bigg \{
\Big [	\hat P^{\tau,\xi}  g_m (t,x)-\hat P^{\tau,\tilde \xi'}  g_m (t,x')\Big ] + \Big [\hat G_{t_0,t}^{\tau,\xi} f_m(t,x) -\hat G_{t_0,t}^{\tau,\xi'} f_m(t,x') \Big ]
\nonumber \\
&&+ \Big [\hat G_{0,t_0}^{\tau,\xi} f_m(t,x) -\hat G_{0,t_0}^{\tau,\tilde \xi'} f_m(t,x') \Big ]
\nonumber \\
&&+ \Big [ \tilde G_{t_0,t}^{\tau,\xi}  \big [ b _{\Delta}^{m} [\tau,\xi] \cdot  \nabla  u^{m,\nu} \big ](t,x)-\tilde G_{t_0,t}^{\tau,\xi'}  \big [ b _{\Delta}^{m'}[\tau,\xi'] \cdot  \nabla  u^{m,\nu} \big ](t,x') \Big ]
\nonumber \\
&&+ \Big [ \tilde G_{0,t_0}^{\tau,\xi}  \big [ b _{\Delta}^{m} [\tau,\xi] \cdot  \nabla  u^{m,\nu} \big ](t,x)-\tilde G_{0,t_0}^{\tau,\xi'}  \big [ b _{\Delta}^{m'}[\tau,\xi'] \cdot  \nabla  u^{m,\nu} \big ](t,x') \Big ]
\nonumber \\
&&+
[\hat P_{t_0}^{\tau,\xi'} u^{m,\nu}(t_0,\cdot)](t,x')
- [\hat P_{t_0}^{\tau,\tilde \xi'} u^{m,\nu}(t_0,\cdot)](t,x')\Bigg \}.
\end{eqnarray*}
However, we aim to control the source functions term only in the \textit{diagonal} regime, see Section \ref{sec_comm_diag_G_P} below for details.
We rewrite,
\begin{eqnarray*}
&&	\Big [\hat G_{t_0,t}^{\tau,\xi} f_m(t,x) -\hat G_{t_0,t}^{\tau,\xi'} f_m(t,x') \Big ]
	+ \Big [\hat G_{0,t_0}^{\tau,\xi} f_m(t,x) -\hat G_{0,t_0}^{\tau,\tilde \xi'} f_m(t,x') \Big ]
	\nonumber \\
	&=& \Big [\hat G^{\tau,\xi} f_m(t,x) -\hat G^{\tau,\tilde \xi'} f_m(t,x') \Big ]
+ \Big [\hat G_{0,t_0}^{\tau,\xi'} f_m(t,x') -\hat G_{0,t_0}^{\tau,\tilde \xi'} f_m(t,x') \Big ],
\end{eqnarray*}
and, 
\begin{eqnarray}\label{identi_Holder_final}
	&&u^{m,\nu}(t,x)-u^{m,\nu}(t,x') 
	\nonumber \\
	&=& \inf_{\xi,\xi',\tilde \xi ' \in \R } \Bigg \{
	\Big [	\hat P^{\tau,\xi}  g_m (t,x)-\hat P^{\tau,\tilde \xi'}  g_m (t,x')\Big ] + \Big [\hat G^{\tau,\xi} f_m(t,x) -\hat G^{\tau,\tilde \xi'} f_m(t,x') \Big ]
	\nonumber \\
	&&+ \Big [\hat G_{0,t_0}^{\tau,\xi'} f_m(t,x') -\hat G_{0,t_0}^{\tau,\tilde \xi'} f_m(t,x') \Big ]
	\nonumber \\
	&&+ \Big [ \tilde G_{t_0,t}^{\tau,\xi}  \big [ b _{\Delta}^{m} [\tau,\xi] \cdot  \nabla  u^{m,\nu} \big ](t,x)-\tilde G_{t_0,t}^{\tau,\xi'}  \big [ b _{\Delta}^{m'}[\tau,\xi'] \cdot  \nabla  u^{m,\nu} \big ](t,x') \Big ]
	\nonumber \\
	&&+ \Big [ \tilde G_{0,t_0}^{\tau,\xi}  \big [ b _{\Delta}^{m} [\tau,\xi] \cdot  \nabla  u^{m,\nu} \big ](t,x)-\tilde G_{0,t_0}^{\tau,\xi'}  \big [ b _{\Delta}^{m'}[\tau,\xi'] \cdot  \nabla  u^{m,\nu} \big ](t,x') \Big ]
	\nonumber \\
	&&+
	[\hat P_{t_0}^{\tau,\xi'} u^{m,\nu}(t_0,\cdot)](t,x')
	- [\hat P_{t_0}^{\tau,\tilde \xi'} u^{m,\nu}(t_0,\cdot)](t,x')\Bigg \},
\end{eqnarray} 
because
we have 
$\hat  G^{\tau,\tilde \xi'}f_m(t,x') = \hat  G_{0,t_0}^{\tau,\tilde \xi'}f_m(t,x')+ \hat G^{\tau,\tilde \xi'}_{t_0,t} f_m(t,x')$.

Hence, taking $(\tau,\xi,\xi',\tilde \xi ')= (t,x,x',x)$ yields the previous terms already controlled plus a new extra contribution $\Big [\hat G_{0,t_0}^{\tau,\xi'} f_m(t,x') -\hat G_{0,t_0}^{\tau,\tilde \xi'} f_m(t,x') \Big ]$.

\subsubsection{Control  of the new extra contribution $\Big [\hat G_{0,t_0}^{\tau,\xi'} f_m(t,x') -\hat G_{0,t_0}^{\tau,\tilde \xi'} f_m(t,x') \Big ]$}
\label{sec_new_extra}

This last extra term is tackled similarly as the first one, see Section \ref{sec_extra_discon}.
\begin{eqnarray*}
&&	\Big [\hat G_{0,t_0}^{\tau,\xi'} f_m(t,x') -\hat G_{0,t_0}^{\tau,\tilde \xi'} f_m(t,x') \Big ]
\nonumber \\
&=&
\int_0^{t_0} 
\int_{\R^{d}}  \tilde {p}(s,t, x',y)   \Big [f_m\big (s,y +\int_{s}^t b_m(\tilde s, \theta_{\tilde s, \tau '}^m (\xi'))d\tilde s\big )   
- f_m\big (s,y +\int_{s}^t b_m(\tilde s, \theta_{\tilde s, \tilde \tau '}^m (\xi')d\tilde s)\big ) \Big ] dy \, ds .
\end{eqnarray*}
We readily obtain,
\begin{eqnarray*}
&& 
\Big |\hat G_{0,t_0}^{\tau,\xi'} f_m(t,x') -\hat G_{0,t_0}^{\tau,\tilde \xi'} f_m(t,x') \Big |
\nonumber \\
&\leq &
\|f_m\|_{L^\infty(C^1)}
\int_0^{t_0}
\Big |\int_{s}^t b_m(\tilde s, \theta_{\tilde s, \tau}^m (\xi'))- b_m(\tilde s, \theta_{\tilde s,  \tau}^m (\tilde \xi')d\tilde s \Big | ds
\nonumber \\
&\leq &
\|f_m\|_{L^\infty(C^1)}
\|b_m\|_{L^\infty(C^1)}
\int_0^{t_0}\int_{s}^t | \theta_{\tilde s, \tau '}^m (\xi')- \theta_{\tilde s, \tau }^m (\tilde \xi')|d\tilde s \, ds.
\end{eqnarray*}
Again, we use, for any $(x,x') \in \R^d \times \R^d$, the control of the flow,
\begin{eqnarray}\label{ineq_theta_peruve}
\sup_{\tilde s \in [0,\tau ]}	|\theta_{s,\tau }^m(x)-\theta_{s,\tau}^m(x')|
&\leq &
|x-x'|
\exp \big ( \|b_m\|_{L^\infty(C^1)} \tau  \big ),
\end{eqnarray}
see Lemma \ref{lemma_flot}.

We then deduce for $(\tau,\xi',\tilde \xi')=(t,x',x)$ that
\begin{eqnarray}\label{ineq_extra2}
	&& 
	\Big |\hat G_{0,t_0}^{\tau,\xi'} f_m(t,x') -\hat G_{0,t_0}^{\tau,\tilde \xi'} f_m(t,x') \Big |\Bigg |_{(\tau,\xi',\tilde \xi')=(t,x',x)}
	\nonumber \\
	&\leq &
	\|f_m\|_{L^\infty(C^1)}
	\|b_m\|_{L^\infty(C^1)}
	\int_0^{t_0}\int_{s}^t |x-x'|
	\exp \big ( \|b_m\|_{L^\infty(C^1)} t  \big ) d\tilde s \, ds
	\nonumber \\
	&\leq &
\frac{1}{2}m^{1-\gamma}	\|f\|_{L^\infty(C^\gamma)}
m^{1+\beta} \|b\|_{L^\infty(B_{\infty,\infty}^{-\beta})} 
	 |x-x'|t^2
\exp \big (m^{1+\beta} \|b\|_{L^\infty(B_{\infty,\infty}^{-\beta})} t  \big ) .
\end{eqnarray}
From the definition of $t_0$ in \eqref{def_t0}, $|x-x'|= \nu^{\frac 12} (t-t_0)^{\frac 12}$, we get
\begin{eqnarray} \label{ineq_extra_Gf}
&& 	\Big |\hat G_{0,t_0}^{\tau,\xi'} f_m(t,x') -\hat G_{0,t_0}^{\tau,\tilde \xi'} f_m(t,x') \Big | \Bigg |_{(\tau,\xi',\tilde \xi')=(t,x',x)}
\nonumber \\
	&\leq &	
	|x-x'| ^\gamma
\frac{\nu^{\frac{1-\gamma}2} t^{\frac{1-\gamma}2 +2}}{2}	
m^{2+\beta-\gamma} \|f\|_{L^\infty(C^\gamma)} \|b\|_{L^\infty(B_{\infty,\infty}^{-\beta})} 
\exp \big (m^{1+\beta} \|b\|_{L^\infty(B_{\infty,\infty}^{-\beta})} t  \big ) 
\nonumber \\
&\leq  &
C	|x-x'| ^\gamma
\nu^{\frac{1-\gamma}2} t^{\frac{1-\gamma}2+1}	
m^{1-\gamma}
 \|f\|_{L^\infty(C^\gamma)} 
\exp \big (c m^{1+\beta} \|b\|_{L^\infty(B_{\infty,\infty}^{-\beta})} t  \big ) 
,
\end{eqnarray}
which goes to $0$ if
\begin{equation}\label{nu_contrainte_extra2}
	\nu \ll (m^{1-\gamma}  T^{\frac{1-\gamma}{2} +1} \|f\|_{L^\infty(C^\gamma)}  )^{-\frac{2}{1-\gamma}}
	\exp \big (- c \frac{m^{1+\beta} \|b\|_{L^\infty(B_{\infty,\infty}^{-\beta})} T}{1-\gamma}  \big ) .
\end{equation}
We postpone, in Section \ref{sec_comm_diag_G_P}, a full comment on the strategy associated with the choice of \textit{freezing} point for the source functions terms.

\subsubsection{Control of $[ u^{m,\nu}(t,\cdot)]_{\gamma,\nu^{1/2}T^{1/2}}$}
\label{sec_holder_final}

We have from the final Duhamel formula \eqref{eq_u_r_final} combined with the estimates of each contribution stated in \eqref{parab_ineq_R}
\eqref{ineq_extra1},
\eqref{ineq_extra_Gf}:
\begin{eqnarray}\label{ineq_u_Holder_prelim_alpha}
	&&[ u^{m,\nu}(t,\cdot)]_{\gamma,\nu^{1/2}T^{1/2}}
	\nonumber \\
	&\leq&
	\int_0^t [f(s,\cdot)]_{\gamma,\nu^{1/2}T^{1/2}} ds + [g]_{\gamma,\nu^{1/2}T^{1/2}}
	\nonumber \\
	&& +	
	C 	m^{1+\beta}  \|b\|_{L^\infty( B_{\infty,\infty}^{-\beta})}	 
(O_m(t)+O_m^{(2)}(t))
\int_0^t   \nu^{\frac{1-\gamma}2}(t-s)^{\frac{1-\gamma}2} ds
	\nonumber \\
	&&+
		C t^{\frac{1-\gamma}2 }
	\nu^{\frac{1-\gamma}2}
	\Big ( T \| f_m\|_{L^\infty(C^1)}+\|g_m\|_{C^1} \Big ) \exp \Big ( C m^{1+\beta} t \|b\|_{L^\infty( B_{\infty,\infty}^{-\beta})} \Big )
	\nonumber \\
	&&+ 
C
\nu^{\frac{1-\gamma}2} t^{\frac{1-\gamma}2+1}	
m^{1-\gamma}
\|f\|_{L^\infty(C^\gamma)} 
\exp \big (c m^{1+\beta} \|b\|_{L^\infty(B_{\infty,\infty}^{-\beta})} t  \big ) 
 ,
\end{eqnarray}
where we recall that 
\begin{eqnarray*}
	O_m(t) &=&\Big ( t \| f_m\|_{L^\infty(C^1)}+\|g_m\|_{C^1} \Big ) \exp \Big ( m^{1+\beta} t \|b\|_{L^\infty( B_{\infty,\infty}^{-\beta})} \Big )
	\nonumber \\
	&\leq &
	C m^{1-\gamma}(t  \| f\|_{L^\infty(C^\gamma)}+ [g]_{\gamma} )  \exp \Big ( m^{1+\beta} t \|b\|_{L^\infty( B_{\infty,\infty}^{-\beta})} \Big ) ,
\end{eqnarray*}
and
\begin{equation*}
 O^{(2)}_m(t)=
C	 	m^{2-\gamma}\big ( t \| f\|_{L^\infty(C^\gamma)}+[g]_\gamma\big )  
\exp(C  m^{1+\beta}t \|b\|_{L^\infty( B_{\infty,\infty}^{-\beta})}).
\end{equation*}
Then, for $\nu \ll T$ and because $\frac{1-\gamma^2}{4}>\frac{\gamma(1-\gamma)}{4}$, we obtain
\begin{eqnarray*}
[ u^{m,\nu}(t,\cdot)]_{\gamma,\nu^{1/2}T^{1/2}}
&\leq& 
\int_0^t [f(s,\cdot)]_{\gamma,\nu^{1/2}T^{1/2}} ds+ [g]_{\gamma,\nu^{1/2}T^{1/2}}
\nonumber \\
&&
+
C t^{\frac{1-\gamma}2 }
\nu^{\frac{1-\gamma}2}  m^{2-\gamma}(t  \| f\|_{L^\infty(C^\gamma)}+ [g]_{\gamma} )  \exp \Big (c m^{1+\beta} t \|b\|_{L^\infty( B_{\infty,\infty}^{-\beta})} \Big ) .
\end{eqnarray*}
Finally, we can also impose the condition on the viscosity
\begin{eqnarray}\label{ineq_Holder_condi_nu}
\nu
&\ll& T^{-1} m^{-\frac{2-\gamma}{1-\gamma}} (t  \| f\|_{L^\infty(C^\gamma)}+ [g]_{\gamma} ) ^{-\frac{2}{1-\gamma}}
\exp \Big (-c\frac{m^{1+\beta} T \|b\|_{L^\infty( B_{\infty,\infty}^{-\beta})}} {1-\gamma}\Big ),
\end{eqnarray}

the required limit modulus of continuity control \eqref{ineq_THEO_SCHAU} is then established when  
$(m,\nu) \to (+ \infty,0)$ according to the above condition.


We fail to obtain a uniqueness of a viscous selection principle with this first analysis. Indeed in the \textit{a priori} controls of \eqref{parabolic_moll_def} we have to suppose that $\nu$ goes to $0$ much faster than $m$ towards $+ \infty$. 
This constraint prevents us to take advantage of the convergence of $b_m$ towards $b$ to balance the blow up in the viscosity $\nu$ occurring in the computations to get uniqueness.

	\subsection{The last part of the control of H\"older modulus: the \textit{time decomposition locus} trick}
\label{sec_holder_time_cut}
\label{sec_proof_parab}

The beginning of the analysis for the parabolic case is similar as for the transport equation performed in Section \ref{sec_Proof_transport}, except that the goal of the regularisation is different. We aim here to raise the time contribution instead of the viscosity one.
This allows us to conclude, thanks to  a time decomposition trick, without any vanishing viscosity.
\\

We recall,
\begin{equation}\label{parab_ineq_R_recall}
|R^{\tau,\xi,\xi'}(t,x,x')| \big  |_{\tau=t,\xi=\xi'=x}
\leq 
C 	m^{1+\beta}  \|b\|_{L^\infty( B_{\infty,\infty}^{-\beta})}	  |x-x'|^\gamma
(O_m(t)+O_m^{(2)}(t))
\int_0^t   \nu^{\frac{1-\gamma}2}(t-s)^{\frac{1-\gamma}2} ds
.
\end{equation}

\subsubsection{Adapting the controls of the extra contributions}

We have to change the parameters $(\alpha_1, \alpha_2)$ into \eqref{ineq_extra_Pu} and \eqref{ineq_extra_Gf}, which gives
\begin{eqnarray}\label{parab_ineq_extra_Pu}
&& \Big |	[\hat P_{t_0}^{\tau,\xi'} u^{m,\nu}(t_0,\cdot)](t,x')
- [\hat P_{t_0}^{\tau,\tilde \xi'} u^{m,\nu}(t_0,\cdot)](t,x') \Big | \Bigg |_{(\tau,\xi',\tilde \xi')=(t,x',x)}
\nonumber \\
&\leq &
|x-x'| ^\gamma
m^{1+\beta} O_m(t)
\|b\|_{L^\infty(B_{\infty,\infty}^{-\beta})} 
\nu^{\frac{1-\gamma}{2}}
(t-t_0)^{\frac{(1-\gamma)}2 +1} \exp \big ( \|b_m\|_{L^\infty(C^1)} t  \big ) 
,
\end{eqnarray}
and
\begin{eqnarray}\label{parab_ineq_extra_Gf}
&& 	\Big |\hat G_{0,t_0}^{\tau,\xi'} f_m(t,x') -\hat G_{0,t_0}^{\tau,\tilde \xi'} f_m(t,x') \Big | \Bigg |_{(\tau,\xi',\tilde \xi')=(t,x',x)}
\\
&\leq &	
|x-x'| ^\gamma
m^{2+\beta-\gamma} \|f\|_{L^\infty(C^\gamma)} \|b\|_{L^\infty(B_{\infty,\infty}^{-\beta})} 
\exp \big (m^{1+\beta} \|b\|_{L^\infty(B_{\infty,\infty}^{-\beta})} t  \big ) 
\int_0^{t_0}\int_{s}^t
\frac{\nu^{\frac{1-\gamma}{2}} (t-t_0)^{\frac{(1-\gamma)}2 }}{2}	 d \tilde  s \, d s
.\nonumber
\end{eqnarray}

\subsubsection{Gathering the controls}

To put in a nutshell, gathering all the previous estimates \eqref{ineq_P_g_Holder}, \eqref{ineq_Gf_Holder}, \eqref{parab_ineq_R}, \eqref{parab_ineq_extra_Pu}, \eqref{parab_ineq_extra_Gf} into \eqref{identi_Holder_final}
\begin{eqnarray}\label{parab_ineq_u_Holder_prelim}
&&[ u^{m,\nu}(t,\cdot)]_{\gamma,\nu^{1/2}T^{1/2}}
\nonumber \\
&\leq& 
\int_0^t [f(s,\cdot)]_{\gamma,\nu^{1/2}T^{1/2}} ds+ [g]_{\gamma,\nu^{1/2}T^{1/2}}
\nonumber \\
&&+
C 	m^{1+\beta}  \|b\|_{L^\infty( B_{\infty,\infty}^{-\beta})}	 
(O_m(t)+O_m^{(2)}(t))
\int_0^t   \nu^{\frac{1-\gamma}2}(t-s)^{\frac{1-\gamma}2} ds
\nonumber \\
&&+
Cm^{1+\beta} O_m(t)
\|b\|_{L^\infty(B_{\infty,\infty}^{-\beta})} 
\nu^{\frac{1-\gamma}{2}}
(t-t_0)^{\frac{(1-\gamma)}2 +1} \exp \big ( \|b_m\|_{L^\infty(C^1)} t  \big ) 
\\
&&+
m^{2+\beta-\gamma} \|f\|_{L^\infty(C^\gamma)} \|b\|_{L^\infty(B_{\infty,\infty}^{-\beta})} 
\exp \big (m^{1+\beta} \|b\|_{L^\infty(B_{\infty,\infty}^{-\beta})} t  \big ) 
\int_0^{t}\int_{s}^t
\frac{\nu^{\frac{1-\gamma}{2}} (t-t_0)^{\frac{(1-\gamma)}2 }}{2}	 d \tilde  s \, d s
, \nonumber
\end{eqnarray}
with 
\begin{eqnarray*}
	O_m(t) &=&\Big ( t \| f_m\|_{L^\infty(C^1)}+\|g_m\|_{C^1} \Big ) \exp \Big ( m^{1+\beta} t \|b\|_{L^\infty( B_{\infty,\infty}^{-\beta})} \Big )
	\nonumber \\
	&\leq &
	C\Big ( m^{1-\gamma}(t  \| f\|_{L^\infty(C^\gamma)}+ [g]_{\gamma} ) \Big ) \exp \Big ( m^{1+\beta} t \|b\|_{L^\infty( B_{\infty,\infty}^{-\beta})} \Big ) .
\end{eqnarray*}
Now, the idea is to make negligible the terms involving positive time contribution in the time integral.

\subsection{The time decomposition of the Cauchy problem}
\label{sec_decoup_temps}

We cut the Cauchy problem in small intervals of $[0,T]$, and we see, from computations below, that the first order term $b^m _ \Delta \cdot \nabla  u^{m,\nu}$ and the extra terms are negligible when the size of the time intervals goes to $0$. 

Let $n \in \N$, and for each $k \in \leftB 0, n \rightB$,
\begin{equation*}
\tau_k:= \frac{k}{n}T,
\end{equation*}
we also denote, for all $x \in \R^3$, $k \in \leftB 0, n-1 \rightB$ 
and $t \in (\tau_{k}, \tau_{k+1}]$,
\begin{equation*}
\begin{cases}
u^{m,\nu}_{k+1} (t,x)&:= u^{m,\nu}  (t,x),
\\
u^{m,\nu}_1 (0,x) &:= g_m(x)	.
\end{cases}
\end{equation*}
The associated Cauchy problems write
\begin{trivlist}
	\item[If $k \in \leftB 1, n-1 \rightB$]
	\begin{equation*}
	\begin{cases}
	\partial_t  u^{m,\nu}_{k+1} (t,x)+ \langle b_m    , \nabla   u^{m,\nu}_{k+1} \rangle (t,x)
	= \nu \Delta u^{m,\nu}_{k+1} (t,x)+   f_m(t,x) 
	, \ t \in (\tau_k,\tau_{k+1}]\\
	u^{m,\nu} _{k+1}(\tau_k,x)= u^{m,\nu}_{k} (\tau_k,x), 
	\end{cases}
	\end{equation*}
	\item[if $k=0$] 
	\begin{equation*}
	\begin{cases}
	\partial_t  u^{m,\nu}_1 (t,x)+  \langle b_m    ,\nabla   u^{m,\nu}_1 \rangle (t,x)
	= \nu \Delta u^{m,\nu}_1 (t,x)+   f_m(t,x) 
	, \ t \in (0,\tau_1),\\
	u^{m,\nu} _1(0,x)= g_m(x). 
	\end{cases}
	\end{equation*}
\end{trivlist}
Using the Duhamel formulation \eqref{Duhamel_u} around the consider heat like equation, we get for any $t \in (\tau_k, \tau_{k+1}]$, if $k \in \leftB 1, n-1 \rightB$,
\begin{equation}\label{Duhamel_uk_FINAL}
u^{m,\nu}_{k+1}  (t,x) =  \hat P^{\tau,\xi}_{\tau_k} u^{m,\nu}_{k}  (t,x) +  \hat G^{\tau,\xi}_{\tau_k} f_m(t,x) 
+  \hat  G^{\tau,\xi}_{\tau_k}  \big ( \langle b^m _{\Delta}  [\tau,\xi]  , \nabla   u^{m,\nu}_{k} \rangle  \big )(t,x) ,
\end{equation}
and if  $k=0$,
\begin{equation*}
u^{m,\nu}_{1}  (t,x) =   \hat P^{\tau,\xi}_{0} g_m (t,x) +  \hat G^{\tau,\xi}_{0} f_m(t,x) 
+  \hat  G^{\tau,\xi}_{0}  \big ( \langle b^m _{\Delta}  [\tau,\xi]  ,  \nabla   u^{m,\nu} _1 \rangle \big )(t,x) .
\end{equation*}

Next, we use the corresponding H\"older control \eqref{parab_ineq_u_Holder_prelim}, 
including the extra contributions coming from the \textit{cut locus} argument,
\begin{eqnarray}\label{ineq_Holder_u_k}
&&[ u^{m,\nu}_{k+1}(t,\cdot)]_{\gamma,\nu^{1/2}(T/n)^{1/2}}
\nonumber \\
&\leq& \int^t_{\tau_k} [f(s,\cdot)]_{\gamma,\nu^{1/2}(T/n)^{1/2}} ds + [u_k^{m,\nu}(\tau_{k},\cdot)]_{\gamma,\nu^{1/2}(T/n)^{1/2}}
\nonumber \\
&&+
C 	m^{1+\beta}  \|b\|_{L^\infty( B_{\infty,\infty}^{-\beta})}	
(O_m(t)+O_m^{(2)}(t))
\int_{\tau_k}^t   \nu^{\frac{1-\gamma}2}(t-s)^{\frac{1-\gamma}2} ds
\nonumber \\
&&+
Cm^{1+\beta} O_m(t)
\|b\|_{L^\infty(B_{\infty,\infty}^{-\beta})} 
\nu^{\frac{1-\gamma}{2}}
(t-t_0)^{\frac{(1-\gamma)}2 +1} \exp \big ( \|b_m\|_{L^\infty(C^1)} t  \big ) 
\\
&&+
m^{2+\beta-\gamma} \|f\|_{L^\infty(C^\gamma)} \|b\|_{L^\infty(B_{\infty,\infty}^{-\beta})} 
\exp \big (m^{1+\beta} \|b\|_{L^\infty(B_{\infty,\infty}^{-\beta})} t  \big ) 
\int_{\tau_{k}}^{t}\int_{\tau_k}^t
\frac{\nu^{\frac{1-\gamma}{2}} (t-t_0)^{\frac{(1-\gamma)}2 }}{2}	 d \tilde  s \, d s
. \nonumber
\end{eqnarray}
Let us carefully point out that we suppose that $(t-t_0) \leq (t-\tau_{k})$, otherwise there is no off-diagonal regime and we cannot use the \textit{cut locus} technique.
This condition allows use to replace the H\"older modulus in the r.h.s. in \eqref{ineq_P_g_Holder} by the \textit{limit H\"older continuity} defined in \eqref{Def_modulus_continuity}.

Hence, recalling that $\tau_{k+1}-\tau_{k}= \frac{T}{n}$,
\begin{eqnarray}
&&
[u^{m,\nu}_{k+1}(t,\cdot)]_{\gamma,\nu^{1/2}(T/n)^{1/2}}
\nonumber \\
&\leq& \int^t_{\tau_k} [f(s,\cdot)]_{\gamma,\nu^{1/2}(T/n)^{1/2}} ds + [u_k^{m,\nu}(\tau_{k},\cdot)]_{\gamma,\nu^{1/2}(T/n)^{1/2}}
\nonumber \\
&&+
C 	m^{1+\beta}  \|b\|_{L^\infty( B_{\infty,\infty}^{-\beta})}	 
(O_m(t)+O_m^{(2)}(t))
\nu^{\frac{1-\gamma}2}\Big ( \frac{T}{n}\Big )^{\frac{1-\gamma}2} \int_{\tau_k}^t    ds
\nonumber \\
&&+
Cm^{1+\beta} O_m(t)
\|b\|_{L^\infty(B_{\infty,\infty}^{-\beta})} 
\nu^{\frac{1-\gamma}{2}}
\Big ( \frac{T}{n}\Big )^{\frac{1-\gamma}2} 
(t-\tau_{k}) \exp \big ( \|b_m\|_{L^\infty(C^1)} t  \big ) 
\\
&&+
C \nu^{\frac{1-\gamma}{2}}  m^{2+\beta-\gamma} \|f\|_{L^\infty(C^\gamma)} \|b\|_{L^\infty(B_{\infty,\infty}^{-\beta})} 
\exp \big (m^{1+\beta} \|b\|_{L^\infty(B_{\infty,\infty}^{-\beta})} t  \big ) 
(t-\tau_{k})^{\frac{(1-\gamma)}2+1 }
\int_{\tau_{k}}^{t}
d s
. \nonumber
\end{eqnarray}
For any $t \in [0,T]$, let $\kappa_n(t) = \lfloor \frac{nt}{T} \rfloor \in \leftB 0, n \rightB$, such that $\tau_{\kappa_n(t)} \leq t < \tau_{\kappa_n(t)+1}$, and iterating the above inequality,
\begin{equation*}
[ u^{m,\nu}(t,\cdot)]_{\gamma,\nu^{1/2}(T/n)^{1/2}}
\leq 
	[ g  ]_{\gamma,\nu^{1/2}(T/n)^{1/2}} + \Lambda_{\tau_{\kappa_n(t)},t}
+ \sum_{k=0}^{\kappa(t)-1} 
	 \Lambda_{\tau_{k},\tau_{k+1}},
\end{equation*}
with, for any $0\leq  r\leq t\leq T$, 
\begin{eqnarray}\label{def_Lambda}
\Lambda_{r,t} &:=& 	\int_{r}^{t} [  f(s,\cdot) ]_{\gamma,\nu^{1/2}(T/n)^{1/2}} ds 
\nonumber \\
&&+
C 	m^{1+\beta}  \|b\|_{L^\infty( B_{\infty,\infty}^{-\beta})}	  
(O_m(t)+O_m^{(2)}(t))
\nu^{\frac{1-\gamma}2}\Big ( \frac{T}{n}\Big )^{\frac{1-\gamma}2} \int_{r}^t    ds
\nonumber \\
&&+
Cm^{1+\beta} O_m(t)
\|b\|_{L^\infty(B_{\infty,\infty}^{-\beta})} 
\nu^{\frac{1-\gamma}{2}}
\Big ( \frac{T}{n}\Big )^{\frac{1-\gamma}2} 
(t-r) \exp \big ( \|b_m\|_{L^\infty(C^1)} t  \big ) 
\\
&&+
C \nu^{\frac{1-\gamma}{2}} m^{2+\beta-\gamma} \|f\|_{L^\infty(C^\gamma)} \|b\|_{L^\infty(B_{\infty,\infty}^{-\beta})} 
\exp \big (m^{1+\beta} \|b\|_{L^\infty(B_{\infty,\infty}^{-\beta})} t  \big ) 
\Big (\frac Tn \Big )^{\frac{(1-\gamma)}2+1 }
\int_{r}^{t}
d s ,
\nonumber
\end{eqnarray}
which goes to $	\int_{r}^{t} [  f(s,\cdot) ]_{\gamma,\nu^{1/2}(T/n)^{1/2}} ds $ when $n \to + \infty$ as soon as 
\begin{equation}\label{Condi_Lambda}
1\gg
	 \nu^{\frac{1-\gamma}{2}} (1+\frac Tn m^{1-\gamma} \|f\|_{L^\infty(C^\gamma)}) \|b\|_{L^\infty(B_{\infty,\infty}^{-\beta})} 
	\exp \big (C m^{2+\beta -\gamma} \|b\|_{L^\infty(B_{\infty,\infty}^{-\beta})} T\big ) 
	\Big (\frac Tn \Big )^{\frac{(1-\gamma)}2 }T,
\end{equation}
always true for  given $m <+ \infty$, $\nu>0$ and letting $n \to + \infty$.

By Chasles equality and by telescopic sum, the identity below readily comes
\begin{equation*}
		[	 u^{m,\nu}(t,\cdot)]_{\gamma,\nu^{1/2}(T/n)^{1/2}}
	\leq
	[ g  ]_{\gamma,\nu^{1/2}(T/n)^{1/2}} 
	+	 \Lambda_{0,t}
	.
\end{equation*}
The l.h.s. does not depend on $n$, then we are able to pass to the limit $n \to +\infty $,
\begin{equation}\label{ineq_Holder}
\lim_{n \to + \infty}[	 u^{m,\nu}(t,\cdot)]_{\gamma,\nu^{1/2}(T/n)^{1/2}}
\leq
\lim_{n \to + \infty}[	 g  ]_{\gamma,\nu^{1/2}(T/n)^{1/2}} 
+	\lim_{n \to + \infty} \int_{0}^{t} [ f(s,\cdot) ]_{\gamma,\nu^{1/2}(T/n)^{1/2}} ds 
.
\end{equation}

\subsection{Another control of uniform norm}
\label{sec_Linfty}

	We have directly by the Feynman-Kac formulation the uniform control, see for example \cite{kara:shre:91} (also used in the analysis performed in \cite{hono:21}), 
or from maximum principle for linear parabolic equation 
see e.g. \cite{lieb:96}, the control in $L^\infty$ of the parabolic solution $u^{m,\nu}$.

This maximum principle can also be derived from our strategy.

Indeed, by a similar way as for the H\"older control performed in the previous section,  by Duhamel formula 
\eqref{Duhamel_uk_FINAL}, after choosing $(\tau,\xi)=(t,x)$, we get 
\begin{equation*}
\|u^{m,\nu}_{k+1}  (t,\cdot)\|_{L^\infty}  \leq 
\|u^{m,\nu}_{k}  (\tau_k,\cdot)\|_{L^\infty} +  \int_{\tau_k}^t \|f(s,\cdot)\|_{L^\infty} ds 
+ C \|b_m\|_{L^\infty(C^1)} \|\nabla   u^{m,\nu}_{k}  \|_{L^\infty}\int_{\tau_k}^t (t-s)^{\frac{1}{2}}   ds .
\end{equation*}
Iterating this inequality,
\begin{equation*}
\|u^{m,\nu} \|_{L^\infty}  \leq 
\|g\|_{L^\infty} +  \int_{0}^T \|f(s,\cdot)\|_{L^\infty} ds 
+ C\frac {T^{\frac 32}} {n^{\frac{1}{2}}} \|b_m\|_{L^\infty(C^1)} \|\nabla   u^{m,\nu}  \|_{L^\infty} .
\end{equation*}
We obviously deduce when $n \to + \infty$, the following estimate,
\begin{equation}\label{ineq_Linfty}
\|	 u^{m,\nu}\|_{L^\infty } 
\leq 
\|	 g  \|_{L^\infty} +  \int_{0}^T  \|  f(s,\cdot) \|_{ L^\infty} ds .
\end{equation}
\textit{In fine}, the time decomposition trick allows to retrieve the powerful uniform estimate given by Feynman-Kac formula or usual maximum principle method.

\subsubsection{Comment on the full H\"older modulus}
\label{sec_comment_Holder_modulus}

In order to write the full H\"older modulus instead the \textit{limit} one, we have to restart the computations of Section \ref{sec_Green_semi_Holder} without supposing  \eqref{Condi_diagonal_Semi_group}, i.e. in the case
\begin{equation*}
	|x-x'| > \nu^{\frac 12}		T^{\frac 12}.
\end{equation*}
This situation prevent us to perform the same distinction \textit{diagonal}/\textit{off-diagonal} regimes performed in Section \ref{sec_remainders}.
In other words, the situation \eqref{def_A} cannot happen, and the only possible regime is \eqref{def_AC} which implies the choice $(\tau,\xi,\xi')=(t,x,x')$.
\\

	\textbf{Semi-group}
\\

We readily derive by change of variables:
\begin{eqnarray}
	&&		|\hat P^{\tau,\xi}  g_m(t,x) - \hat P^{\tau,\xi'}  g_m (t,x')| \Big |_{\tau=t,\xi=\xi'=x'}
	\nonumber \\
	&=& 	\Big | \int_{\R^d}  \tilde p (0,t,0,y)[ g_m(y-\theta_{ 0,t }(x))- g_m(y-\theta_{ 0,t }(x'))]	 dy  \Big | 
	\nonumber \\
	&\leq & \|\nabla \theta_{ 0,t }(\cdot) \|_{L^\infty} [g]_{\gamma} |x-x'|^\gamma.
\end{eqnarray}


\textbf{Green operator}
\\

We also get by change of variables:
\begin{eqnarray}
	&&	|\hat G^{\tau,\xi}  f_m(t,x) - \hat G^{\tau,\xi'}  f_ m (t,x')| \Big |_{\tau=t,\xi=\xi'=x'}
	\nonumber \\
	&=& 	\Big | \int_0^t \int_{\R^d}  \tilde  {p} (s,t,0,y)[ f_m(s,y-\theta_{ s,t }(x))- f_m(s,y-\theta_{ s,t }(x'))] ds \, dy  \Big | 
	\nonumber \\
	&\leq & 
	\int_0^t \|\nabla \theta_{ s,t }(\cdot) \|_{L^\infty} [f(s,\cdot)]_{\gamma} ds |x-x'|^\gamma.
\end{eqnarray}

In other word, we derive a suitable H\"older modulus control if $\theta_{ s,t }$ is Lipschitz continuous uniformly in $(m,\nu)$, which is guaranteed by assumption \A{A}, see Section \ref{sec_example_b}.

\subsubsection{On a $b$ depending on $t$}\label{sec_time_depend}

In this section, we develop why we can extend the hypothesis \A{A} to involved a time dependency:

We say that \A{A'} is in force if there is a function $\psi :[0,T] \times  \R^d \rightarrow \R^d$, $ C^1$ in time, such that 
\begin{equation*}
\forall (t,x) \in [0,T] \times \R^d, \ b(t,x) = (\nabla \psi (t,x))^{-1},
\end{equation*}
satisfying
\begin{equation*}
\forall (\tau,x) \in [0,T] \times \R, \	\frac{ \nabla \psi  (t,x)}{\nabla  \psi (t,\psi^{-1}(t,c\tau +\psi (t,x)))}<+ \infty,
\end{equation*}
and
\begin{equation*}
\psi_m ^{-1}(t,\tau +\psi_m (t,x)) \underset{m \to + \infty}{\longrightarrow} 		\psi^{-1}(t,\tau+\psi (t,x)) \text{ in } C^1(\R^d, \R^d),
\end{equation*}
with $\psi_m$ a mollified version of $\psi$.
\\

This new assumption is a by-product of \A{A}, because for any $t \in [\tau_k , \tau_{k+1}]$, $k \in \leftB 0, n-1 \rightB$, we derive
\begin{equation*}
\forall x \in \R^d, \ |b_m(t,x)-b_m(\tau_k,x)| \leq C (t-\tau_k) \|\partial_t b_m\|_{L^\infty}
\leq C\frac{T}{n} m^{-\beta} \|\partial_t b\|_{L^\infty}.
\end{equation*}
This above inequality means that we can assume that $b$ does depend on $t$, and at each time interval $t \in [\tau_k,\tau_{k+1}]$ we approximate $b(t,\cdot)$ by the distribution constant in time $b(\tau_k,\cdot)$.


\begin{cor}
The conclusions of	Theorem \ref{THEO_SCHAU_non_bounded_optimal} is still true replacing condition \A{A} by \A{A'} if $b \in C^{1}([0,T],\tilde B_{\infty,\infty}^{-\beta}(\R^d , \R^d))$, $\beta \in \R$.
\end{cor}

\mysection{Analysis of the parabolic equation}
\label{sec_parab}

\subsection{Some \textit{a priori} attempts}

\subsubsection{Peano's heuristic}
\label{sec_Peano}

The Peano counter-example yields a heuristic of the expected minimum regularity of $b$.
The threshold comes from a regularisation by noise argument, for other explanation see \cite{flan:11}, and \cite{chau:meno:17} for a degenerate case.
Similarly to \eqref{counter_example}, let us consider 
$b(x)={\rm sign} (x) (|x|\wedge R)^\gamma $, $\gamma \in \R$, and the associated flow,
\begin{equation}\label{Peano_example}
	\frac{dX_t}{dt} = {\rm sign}(X_t)|X_t|^\alpha, \ X_0=0.
\end{equation}
There are an infinite number of solutions written, for any $t^* \in [0,T]$, by :
\begin{equation}\label{Peano_integrate}
	X_t = \pm c_\alpha (t-t^*)^{\frac 1 {1-\alpha} } \mathds 1_{[t^*,T]}(t).
\end{equation}
The associated stochastic problem is
\begin{equation*}
	d \tilde  X_t= {\rm sign}(\tilde X_t)|\tilde X_t|^\alpha dt + \nu d W_t,
\end{equation*}
where $(W_t)_{t \geq 0}$ stands for a Brownian motion. 
Parabolic equation \eqref{Parabolic_Equation_Moll_Itro} is the determinist counterpart of this SDE, where the solution is given by a stochastic representation, the Feynman-Kac formula, see for instance \cite{kara:shre:91}.
\\

There is a critical time when the noise overwhelms the singular drift, see \cite{dela:flan:14}, after this time the SDE solution fluctuates around a solution of the ODE \eqref{Peano_example}.

As a consequence, we aim to compare the time scaling between the Brownian motion, i.e. $t^{\frac{1}{2}}$,
with a solution of the ODE given in \eqref{Peano_integrate}, i.e. $t^{\frac 1 {1-\alpha} } $.
That is to say, to take advantage of the regularisation by noise before the critical point, the condition is 
\begin{equation*}
	t^{\frac 1 {1-\alpha} } < t^{\frac{1}{2}},
\end{equation*}
which yields for a small time,
\begin{equation*}
	\frac 1 {1-\alpha} > \frac{1}{2} \ \Leftrightarrow \alpha > -1.
\end{equation*}
In other words, for $\tilde \gamma=\alpha =-1+\tilde \gamma$, $\tilde \gamma>0$, the expected minimum regularity of the drift is  $b(t,\cdot) \in B_{\infty,\infty}^{-1+\tilde \gamma}=C^{-1+\tilde \gamma}$.
\\

There are numerous articles dealing with SDE and the associated parabolic equation with irregular
drift, but in all of them there is a ``macro" distance with the above Peano heuristic.

The case $\tilde \gamma>1/3$ is dealt in \cite{dela:diel:16}  in dimension $1$ and 
\cite{cann:chou:18} for the multidimensional version; the authors thoroughly use rough path and para-control.

Another notion of solution of stochastic equation, called virtual, is introduced in \cite{flan:isso:russ:17}, where the constraint is $\tilde \gamma >1/2$.
Let us notice that under this constraint, $\tilde \gamma<1/2$, there is no hope to 
obtain strong solution of the SDE, see for instance the counter-examples presented in \cite{bass:chen:01}, \cite{barl:82}. 

\subsubsection{On the limitation of the para-product}

The last constraint, $\tilde \gamma> \frac 12$, appears naturally with \textit{a priori} computations with usual tools.
Indeed, from Duhamel formula the solution of the parabolic equation \eqref{Parabolic_Equation_Moll_Itro} writes
\begin{equation*}
	\label{DUHAMEL_PERTURB_F_1}
	u^{m,\nu}(t,x)=\tilde P g_m(x)+ \tilde G  f_m(t,x)+\int_0^t  \int_{\R^{d}} \tilde p(s,t,x,y)\langle b_m(y), \nabla u^{m,\nu} (s,y) \rangle dy \, ds,
\end{equation*}
where 
\begin{equation}\label{def_tilde_p}
	\tilde {p}  (s,t,x,y):= 
	\frac{1}{(4\pi \nu (t-s))^{\frac d 2} } 
	\exp \bigg ( -\frac {\left |x-y \right|^2}{4\nu(t-s)} \bigg ),
\end{equation}
stands for the the standard heat kernel, also 
\begin{equation}\label{def_tilde_G}
	\forall (t,x) \in (0,T]\times\R^{d}, \ \tilde  G   f_m(t,x):= \int_0^{t}  \int_{\R^{d}}  \tilde {p} (s,t, x,y)   f_m (s,y) \ dy \ ds,
\end{equation}
is the corresponding Green operator,
and 
\begin{equation}\label{def_tilde_P}
	\tilde P  g_m (t,x):=\int_{\R^{d}} \tilde p  (0,t,x,y)   g_m(y) dy,
\end{equation}
the associated semi-group.
\\

Let us suppose that $\nabla u^{m,\nu}(s,\cdot) \in C^{\delta}$, $\delta \geq 0$,  therefore from the above Duhamel's formula, we should have: 
\begin{equation*}
	x\mapsto  \nabla \int_t^T  \int_{\R^{d}} \tilde p(t,s,x,y)\langle  b_m(s,y), \nabla u^{m,\nu} (s,y) \rangle dy \, ds \in 	C^{\delta}. 
\end{equation*}
However, from the para-product result, derived by Bony's microlocal analysis \cite{bony:81}, see also \cite{gubi:imke:perk:15}, if $\nabla u^{m,\nu}(s,\cdot) \in C^{\delta}$ and $b_m(s,\cdot) \in C^{-\tilde \gamma}$, 
such that  $\delta-\tilde \gamma>0 \Leftrightarrow \tilde \gamma <\delta$ then  $\langle b_m(\cdot), \nabla u^{m,\nu} (s,\cdot) \rangle  \in C^{-\tilde \gamma}$.
Hence, with some common computations of the heat kernel, we obtain that
\begin{equation*}
	\nabla \int_t^T  \int_{\R^{d}} \tilde p(t,s,x,y)\langle b_m(y), \nabla u^{m,\nu} (s,y) \rangle dy\, ds \approx (T-t)^{\frac{1-\tilde \gamma}{2}}.
\end{equation*}
Also, thanks to the typical equivalence of the space-time with a parabolic scaling in the analysis of the parabolic bootstrap, see for instance \cite{chau:hono:meno:18}, where $(T-t)^2 \approx |x-x'|$, we deduce
\begin{equation*}
	x \mapsto \nabla \int_t^T  \int_{\R^{d}} \tilde p(t,s,x,y)\langle b_m(y), \nabla u^{m,\nu} (s,y) \rangle dy \, ds \in C^{1-\tilde \gamma}.
\end{equation*}
Then we readily derive that $ \tilde \gamma\leq  1-\delta$, namely $ \delta \leq 1 - \tilde \gamma $.
Combining with the para-product constraint $ \tilde \gamma < \delta$ yields $\tilde \gamma \leq 1-\tilde \gamma $ and so
$ \tilde \gamma < 1/2$ which is exactly the same regularity constraint  of  \cite{flan:isso:russ:17}.
\\

This heuristic shows up the difficulty to use 
para-product results in such a rough framework.
To handle with regularity $\tilde \gamma> 1/2$, we then have to capitalise on other techniques.
Specifically, we thoroughly exploit the fact that $u^{m,\nu}$ is solution of the parabolic equation \eqref{Parabolic_Equation_Moll_Itro}, which allows to consider cases out of the Bony's para-product scope.
\\

The previous analysis strongly relies on the vanishing result but seems to be incompatible with the parabolic equation.
We developed, here, a new technique  based on a time decomposition trick, detailed in Section \ref{sec_decoup_temps} below.

\subsection{The types of solution to the Cauchy problem}
\label{sec_def_type_solu}


\subsubsection{Parabolic equation}
\label{sec_def_parab}

For $0< \gamma < 1 $, and for a given viscosity $\nu >0$, we define some solutions of the parabolic equation.

\begin{defi}[mild solution]\label{DEFINITION_MILD_parab}
	We say that $u$ is a \textit{mild} solution in $L^\infty\big ([0,T]; C_b^{\gamma}(\R^d,\R)\big )$ of equation \eqref{transport_equation} if there is a sequence\footnote{We have such a sequence $(b_m)_{m \geq 0}$, if $b \in L^\infty([0,T];\tilde{ B}_{\infty,\infty}^{-\tilde \gamma}(\R^d,\R))$.} $(b_m)_{m \in \R_+}$ in $L^\infty  ( [0,T];C_b^\infty(\R^d,\R^d) )$ such that there is $\tilde \gamma \in \R$,
	\begin{equation}\label{converg_b_def_mild_parab}
	\forall \varepsilon>0, 
	\ 
	\lim_{m \to + \infty}\|b_m-b\|_{ L^\infty  ( [0,T];  B_{\infty,\infty}^{-\tilde \gamma-\varepsilon}(\R^d,\R^d) )}=0,
	\end{equation}
	for any $t \in [0,T]$,  there exists a sub-sequence of  $u^{m, \nu}(t,\cdot)_{m\in \R_+}$ lying in $
	C_b^{\gamma}(\R^d,\R) $ 
	converging, 
	for any compact subset $K \subset \R^d$, when 
	$m \to + \infty$,  towards $u (t,\cdot) \in  
	C_b^{\gamma}(K,\R) $ and satisfying, for any $m \in \R_+$,
	\begin{equation}\label{parabolic_moll_def_bis}
	\begin{cases}
	\partial_t u^{m, \nu}(t,x)+  \langle b_m(t,x),  \nabla u^{m, \nu} (t,x)\rangle -\nu \Delta u^{m, \nu}(t,x)=f_m(t,x),\ (t,x)\in (0,T]\times \R^{d},\\
	u^{m, \nu}(0,x)=g_m(x),\ x\in \R^{d},
	\end{cases}
	\end{equation}
	where $(f_m,g_m) \underset{m \to + \infty}\longrightarrow (f,g)$ in $L^\infty([0,T]; C_b^{\gamma}(\R^d,\R)) \times C_b^{\gamma}(\R^d,\R)$.
\end{defi}

Let us now, recall the notion of usual \textit{weak} solution.
\begin{defi}[weak solution]\label{DEFINITION_WEAK_parab}
	A function $u$ is a weak solution in $L^\infty\big ([0,T]; C_b^{\gamma}(\R^d,\R)\big )$ of equation \eqref{transport_equation} if $u$ is a \textit{mild} solution such that for any function $\varphi \in C_0^\infty([0,T]\times \R^d,\R)$:
	\begin{eqnarray}\label{KOLMO_weak}
	\int_{\R^d} \Big \{   \varphi(t,y) u(t,y) 
	+ \int_0^t  
	\big \{-\partial_t\varphi(s,y) u (s,y)+  \langle b(s,y),  \nabla u(s,y)\rangle \varphi(s,y) + \nu u(s,y) \Delta \varphi(s,y)
	\big \} ds \Big \} dy  
	\nonumber \\
	= \int_{\R^d}   \varphi(0,y) g(y) dy +\int_{\R^d} \int_0^t    \varphi(s,y) f(s,y) ds \, dy .
\nonumber \\
	\end{eqnarray}
\end{defi}
Again, this formulation allows us to give a usual distributional meaning of the product 
$ \langle b,  \nabla u \rangle $.

\subsection{Main result on the parabolic equation}
\label{sec_b_non_bounded}


When $b$ lies in H\"older-Besov space, we succeed in obtaining the same regularity of the solution as for $f$ and  $g$.
The type of solution strongly depends on the regularity of $b$.

\begin{THM}[Rough parabolic equation in H\"older spaces] \label{THEO_SCHAU_non_bounded_para}
	For  
	$\tilde \gamma \in \R^*$ and $0 < \gamma <1$, be given.
	Let $f\in L^\infty([0,T];  C_b^{\gamma}(\R^d,\R))$ and $ g \in C_b^{\gamma}(\R^d,\R)$.
	For a distribution $ b \in C_b^1([0,T], \tilde B_{\infty,\infty}^{\tilde \gamma} (\R^d,\R^d))$, 
	then, for $u^{m,\nu}$ defined in \eqref{parabolic_moll_def_bis},  
	we have 
	\begin{eqnarray}
	\label{ineq_THEO_SCHAU_para}
\lim_{n \to + \infty}	\sup_{t \in [0,T]}[ u^{(m,\nu)}(t,\cdot)]_{\gamma,(\nu T/n)^{1/2}}
	&\leq&
\lim_{n \to + \infty}	\int_0^T [ f(s,\cdot)]_{\gamma,(\nu T/n)^{1/2}}+  [ g]_{\gamma,(\nu T/n)^{1/2}}
		\nonumber \\
		\|  u \|_{L^\infty} &\leq&  \int_0^T\| f(s,\cdot) \|_{L^\infty}+ \|  g \|_{L^\infty}  .
	\end{eqnarray}
	Moreover, if \A{A'} is in force, then there is a \textit{mild} solution $u \in L^\infty([0,T];C_b^{\gamma}(\R^{d},\R)) $ of \eqref{transport_equation}; also
	if $\tilde \gamma <\gamma$ and $\nabla \cdot b=0$ then the solution $u$ is also a \textit{weak}
	solution. 
	
\end{THM}

We do not consider the positive regular case, i.e. $\tilde \gamma <0$, whose control is well-known, the Schauder estimates are even in force, see \cite{frie:64}.

\begin{remark}
	The control \eqref{ineq_THEO_SCHAU_para} exactly matches with the \textit{limit} H\"older estimates of the solution of the heat equation, independently of the dimension, and above all of  $b$. As  consequence, the mild solution has no condition on $\tilde \gamma$, there is to say we define a solution beyond the Peano condition, $\tilde \gamma >-1$, developed in Section \ref{sec_Peano}. 
	
	In the incompressible case, $\nabla \cdot b$, the condition
	$\tilde \gamma <\gamma<1$
	corresponds ``almost" to the Peano's heuristic\footnote{The true condition is $\tilde \gamma<1$, but this is the case, here, as $\gamma$ which can be arbitrarily close to $1$.} in Section \ref{sec_Peano}.
	However, we fail to obtain uniqueness of the solution out of the usual H\"older continuous case (i.e. $\tilde \gamma <0$) handled in \cite{frie:64}, for more information see Remark \ref{Rem_uniq_negativ} further.
	
	The case $\tilde \gamma=0$ can be considered, replacing the condition $ b \in L^\infty([0,T], \tilde B_{\infty,\infty}^{0} (\R^d,\R^d))$ by $ b \in L^\infty([0,T], L^\infty (\R^d,\R^d))$.
	Considering the Besov space $B_{\infty,\infty}^{0}$ would yield some refinements involving some logarithm corrections.
	
	Again, the case $\tilde \gamma<0$ is the usual framework, see for instance \cite{frie:64}, \cite{kryl:96}.
\end{remark}

The proof is a direct consequence of the analysis performed in Section \ref{sec_Proof_transport}, letting $m$ going to $+\infty$.
Nevertheless, there is no vanishing viscosity which would restore uniqueness.


%

%


\mysection{Inviscid Burgers' equation}
\label{sec_Buregers}

The controls \eqref{Schauder_ineq} of the vanishing viscous solution of the PDE \eqref{transport_equation} being independent on the first order term $b$, we can expect to obtain some fixed-point argument to consider that $b$ being the solution $u$ itself in dimension $1$\footnote{It is possible to adapt the analysis for a more general dimension $d \geq 1$ by a reformulation of the product $\langle u(t,x),  \nabla u(t,x)\rangle$.}.

Such a Cauchy problem thus defined is called the inviscid Burgers' equation,
\begin{eqnarray}
\label{Burgers_equation}
\begin{cases}
\partial_t u(t,x)+ u(t,x) \partial_x u(t,x) =f(t,x),\ (t,x)\in \R_+ \times \R,\\
u(0,x)=g(x),\ x\in \R.
\end{cases}
\end{eqnarray}
For more information about the Burgers' equation and the corresponding turbulence phenomenon, we refer to the recent book \cite{bori:kuks:21}.

\subsection{Statement about the Inviscid Burgers' equation}

We obtain a different notion of uniqueness for this equation because the convergence of the mollified first order term, being the solution itself, is more intricate comparing with the transport equation case.

\begin{defi}[Turbulent uniqueness]\label{DEFINITION_Unique_turbulent}
	There is a turbulent unique solution if there are two solutions $u^{m,\nu}$ and $u^{m,\bar \nu}$ of 
	\begin{equation}\label{eq1_turbulenet_uniq}
		\begin{cases}
			\partial_t u^{m,\nu}(t,x)+   u_m^{m,\nu} (t,x) \partial_{x} u^{m,\nu} (t,x)-\nu \partial_{xx}^2 u^{m,\nu}(t,x)=f_m(t,x),\ (t,x)\in (0,T]\times \R ,\\
			u^{m,\nu}(0,x)=g_m(x),\ x\in \R,
		\end{cases}
	\end{equation}
	and respectively
	\begin{equation}\label{eq2_turbulenet_uniq}
		\begin{cases}
			\partial_t u^{m,\bar \nu}(t,x)+  u_m^{m,\bar \nu} (t,x)   \partial_{x} u^{m,\bar \nu} (t,x) -\bar \nu \partial_{xx}^2  u^{m,\bar \nu}(t,x)=f_m(t,x),\ (t,x)\in (0,T]\times \R,\\
			u^{m,\bar \nu}(0,x)=g_m(x),\ x\in \R,
		\end{cases}
	\end{equation}
	for $\nu,\bar \nu>0$, and for any $(t,x) \in [0,T] \times \R$,
	where $u_{m}^{m, \nu}$ and $u_{m}^{m,\bar \nu}$ stand respectively for a mollified version of $u^{m, \nu}$ and $u^{m,\bar \nu}$,
	such that, for any $(t,x) \in [0,T] \times \R$, 
	\begin{equation*}
		\forall \varepsilon>0, 
		\ 
		\lim_{n \to + \infty}\|u_n^{m, \nu}  - u^{m, \nu} \|_{ L^\infty  ( [0,T];  C_b^{\gamma-\varepsilon})}
		= 	\lim_{n \to + \infty}\|u_n^{m,\bar \nu}  - u^{m,\bar \nu} \|_{ L^\infty  ( [0,T];  C_b^{\gamma-\varepsilon})}
		=0,
	\end{equation*}
and,
	\begin{equation}\label{ineq_u_m_nu_moll}
	|u_m^{m, \nu}(t,x)  - u_m^{m,\bar \nu} (t,x) | \leq 	|u^{m,\bar \nu}(t,x)  - u^{m,\bar \nu} (t,x) |,
\end{equation}
	converging, up to sub-sequence selection, towards two H\"older continuous solutions $u$, $\bar u$, when $(m,\nu,\bar \nu ) \to (+ \infty,0,0)$ then $u= \bar  u$. 
\end{defi}
The last condition \eqref{ineq_u_m_nu_moll} simply means that the mollification procedure behaves like a convolution with a smooth kernel (like performed in \eqref{def_b_epsilon}).
 \\
 
 Let us insist that the difference with the uniqueness introduced in Definition \ref{DEFINITION_Unique} is that the considered regularisation procedure for the first order terms $u^{m,\nu}$ and $u^{m,\bar \nu}$ is the same in equations \eqref{eq1_turbulenet_uniq} and \eqref{eq2_turbulenet_uniq}; this explains the definition 
 \eqref{ineq_u_m_nu_moll}.
 
  In particular, uniqueness  introduced in Definition \ref{DEFINITION_Unique} (for the transport equation) yields turbulent uniqueness.  
  We detail in Remark \ref{rem_non_uniq_Burgers} why we have to handle with such a turbulent uniqueness or the viscous uniqueness, defined below, instead of the classic uniqueness\footnote{Definition \ref{DEFINITION_Unique}.} for the inviscid Burgers' equation.
 
 \begin{defi}[Viscous uniqueness]\label{DEFINITION_Unique_viscous}
 	There is a viscous unique solution if there are two solutions $u^{m,\nu}$ and $u^{m,\bar \nu}$ of 
 	\begin{equation}\label{eq1_viscous_uniq}
 		\begin{cases}
 			\partial_t u^{m,\nu}(t,x)+   u_m^{m,\nu} (t,x) \partial_{x} u^{m,\nu} (t,x)-\nu \partial_{xx}^2 u^{m,\nu}(t,x)=f_m(t,x),\ (t,x)\in (0,T]\times \R,\\
 			u^{m,\nu}(0,x)=g_m(x),\ x\in \R,
 		\end{cases}
 	\end{equation}
 	and respectively
 	\begin{equation}\label{eq2_viscous_uniq}
 		\begin{cases}
 			\partial_t \bar  u^{m, \nu}(t,x)+ \bar  u_m^{m, \nu} (t,x)   \partial_{x} \bar u^{m, \nu} (t,x) - \nu \partial_{xx}^2  \bar u^{m, \nu}(t,x)=f_m(t,x),\ (t,x)\in (0,T]\times \R,\\
 			\bar u^{m, \nu}(0,x)=g_m(x),\ x\in \R,
 		\end{cases}
 	\end{equation}
 	for $\nu>0$,
 	with
	\begin{equation*}
	\forall \varepsilon>0, 
	\ 
	\lim_{n \to + \infty}\|u_n^{m, \nu}  - u^{m, \nu} \|_{ L^\infty  ( [0,T];  C_b^{\gamma-\varepsilon})}
	= 	\lim_{n \to + \infty}\|\bar u_n^{m, \nu}  - \bar u^{m, \nu} \|_{ L^\infty  ( [0,T];  C_b^{\gamma-\varepsilon})}
	=0,
\end{equation*}
 	converging, up to sub-sequence selection, towards two H\"older continuous solutions $u$, $\bar u$, when $(m,\nu,\bar \nu ) \to (+ \infty,0,0)$ then $u= \bar  u$. 
 \end{defi}
 Importantly, the above equations have the same viscosity $\nu>0$, but the mollification procedure of the first order coefficient may be different.
 This uniqueness definition is, somehow, the complementary of the \textit{turbulent} uniqueness in the usual uniqueness introduced in Definition \ref{DEFINITION_Unique}.

Replacing $b$ by $u$ in the different definitions of solution in Section \ref{sec_def_type_solu}, we establish the last result of this paper.
\begin{THM}[Uniqueness of H\"older solution of the inviscid Burgers' equation] \label{THEO_SCHAU_Burgers}
	For  $\gamma\in (0,1)$  be given.
For all
$f\in L^\infty([0,T];  C_b^{\gamma}(\R,\R))$ and $ g \in C^{\gamma}_b(\R,\R)$, 
if there is 
a
mild vanishing viscosity solution $u \in L^\infty([0,T];C_b^{\gamma}(\R,\R))  $ 
then $u$ satisfies
\begin{eqnarray}\label{Schauder_ineq_Burgers}
\lim_{n \to + \infty}	\sup_{t \in [0,T]}[ u^{m,\nu}(t,\cdot)]_{\gamma,(\nu T/n)^{1/2}}
&\leq& 
\lim_{n \to + \infty} \int_0^T [f(s,\cdot)]_{\gamma,(\nu T/n)^{1/2}}+ [g]_{\gamma,(\nu T/n)^{1/2}},
\nonumber \\
\|  u^{m,\nu} \|_{L^\infty} &\leq&  \int_0^T\| f(s,\cdot) \|_{L^\infty}+ \|  g \|_{L^\infty}  .
\end{eqnarray}
\begin{trivlist}
\item[i)
\textbf{Good regularity.}] If $\gamma>\frac{1}{2}$  then the considered mild vanishing viscosity solution is  also a \textit{mild-weak} and  a \textit{weak} solution, and if 
	\begin{equation}\label{CONDINU_Holder_bis_Burgers}
	\nu \|\nabla ^2g_m\|_{L^\infty} 
	+ 
	\nu ^{\frac{\gamma}{2}}	m^{1- \gamma}	
	\ll 1,
\end{equation}
then
$\partial_t u (t,\cdot) \in B_{\infty,\infty}^{-1+\gamma}(\R^d,\R)$, $\forall t \in  (0,T]$.
\item[ii) 
\textbf{Fast vanishing viscosity.}] If, for a constant $C>0$ big enough,
\begin{equation}\label{CONDINU_UNIQ_Burgers}
		\exp \bigg ( C T    ( T \| f_m\|_{L^\infty(C^1)}+\|g_m\|_{C^1}  ) \exp \Big (  C m t \big ( T \|f\|_{L^\infty}+ \|g \|_{L^\infty}\big ) \Big ) \bigg )	
		\bigg ( (1+m
		)\nu^{\frac{\gamma}{2}} + \nu \|\nabla ^2g_m\|_{L^\infty}  \bigg )
		\ll 1,
	\end{equation}
then the solution is \textit{turbulent} unique.
\item[iii) 
\textbf{Slow vanishing viscosity.}] If $\|\partial_{x}  f\|_{L^\infty}+\|\partial_{x} g\|_{L^\infty} <+ \infty$ and, for a constant $C>0$ big enough,
\begin{equation}\label{CONDINU_UNIQ_Burgers_Viscous}
	\exp \bigg ( C T   \Big ( T \|\partial_{x}  f\|_{L^\infty}+\|\partial_{x} g\|_{L^\infty} \Big ) 
	\exp \Big (  C T \nu^{-1} 
	(T\|f\|_{L^\infty}+ \|g\|_{L^\infty})^2 \Big ) \bigg ) 
	m^{-\gamma} 	
		\ll 1,
	\end{equation}
	then the solution is \textit{viscous} unique.
\end{trivlist}
%

\end{THM}
It is well-known that there is a exploding time $T^*$, such that $u $ is ``smooth" before $T^*$. We could imagine some structure conditions on the associated flow, or like in Section \ref{sec_example_b}.

Let us remark that the condition ii) in Theorem \ref{THEO_SCHAU_non_bounded_optimal} is satisfied if the considered \textit{a priori} regularity of the solution is strong enough (\textit{a priori} not the condition i) in Theorem \ref{THEO_SCHAU_non_bounded_optimal}, as $u$ is not incompressible, except if $u$ is also solution of the Euler equation), as we have $-\gamma<-1+\gamma$ $ \Longleftrightarrow$ $ \gamma> \frac{1}{2}$.

%

\begin{remark}
	Without considering the regularity condition $ \gamma> \frac{1}{2}$ to get a weak solution, we may have pathologic situation.
	Specifically, let us consider  the steady-state non-linear problem
	\begin{equation}\label{eq_uu'_sign}
		u(x)u'(x)=\frac{1}{2}{\rm sgn }(x),
	\end{equation}
	whose $x \mapsto \sqrt{|x|}$ is solution\footnote{Also the function $x \mapsto - \sqrt{|x|}$, but this non-uniqueness should be related with some ergodic properties.} which is as expected $\frac{1}{2}$-H\"older continuous.
	In other words, if $\gamma=\frac{1}{2}$, we can explicitly find a $\gamma$-H\"older steady-state solution of the inviscid Burgers' equation with source function being in $B^0_{\infty,\infty}(\R,\R)$ but $C^\infty_b$ almost everywhere and being the limit of a $C^\infty_b$  function, e.g. $\tanh$.
%
\end{remark}
\begin{remark}
With our current approach, we fail to provide any Lipschitz control of a solution of the inviscid Burgers' equation \eqref{Burgers_equation} for the same reason as for the transport equation \eqref{transport_equation}. This is not surprising by the well-known blowing-up of the gradient of a solution of the inviscid Burgers' equation \eqref{Burgers_equation}.
%

\end{remark}

\begin{remark}\label{rem_non_uniq_Burgers}
	The conditions \eqref{CONDINU_UNIQ_Burgers} and \eqref{CONDINU_UNIQ_Burgers_Viscous} are not compatible, that is to say there is no uniqueness in the sense of Definition \ref{DEFINITION_Unique}, being the combination of \textit{turbulent} and \textit{viscous} uniqueness.
	Condition \eqref{CONDINU_UNIQ_Burgers} to get turbulent uniqueness means that $\nu$ goes to $0$ exponentially faster than $m$ goes to $+ \infty$; 
	whereas condition \eqref{CONDINU_UNIQ_Burgers_Viscous} to get viscous uniqueness implies that $m$ goes to $+ \infty$ exponentially faster than $\nu$ goes to $0$. 
	
	We insist that \textit{turbulent} uniqueness or \textit{viscous} uniqueness of a such smooth solution of the inviscid Burgers' equation is not a contradiction with the usual counter-example built by characteristics, because we only consider solution selected by a vanishing viscosity approximation for a given mollification procedure.
	Somehow, this selection principle allows to avoid the blow-up time appearing in the characteristic building for a given vanishing viscous path.
	
	Actually, from the \textit{mild vanishing viscous} solution, we see that for any $t \in [0,T]$, the solution $u(t,\cdot)$ given by the limit of a sub-sequence of $u^{m,\nu}(t,\cdot)$
	depends on the mollification choice, moreover the sub-sequence choice also depends on the current time $t$.
	In other words, $u(t,\cdot)$ seems to avoid the time of blowing-up thanks to a different choice of sub-sequence at each current time.

\end{remark}


\subsection{Proof of Theorem \ref{THEO_SCHAU_Burgers}}

To establish this result, we consider the mollified version of Burgers' equation,
for all $m \in \R_+$ and $\nu >0$,
\begin{equation}\label{def_Burgers_viscous}
\begin{cases}
\partial_t u^{m,\nu}(t,x)+  u_{m}^{m,\nu}(t,x)\partial_x u^{m,\nu}(t,x) -\nu \partial_{xx}^2 u^{m,\nu}(t,x)=f_m(t,x),\ (t,x)\in [0,T)\times \R,\\
u_m(0,x)=g_m(x),\ x\in \R,
\end{cases}
\end{equation}
where $u_{m}^{m,\nu}$ stands for a mollified version of $u^{m,\nu}$, such that 
	\begin{equation*}
	\forall \varepsilon>0, 
	\ 
	\lim_{n \to + \infty}\|u_n^{m, \nu}  - u^{m, \nu} \|_{ L^\infty  ( [0,T];  C_b^{\gamma-\varepsilon})}
	=0,
\end{equation*}
%
It is direct from Theorem 3 in \cite{hono:21} that there is a smooth solution of \eqref{def_Burgers_viscous}.\\

We then perform the same computations as for the transport equation, where $-\tilde \gamma=\gamma$ and we change in the viscous condition $\|b\|_{L^\infty(B_{\infty,\infty}^{-\tilde \gamma })}$ by an upper-bound of $\|u\|_{L^\infty}$, namely by $T \|f\|_{L^\infty(C^\gamma)}+ [g]_\gamma$ given by Feynman-Kac formula or from the time decomposition trick, see Section \ref{sec_Linfty}.
Finally, we are able to take the limit, thanks to a compact argument, of a suitable sub-sequence yields the result, all the computations readily derives from the analysis of Theorems \ref{THEO_SCHAU_non_bounded}, \ref{THEO_SCHAU_non_bounded_para} and  \ref{THEO_SCHAU_non_bounded_optimal}.

\subsubsection{Turbulent uniqueness}

To establish  uniqueness, let us consider a regularised Burgers' equation with another viscosity $\bar \nu$, 
\begin{equation}
	\label{Burgers_moll_bar_u}
	\begin{cases}
		\partial_t  u ^{m,\bar \nu}(t,x)+   u _{m}^{m,\bar \nu}(t,x)\partial_x  u ^{m,\bar \nu}(t,x) -\bar \nu \partial_{xx}^2  u ^{m,\bar \nu}(t,x)=f_m(t,x),\ (t,x)\in [0,T)\times \R,\\
		 u ^{m,\bar \nu}(0,x)=g_m(x),\ x\in \R.
	\end{cases}
\end{equation}
We highlight that $u _{m}^{m,\bar \nu}$ is also 
a regularisation version of $u _{m}^{m,\bar \nu}$  such that \eqref{ineq_u_m_nu_moll} is in force and, 
	\begin{equation*}
	\forall \varepsilon>0, 
	\ 
	\lim_{n \to + \infty}\|\bar u_n^{m, \nu}  - \bar u^{m, \nu} \|_{ L^\infty  ( [0,T];  C_b^{\gamma-\varepsilon})}
	=0,
\end{equation*}
Like for $u ^{m,\nu} $, we suppose that
 there is a constant $ C>0$ such that for any $(m,\bar \nu) \in \R_+^2$,
\begin{equation*}
		\| u ^{m,\bar \nu} \|_{L^\infty(C^\gamma_b)} \leq  C  (T \|f\|_{L^\infty(C^\gamma)}+ [g]_\gamma  ), 
\end{equation*}
under some specific asymptotic conditions on $(m,\bar \nu)$.
%
We still write $U_{m}:=u^{m,\nu}- u ^{m,\bar \nu}$ which is solution of
\begin{equation}
	\label{Burgers_moll_u1_u2}
		\partial_t U_{m}(t,x)+ u_m^{m,\nu}(t,x)   \partial_x U_{m}(t,x)  -\bar \nu \Delta U_{m}(t,x)=
		 -  [u_m^{m,\nu}- u_m ^{m,\bar \nu}](t,x)  \partial_x  u ^{m,\bar \nu}(t,x) +(\nu-\bar \nu ) \Delta u^{m, \nu} (t,x) ,
\end{equation}
with $	U_{m}(0,x)=0$.
%
%

Adapting inequality \eqref{ineq_Umn_1} yields
\begin{equation*}
	\|U_{m}(t,\cdot)\|_{L^\infty}
	\leq 
	\int_0^t  \| [u_m^{m,\nu}- u_m ^{m,\bar \nu}]
	(s,\cdot) \partial_{x}  u _{m}^{m,\bar \nu}  (s,\cdot)\|_{L^\infty}ds+t |\nu - \bar \nu | \|\Delta u^{m, \nu} \|_{L^\infty}.
\end{equation*}
We then get by triangular inequality and by \eqref{ineq_u_m_nu_moll}, we can suppose w.l.o.g. that $\bar \nu \leq \nu$ (if not we switch the roles of $\nu $ and $\bar \nu$),
\begin{equation*}
	\|U_{m}(t,\cdot)\|_{L^\infty}
\leq 
	\bar C 
	\int_0^t \|U_{m}(s,\cdot)\|_{L^\infty}  
	 \| \partial_{x} u _{m}^{m,\bar \nu}\|_{L^\infty} ds
	+ t 
	\nu \| \Delta u^{m,\nu }\|_{L^\infty} 
		.
\end{equation*}
Recalling the Hessian estimate of Lemma \ref{lemma_apriori},
\begin{eqnarray*}
	&&	|\nabla^2 u^{m,\nu}(t,x) |
	\nonumber \\
	&\leq& 
	\Big ( C \nu ^{\frac{\gamma}{2}-1} t^{\frac \gamma 2}  \| f\|_{L^\infty(C^\gamma)}+\|\nabla ^2g_m\|_{L^\infty} \Big ) 
	+ C
	\|u_m^{m,\nu}\|_{L^\infty(C^1)}
	\Big ( t \|  f_m\|_{L^\infty(C^\gamma)}+\| g_m\|_{C^\gamma} \Big ) 
	\nu ^{\frac{\gamma}{2}-1} 
	t^{\frac{\gamma}{2}} 
	\nonumber \\
		&\leq& 
	\Big ( C \nu ^{\frac{\gamma}{2}-1} t^{\frac \gamma 2}  \| f\|_{L^\infty(C^\gamma)}+\|\nabla ^2g_m\|_{L^\infty} \Big ) 
+ C
m
	\|u^{m,\nu}\|_{L^\infty}
	\Big ( t \|  f_m\|_{L^\infty(C^\gamma)}+\| g_m\|_{C^\gamma} \Big ) 
	\nu ^{\frac{\gamma}{2}-1} 
	t^{\frac{\gamma}{2}} 
	.
\end{eqnarray*}
The \textit{a priori} control of $\|u^{m,\nu}\|_{L^\infty(C^\gamma)}$ obtained in \eqref{ineq_Linfty}
yields
\begin{eqnarray*}
		|\nabla^2 u^{m,\nu}(t,x) |
	&\leq &
	\Big ( C \nu ^{\frac{\gamma}{2}-1} t^{\frac \gamma 2}  \| f\|_{L^\infty(C^\gamma)}+\|\nabla ^2g_m\|_{L^\infty} \Big ) 
\nonumber \\
&&	+ C m
\Big (T \|f\|_{L^\infty}
+ \|g\|_{L^\infty}
\Big )\Big ( t \|  f\|_{L^\infty(C^\gamma)}+\| g\|_{C^\gamma} \Big ) 
	\nu ^{\frac{\gamma}{2}-1} 
	t^{\frac{\gamma}{2}} 
	.
\end{eqnarray*}
Hence, 
\begin{eqnarray*}
		\|U_{m}(t,\cdot)\|_{L^\infty}
	&\leq & 
	C 
\int_0^t  \|U_{m}(s,\cdot)\|_{L^\infty}    \| \partial_{x} u _{m}^{m,\bar \nu}\|_{L^\infty} ds
	+
	T \nu\bigg (  C \nu ^{\frac{\gamma}{2}-1} t^{\frac \gamma 2}  \| f\|_{L^\infty(C^\gamma)}+\|\nabla ^2g_m\|_{L^\infty} 
	\nonumber \\
	&& 
	+C m
	\Big (T \|f\|_{L^\infty}
+ \|g\|_{L^\infty}
\Big )\Big ( t \|  f\|_{L^\infty(C^\gamma)}+\| g\|_{C^\gamma} \Big ) 
	\nu ^{\frac{\gamma}{2}-1} 
	t^{\frac{\gamma}{2}} \bigg )
	.
\end{eqnarray*}
By Gr\"onwall's inequality, we derive
\begin{eqnarray*}
		\|U_{m}(t,\cdot)\|_{L^\infty}
	&\leq & 
e^{ C T  \| \partial_{x} u _{m}^{m,\bar \nu}\|_{L^\infty}  }	T \bigg (  C \nu ^{\frac{\gamma}{2}} t^{\frac \gamma 2}  \| f\|_{L^\infty(C^\gamma)}+ \nu \|\nabla ^2g_m\|_{L^\infty} 
\nonumber \\
&& 
+C m
\Big (T \|f\|_{L^\infty}
+ \|g\|_{L^\infty}
\Big )\Big ( t \|  f\|_{L^\infty(C^\gamma)}+\| g\|_{C^\gamma} \Big ) 
\nu ^{\frac{\gamma}{2}} 
t^{\frac{\gamma}{2}} \bigg )
.
\end{eqnarray*}
Recalling 
the gradient estimate from Lemma \ref{lemma_apriori}
\begin{eqnarray*}
		\|\nabla u^{m,\nu}(t,\cdot)\|_{L^\infty} 
		&\leq& \Big ( T \| f_m\|_{L^\infty(C^1)}+\|g_m\|_{C^1} \Big ) \exp \Big ( C m t \|u^{m,\nu}\|_{L^\infty} \Big )
		\nonumber \\
		&\leq& \Big ( T \| f_m\|_{L^\infty(C^1)}+\|g_m\|_{C^1} \Big ) \exp \Big ( C m t \big ( T \|f\|_{L^\infty}+ \|g \|_{L^\infty}\big )\Big ).
\end{eqnarray*}
Hence, we get
\begin{eqnarray}\label{ineq_Umn_Burgers}
		&&
	\|U_{m}(t,\cdot)\|_{L^\infty}
		 \\
	&\leq & 
 T 	\exp \bigg ( C T    ( T \| f_m\|_{L^\infty(C^1)}+\|g_m\|_{C^1}  ) \exp \Big (  C m t \big ( T \|f\|_{L^\infty}+ \|g \|_{L^\infty}\big ) \Big ) \bigg )	
	\nonumber \\
	&& \bigg (  C \nu ^{\frac{\gamma}{2}} t^{\frac \gamma 2}  \| f\|_{L^\infty(C^\gamma)}+ \nu \|\nabla ^2g_m\|_{L^\infty} 
	+C m
	\Big (T \|f\|_{L^\infty}
+ \|g\|_{L^\infty}
\Big )\Big ( t \|  f\|_{L^\infty(C^\gamma)}+\| g\|_{C^\gamma} \Big ) 
	\nu ^{\frac{\gamma}{2}} 
	t^{\frac{\gamma}{2}} \bigg )
	,\nonumber
\end{eqnarray}
which goes to $0$ if
\begin{equation*}
		\exp \bigg ( C T    ( T \| f_m\|_{L^\infty(C^1)}+\|g_m\|_{C^1}  ) \exp \Big (  C m t \big ( T \|f\|_{L^\infty}+ \|g \|_{L^\infty}\big ) \Big ) \bigg )	
	\bigg ( (1+m
	)\nu^{\frac{\gamma}{2}} + \nu \|\nabla ^2g_m\|_{L^\infty}  \bigg )
\ll 1.
\end{equation*}
Turbulent uniqueness of \textit{vanishing viscous} solution is then established.


\subsubsection{Viscous uniqueness}

Let us consider another regularised Burgers' equation with different mollification procedure but with the same viscosity $\nu >0$,
\begin{equation}
	\label{Burgers_moll_bar_u_different_moll}
	\begin{cases}
		\partial_t \bar u^{m, \nu}(t,x)+  \bar u_{m}^{m, \nu}(t,x)\partial_x \bar u^{m, \nu}(t,x) - \nu \partial_{xx}^2 \bar u^{m, \nu}(t,x)=f_m(t,x),\ (t,x)\in [0,T)\times \R,\\
		\bar u^{m, \nu}(0,x)=g_m(x),\ x\in \R,
	\end{cases}
\end{equation}
with $ \nu >0$, and where $ \bar u_{m}^{m, \nu}$ is a regularisation of $\bar u^{m, \nu}$ (not necessarily a defined by a convolution)
such that 
\begin{equation*}
	\sup_{(m,\nu) \in \R_+^2}	\|\bar u^{m, \nu} \|_{L^\infty(C^\gamma_b)} \leq C ( T \|f\|_{L^\infty(C^\gamma)}+ [g]_\gamma) <+\infty,
\end{equation*}
and for any $ 0 < \varepsilon <1$, 
\begin{equation}\label{cond_lim_u_u_n_0}
	\lim_{n \to + \infty} \|\bar u_{n}^{m,  \nu}-\bar u^{m, \nu}\|_{L^\infty}= \lim_{n \to + \infty} \|\bar u_{n}^{m, \nu}-\bar u^{m,  \nu}\|_{L^\infty(C_b^{\gamma-\varepsilon})}=0.
\end{equation}
We still write $U_{m}:=u^{m,\nu}-\bar u^{m, \nu}$ which is solution of
\begin{equation}
	\label{Burgers_moll_u1_u2}
	\begin{cases}
		\partial_t U_{m}(t,x)+ u_m^{m,\nu}(t,x)   \partial_x U_{m}(t,x)  - \nu \Delta U_{m}(t,x)=
		-  [u_{m}^{m,\nu}-\bar u_{m}^{m, \nu}](t,x)  \partial_x \bar u^{m, \nu}(t,x) 
		,\\
		U_{m}(0,x)=0.
	\end{cases}
\end{equation}
We again adapt inequality \eqref{ineq_Umn_1} yields
\begin{equation*}
	\|U_{m}(t,\cdot)\|_{L^\infty}
	\leq 
	\int_0^t  \| (u_{m}^{m,\nu}-\bar u_{m}^{m, \nu})  \partial_{x} \bar u_{m}^{m, \nu}  (s,\cdot)\|_{L^\infty}ds
	.
\end{equation*}
Let us denotes 
\begin{equation}\label{def_bar_u_mn_nu_m}
	\bar u_{m,m}^{m, \nu}(t,x):=\int_{\R} \rho_m (x-y) \bar u^{m, \nu} (t,y) dy,
\end{equation}
such that
\begin{equation*}
	\|\bar u_{m,m}^{m, \nu}- \bar u_{m}^{m, \nu} \|_{L^\infty} \leq C m^{-\gamma} \| \bar u^{m, \nu}\|_{L^\infty(C^\gamma)},
\end{equation*}
and
\begin{equation*}
\|u_{m}^{m, \nu}-\bar u_{m,m}^{m, \nu}\|_{L^\infty} \leq 	\|u^{m,\nu}-\bar u^{m, \nu}\|_{L^\infty} = \|U_{m}(t,\cdot)\|_{L^\infty} .
\end{equation*}
We then get by triangular inequality, 
\begin{eqnarray*}
		\|U_{m}(t,\cdot)\|_{L^\infty}
	&\leq & 
	\int_0^t
	\Big ( \|(u_{m}^{m,\nu}-\bar u_{m,m}^{m, \nu})(s,\cdot)\|_{L^\infty}+\|(\bar u_{m,m}^{m, \nu}-\bar u_{m}^{m, \nu})(s,\cdot)\|_{L^\infty} \Big ) \| \partial_{x}\bar u_{m}^{m, \nu}\|_{L^\infty} 
	ds
	\nonumber \\
	&\leq & 
	\bar C 
	\int_0^t 
	\Big (C m^{-\gamma} \| \bar u^{m, \nu}\|_{L^\infty(C^\gamma)}+ \|U_{m}(s,\cdot)\|_{L^\infty}  \Big )
	\| \partial_{x}\bar u_{m}^{m, \nu}\|_{L^\infty} 
	ds,
	\end{eqnarray*}
and from \eqref{Schauder_ineq_Burgers}, we get,
\begin{equation*}
	\|U_{m}(t,\cdot)\|_{L^\infty}
	\leq 
	\bar C 
	\int_0^t 
	\Big (C m^{-\gamma} 	\big (T \|f\|_{L^\infty(C^\gamma)}+ [g]_\gamma \big ) + \|U_{m}(s,\cdot)\|_{L^\infty}  \Big ) \| \partial_{x}\bar u_{m}^{m, \nu}\|_{L^\infty} 
	 ds
	.
\end{equation*}
Also by Gr\"onwall's inequality, we derive
\begin{equation*}
	\|U_{m}(t,\cdot)\|_{L^\infty}
	\leq 
	e^{ C T  \| \partial_{x}\bar u_{m}^{m, \nu}\|_{L^\infty}  } 
	T
	C m^{-\gamma} 	 (T \|f\|_{L^\infty(C^\gamma)}+ [g]_\gamma ) \| \partial_{x}\bar u_{m}^{m, \nu}\|_{L^\infty}.
\end{equation*}
From exponential absorbing property, 
\begin{eqnarray}\label{ineq_Umn_Burgers_UNIQ2}
	\|U_{m}(t,\cdot)\|_{L^\infty}
	&\leq  &
	e^{ C T  \| \partial_{x}\bar u_{m}^{m, \nu}\|_{L^\infty}  } 
	T
	C m^{-\gamma} 	 (T \|f\|_{L^\infty(C^\gamma)}+ [g]_\gamma ) \| \partial_{x}\bar u_{m}^{m, \nu}\|_{L^\infty}
	\nonumber \\
	&\leq  &
	e^{ C T  \| \partial_{x}\bar u_{m}^{m, \nu}\|_{L^\infty}  } 
	m^{-\gamma} 	 (T \|f\|_{L^\infty(C^\gamma)}+ [g]_\gamma ) 
		.
\end{eqnarray}
We need a new \textit{estimate} of the gradient to avoid any blowing up terms in $m$ which cannot be balanced by the $m^{-\gamma}$ in front of the exponential; the singularity has to be in $\nu$.
%

\begin{lemma}\label{lemme_calcul_apriori_ter}
	For $u^{m,\nu}$ solution to \eqref{def_Burgers_viscous}, we have
	\begin{equation*}
		\|\partial_{x} u^{m,\nu} \|_{L^\infty}
		\leq
		2 \Big ( T \|\partial_{x}  f_m\|_{L^\infty}+\|\partial_{x} g_m\|_{L^\infty} \Big ) 
		\exp \Big ( C^2 \nu^{-1} \pi 	(T\|f\|_{L^\infty}+ \|g\|_{L^\infty}) ^2  T  \Big )
	.
	\end{equation*}
\end{lemma}
The proof is postponed in Section \ref{sec_proof_apriori_ter}.

We deduce from \eqref{ineq_Umn_Burgers_UNIQ2},
\begin{eqnarray}
	\|U_{m}(t,\cdot)\|_{L^\infty}
	&\leq&  
\exp \bigg ( C T   
	2 \Big ( T \|\partial_{x}  f_m\|_{L^\infty}+\|\partial_{x} g_m\|_{L^\infty} \Big ) 
\exp \Big ( C^2 \nu^{-1} \pi 	(T\|f\|_{L^\infty}+ \|g\|_{L^\infty}) ^2  T \Big )
 \bigg ) 
\nonumber \\
&& \times 
m^{-\gamma} 	 (T \|f\|_{L^\infty(C^\gamma)}+ [g]_\gamma ) 	,
\nonumber 
\end{eqnarray}
which goes to $0$ under condition \eqref{CONDINU_UNIQ_Burgers_Viscous}.

\mysection{Proofs of some \textit{a priori} controls}
\label{sec_apriori_global}

\subsection{Proof of Lemma \ref{lemma_apriori}}
\label{sec_grad}

\subsubsection{Gradient estimates}

Let us precise the control previously used:
\begin{equation}\label{ineq_O}
\|\nabla u^{m,\nu}(t,\cdot)\|_{L^\infty} \leq \Big ( T \| f_m\|_{L^\infty(C^1)}+\|g_m\|_{C^1} \Big ) \exp \Big ( C m^{1-\tilde \gamma} t \|b\|_{L^\infty( B_{\infty,\infty}^{\tilde \gamma})} \Big )= O_m(t).
\end{equation}
We directly have from Duhamel formula \eqref{Duhamel_u}:
\begin{eqnarray*}
&&|\nabla u^{m,\nu}(t,x) |
\nonumber \\
&\leq& \Big ( t \|\nabla  f_m\|_{L^\infty}+\|\nabla g_m\|_{L^\infty} \Big ) 
\nonumber \\
&&+
\Big |
\int_0^t \int_{\R^d} \nabla \hat{p}^{\tau,\xi} (s,t,x,y) [b_m(s,\theta^m_{s,\tau}(\xi ))-b_m(s,y)] \cdot \nabla u^{m,\nu}(s,y) dy \, ds  \Big | \Bigg |_{(\tau,\xi) =(t,x)}
\nonumber \\
&\leq& \Big ( t \|\nabla  f_m\|_{L^\infty}+\|\nabla g_m\|_{L^\infty} \Big ) 
\nonumber \\
&&+ C
\|b_m\|_{L^\infty(C^1)}
\int_0^t \int_{\R^d} [\nu (t-s)]^{-\frac{1}{2}} \bar p^{\tau,\xi} (s,t,x,y)
|\theta^m_{s,\tau}(\xi )-y| \|\nabla u^{m,\nu}(s,\cdot)\|_{L^\infty} dy \, ds\Big | \Bigg |_{(\tau,\xi) =(t,x)}
.
\end{eqnarray*}
By absorbing property \eqref{ineq_absorb}, we derive
\begin{eqnarray}
\|\nabla \times  u^{m,\nu}(t,\cdot) \|_{L^p}
&\leq& \Big ( t \|\nabla  f_m\|_{L^p}+\|\nabla g_m\|_{L^p} \Big ) 
\nonumber \\
&&+ 
C  \|u^m\|_{L^\infty( C^1)}
\int_0^t  \|\nabla u^{m,\nu}(s,\cdot)\|_{L^p} ds
\end{eqnarray}

Finally, Gr\"onwall's lemma yields the result.

This useful \textit{a priori} control allows to avoid any blow-up when $\nu \to 0$.
However, to be able to prove uniqueness, we also need another estimate stated in Lemma \ref{lemma_apriori_bis} and proved in Section \ref{sec_proof_apriori_bis}.

\subsubsection{Hessian estimates}
\label{sec_D2_u_m_nu}

We perform a similar argument, but for the second derivatives we have to put a second derivatives on $ [b_m(s,\theta_t(\xi ))-b_m(s,y)] \cdot \nabla u(s,y) $.
Indeed, if we twice differentiate $\hat{p}^{\tau,\xi} (s,t,x,y) $ there is no possibility to smoothen the blowing up the contribution  of $\nu$ by H\"older control (or even Lipschitz).

We obtain by Leibniz rules
\begin{eqnarray*}
&&|\nabla^2 u^{m,\nu}(t,x) |
\nonumber \\
&\leq& \Big ( t \|\nabla^2  f_m\|_{L^\infty}+\|\nabla^2 g_m\|_{L^\infty} \Big ) 
+
\Big |
\int_0^t \int_{\R^d} \nabla \hat{p}^{\tau,\xi} (s,t,x,y) \nabla b(s,y) \cdot \nabla u(s,y) dy \, ds  \Big | \Bigg |_{\xi =x}
\nonumber \\
&&+
\Big |
\int_0^t \int_{\R^d} \nabla \hat{p}^{\tau,\xi} (s,t,x,y) [b_m(s,\theta^m_{s,\tau}(\xi ))-b_m(s,y)] \nabla^2 u(s,y) dy \, ds  \Big | \Bigg |_{\xi =x}
\nonumber \\
&\leq& 
m^{2-\gamma}\big ( t \| f\|_{L^\infty(C^\gamma)}+[g]_\gamma\big ) 
\nonumber \\
&&+ C
\int_0^t \int_{\R^d} [\nu (t-s)]^{-\frac{1}{2}} \bar p^{\tau,\xi} (s,t,x,y)
|\theta^m_{s,\tau}(\xi )-y| 	\|\nabla b_m \otimes \nabla u^{m,\nu} (s,\cdot)\|_{L^\infty(C^1)} dy \, ds
\nonumber \\
&&+ C
\|b_m\|_{L^\infty(C^1)}
\int_0^t \int_{\R^d} [\nu (t-s)]^{-\frac{1}{2}} \bar p^{\tau,\xi}(s,t,x,y)
|\theta^m_{s,\tau}(\xi )-y| \|\nabla^2 u^{m,\nu}(s,\cdot)\|_{L^\infty} dy \, ds
.
\end{eqnarray*}
Next, with Leibniz rules and absorbing property \eqref{ineq_absorb}, 
\begin{eqnarray}
&&|\nabla^2 u^{m,\nu}(t,x) |
\nonumber \\
&\leq& 
m^{2-\gamma}\big ( t \| f\|_{L^\infty(C^\gamma)}+[g]_\gamma\big ) 
+ 
C m^{2+\tilde \gamma} \|b\|_{L^\infty( B_{\infty,\infty}^{\tilde \gamma})}
\int_0^t \int_{\R^d} \bar p^{\tau,\xi}(s,t,x,y) \|\nabla u^{m,\nu}(s,\cdot)\|_{L^\infty} dy \, ds
\nonumber \\
&&+ C
m^{1-\tilde \gamma} \|b\|_{L^\infty( B_{\infty,\infty}^{\tilde \gamma})}
\int_0^t  \|\nabla^2 u^{m,\nu}(s,\cdot)\|_{L^\infty} ds
\nonumber \\
&\leq& 	m^{2-\gamma}\big ( t \| f\|_{L^\infty(C^\gamma)}+[g]_\gamma\big ) + C t m^{2+\tilde \gamma} \|b\|_{L^\infty( B_{\infty,\infty}^{\tilde \gamma})} 	\|\nabla u^{m,\nu}\|_{L^\infty}
\nonumber \\
&&+ C
m^{1-\tilde \gamma} \|b\|_{L^\infty( B_{\infty,\infty}^{\tilde \gamma})}
\int_0^t  \|\nabla^2 u^{m,\nu}(s,\cdot)\|_{L^\infty} ds .
\end{eqnarray}
We finally get by Gr\"onwall's lemma and by identity \eqref{ineq_O}
\begin{eqnarray}\label{ineq_D2_u_m_nu}
&&\|\nabla^2 u^{m,\nu}(t,\cdot ) \|_{L^\infty}
\\
&\leq& 
\Big ( 	m^{2-\gamma}\big ( t \| f\|_{L^\infty(C^\gamma)}+[g]_\gamma\big )  + C t m^{2+\tilde \gamma }	\|b\|_{L^\infty( B_{\infty,\infty}^{\tilde \gamma})}	O_m(t)\Big )
\exp(t m^{1-\tilde \gamma} \|b\|_{L^\infty( B_{\infty,\infty}^{\tilde \gamma})}).
\nonumber
\end{eqnarray}
We also write by exponential absorbing property:
\begin{equation}\label{ineq_D2_u_m_nu_FACILE}
\|\nabla^2 u^{m,\nu}(t,\cdot ) \|_{L^\infty}
\leq 
C	 	m^{2-\gamma}\big ( t \| f\|_{L^\infty(C^\gamma)}+[g]_\gamma\big )  
\exp(C  m^{1+\tilde \gamma}t \|b\|_{L^\infty( B_{\infty,\infty}^{\tilde \gamma})})
=: O^{(2)}_m(t).
\end{equation}
We insist on the fact that the above inequality does not depend on $\nu$.

\subsubsection{Third derivatives estimates}
\label{sec_D3_u_m_nu}

In this section, we detail how to control the third derivatives of $u^{m,\nu}$.
We use the same method as for the Hessian, the additional derivative  is also put on $ [b_m(s,\theta_t(\xi ))-b_m(s,y)] \cdot \nabla u(s,y) $.
\begin{eqnarray*}
|\nabla^3 u^{m,\nu}(t,x) |
&\leq& 
\Big ( t \|\nabla^3  f_m\|_{L^\infty}+\|\nabla^3 g_m\|_{L^\infty} \Big ) 
+
\Big |
\int_0^t \int_{\R^d} \nabla \hat{p}^{\tau,\xi} (s,t,x,y) \nabla^2 b(s,y) \cdot \nabla u(s,y) dy \, ds  \Big | \Bigg |_{\xi =x}
\nonumber \\
&& +
\Big |
\int_0^t \int_{\R^d} \nabla \hat{p}^{\tau,\xi} (s,t,x,y) \nabla b(s,y) \cdot \nabla^2 u(s,y) dy \, ds  \Big | \Bigg |_{\xi =x}
\nonumber \\
&&+
\Big |
\int_0^t \int_{\R^d} \nabla \hat{p}^{\tau,\xi} (s,t,x,y) [b_m(s,\theta^m_{s,\tau}(\xi ))-b_m(s,y)] \nabla^3 u(s,y) dy \, ds  \Big | \Bigg |_{\xi =x}.
\end{eqnarray*}
We readily obtain the upper-bound
\begin{eqnarray*}
	|\nabla^3 u^{m,\nu}(t,x) |
&\leq& 
C 	m^{3-\gamma}\big ( t \| f\|_{L^\infty(C^\gamma)}+[g]_\gamma\big ) 
\nonumber \\
&&+ C
\int_0^t \int_{\R^d} [\nu (t-s)]^{-\frac{1}{2}} \bar p^{\tau,\xi} (s,t,x,y)
|\theta^m_{s,\tau}(\xi )-y| 
\nonumber \\
&&\times
(	\|\nabla^2 b_m \cdot  \nabla u^{m,\nu} (s,\cdot)\|_{L^\infty(C^1)} +\|\nabla b_m \cdot  \nabla^2 u^{m,\nu} (s,\cdot)\|_{L^\infty(C^1)} )dy \, ds
\nonumber \\
&&+ C
\|b_m\|_{L^\infty(C^1)}
\int_0^t \int_{\R^d} [\nu (t-s)]^{-\frac{1}{2}} \bar p^{\tau,\xi}(s,t,x,y)
|\theta^m_{s,\tau}(\xi )-y| \|\nabla^3 u^{m,\nu}(s,\cdot)\|_{L^\infty} dy \, ds
,
\end{eqnarray*}
and
\begin{eqnarray*}
	|\nabla^3 u^{m,\nu}(t,x) |
&\leq& C	m^{3-\gamma}\big ( t \| f\|_{L^\infty(C^\gamma)}+[g]_\gamma\big ) + C t m^{3+\tilde \gamma} \|b\|_{L^\infty( B_{\infty,\infty}^{\tilde \gamma})} 	\|\nabla u^{m,\nu}\|_{L^\infty}
\nonumber \\
&&
+ C t m^{2+\tilde \gamma} \|b\|_{L^\infty( B_{\infty,\infty}^{\tilde \gamma})} 	\|\nabla^2 u^{m,\nu}\|_{L^\infty}
+ C
m^{2+\tilde \gamma} \|b\|_{L^\infty( B_{\infty,\infty}^{\tilde \gamma})}
\int_0^t  \|\nabla^3 u^{m,\nu}(s,\cdot)\|_{L^\infty} ds .
\end{eqnarray*}
Eventually, by Gr\"onwall's lemma, identities \eqref{ineq_O} and \eqref{ineq_D2_u_m_nu} imply
\begin{eqnarray}\label{ineq_D3_u_m_nu}
\|\nabla^3 u^{m,\nu}(t,\cdot ) \|_{L^\infty}
&\leq& 
C	\Big ( 	m^{3-\gamma}\big ( t \| f\|_{L^\infty(C^\gamma)}+[g]_\gamma\big )  +  t m^{2+\tilde \gamma }	\|b\|_{L^\infty( B_{\infty,\infty}^{\tilde \gamma})}	\big (m O_m(t)+ mO_m^{(2)}(t) \big )\Big )
\nonumber \\
&&\times
\exp(Ct m^{1-\tilde \gamma} \|b\|_{L^\infty( B_{\infty,\infty}^{\tilde \gamma})})
\nonumber \\
&\leq & 
C	 	m^{3-\gamma}\big ( t \| f\|_{L^\infty(C^\gamma)}+[g]_\gamma\big ) 
\exp(C  m^{1+\tilde \gamma}t \|b\|_{L^\infty( B_{\infty,\infty}^{\tilde \gamma})}),
\end{eqnarray}
by exponential absorbing property.

\subsection{Proof of Lemma \ref{lemma_flot}}
\label{sec_flow}

This section is devoted to the regularity of the flow
\begin{equation*}
\theta_{s,\tau}^m(x):= x+ \int_s^\tau   b_m(\tilde s,\theta_{\tilde s,\tau }^m(x)) d \tilde s .
\end{equation*}
By definition, we have
\begin{eqnarray*}
|\theta_{s,\tau }^m(x)-\theta_{s,\tau}^m(x')|
&\leq & |x-x'|+ 
\Big |\int_s^{\tau}  b_m(\tilde s,\theta_{\tilde s,\tau }^m(x))-b_m(s,\theta_{\tilde s,\tau }^m(x'))  ds \Big |
\nonumber \\
&\leq &
|x-x'|+ \|b_m\|_{L^\infty(C^1)} \int_s^\tau |\theta_{\tilde s,\tau }^m(x)-\theta_{\tilde s,\tau }^m(x)|   d\tilde s ,
\end{eqnarray*}
which is not the suitable inequality to apply directly Gr\"onwall's lemma.
To do so, we use a $\sup$ formulation, namely for any $r \leq  \tau$, we write similarly to above
\begin{eqnarray*}
\sup_{ s \in [0,r]}	
|\theta_{s,\tau }^m(x)-\theta_{s,\tau}^m(x')|
&\leq & |x-x'|+ 
\int_0^{r} \big |  b_m(\tilde s,\theta_{\tilde s,\tau }^m(x))-b_m(s,\theta_{\tilde s,\tau }^m(x')) | ds
\nonumber \\
&\leq &
|x-x'|+ \|b_m\|_{L^\infty(C^1)} \int_0^r |\theta_{\tilde s,\tau }^m(x)-\theta_{\tilde s,\tau }^m(x)|   d\tilde s 
\nonumber \\
&\leq &
|x-x'|+ \|b_m\|_{L^\infty(C^1)} \int_0^r \sup_{ \hat s \in [0,\tilde s]}	 |\theta_{\hat  s,\tau }^m(x)-\theta_{\hat  s,\tau }^m(x)|   d\tilde s .
\end{eqnarray*}
We are now in position to use Gr\"onwall's lemma, for $r=\tau$
\begin{eqnarray*}
\sup_{\tilde s \in [0,\tau ]}	|\theta_{s,\tau }^m(x)-\theta_{s,\tau}^m(x')|
&\leq &
|x-x'|
\exp \big ( \|b_m\|_{L^\infty(C^1)} \tau  \big ).
\end{eqnarray*}

\subsection{Proof of Lemma \ref{lemma_apriori_bis}}
\label{sec_proof_apriori_bis}

\subsubsection{Gradient estimates}
\label{sec_other_gradient_estimate}


By integration by parts
\begin{eqnarray*}
&&|\nabla u^{m,\nu}(t,x) |
\nonumber \\
&\leq& 
t \| \nabla  f_m\|_{L^\infty}+\|\nabla  g_m\|_{L^\infty}
\nonumber \\
&&+
\Big |
\int_0^t \int_{\R^d} \nabla^2 \hat{p}^{\tau,\xi} (s,t,x,y) [b_m(s,\theta^m_{s,\tau}(\xi ))-b_m(s,y)] 
u^{m,\nu}(s,y)
 dy \, ds  \Big | \Bigg |_{(\tau,\xi) =(t,x)}
\nonumber \\
&&+
\Big |
\int_0^t \int_{\R^d} \nabla \hat{p}^{\tau,\xi} (s,t,x,y) \nabla \cdot b_m(s,y)  
u^{m,\nu}(s,y)
dy \, ds  \Big | \Bigg |_{(\tau,\xi) =(t,x)}
,
\end{eqnarray*}
and by exponential absorbing property and because $\nabla \cdot b=0$,  
\begin{eqnarray}\label{ineq_grad_u_bis}
|\nabla u^{m,\nu}(t,x) |
&\leq& 
t \|\nabla  f_m\|_{L^\infty}+\|\nabla  g_m\|_{L^\infty}
\nonumber \\
&&+ C
\|b_m\|_{L^\infty(C^{1})}\| u^{m,\nu}\|_{L^\infty}
\int_0^t \int_{\R^d} [\nu (t-s)]^{-\frac 12 }  \bar p^{\tau,\xi} (s,t,x,y)
dy \, ds \Bigg |_{(\tau,\xi) =(t,x)}
\nonumber \\
&\leq& 
t \|\nabla  f_m\|_{L^\infty}+\|\nabla  g_m\|_{L^\infty}
%
+ C
	m ^{1-\tilde \gamma}	\|b_m\|_{L^\infty(C^{\tilde \gamma})}
\Big ( t \|  f_m\|_{L^\infty}+\| g_m\|_{L^\infty} \Big ) 
	\nu ^{-\frac{1}{2}} 
t^{\frac{1}2}
.
\nonumber \\
\end{eqnarray}

The last identity comes from the 
uniform norm estimates stated in Theorem \ref{THEO_SCHAU_non_bounded_para}.

\subsubsection{Hessian estimates}
\label{sec_other_Hessian_estimate}

Like in Section \ref{sec_other_gradient_estimate}, we integrate by parts
\begin{eqnarray*}
|\nabla^2 u^{m,\nu}(t,x) |
&\leq& 
\Big ( C \nu ^{\frac{\gamma}{2}-1} t^{\frac \gamma 2}  \| f\|_{L^\infty(C^\gamma)}+\|\nabla ^2g_m\|_{L^\infty} \Big )
\nonumber \\
&&+
\Big |
\int_0^t \int_{\R^d} \nabla^3 \hat{p}^{\tau,\xi} (s,t,x,y) [b_m(s,\theta^m_{s,\tau}(\xi ))-b_m(s,y)] 
\nonumber \\
&&
[u^{m,\nu}(s,y)-u^{m,\nu}(s,\theta^m_{s,\tau}(\xi ))] dy \, ds  \Big | \Bigg |_{(\tau,\xi) =(t,x)}
\nonumber \\
&&+
\Big |
\int_0^t \int_{\R^d} \nabla^2 \hat{p}^{\tau,\xi} (s,t,x,y) \nabla \cdot b_m(s,y)  [u^{m,\nu}(s,y)-u^{m,\nu}(s,\theta^m_{s,\tau}(\xi ))]  dy \, ds  \Big | \Bigg |_{(\tau,\xi) =(t,x)}
,
\end{eqnarray*}
and by exponential absorbing property after choosing the \text{freezing} parameters,
\begin{eqnarray}\label{ineq_hessian_u_bis}
&&	|\nabla^2 u^{m,\nu}(t,x) |
\nonumber \\
&\leq& 
\Big ( C \nu ^{\frac{\gamma}{2}-1} t^{\frac \gamma 2}  \| f\|_{L^\infty(C^\gamma)}+\|\nabla ^2g_m\|_{L^\infty} \Big )
\nonumber \\
&&+ C
\|b_m\|_{L^\infty(C^{1})}\| u^{m,\nu}\|_{L^\infty(C^1)} 
\int_0^t \int_{\R^d} [\nu (t-s)]^{-\frac{1}{2}} \bar p^{\tau,\xi} (s,t,x,y)
dy \, ds\Big | \Bigg |_{(\tau,\xi) =(t,x)}
\nonumber \\
&\leq& 
\Big ( C \nu ^{\frac{\gamma}{2}-1} t^{\frac \gamma 2}  \| f\|_{L^\infty(C^\gamma)}+\|\nabla ^2g_m\|_{L^\infty} \Big ) 
\nonumber \\
&& + C
m^{1-\tilde \gamma}\|b_m\|_{L^\infty(C^{\tilde \gamma})}
\nu ^{-\frac{1}{2}} 
t^{\frac{1}{2}} 
\big ( 
t \|\nabla  f_m\|_{L^\infty}+\|\nabla  g_m\|_{L^\infty}
\big ) 
\nonumber \\
&& + C m^{2(1-\tilde \gamma)}
\|b_m\|_{L^\infty(C^{\tilde \gamma})}^2
\big ( t \|  f_m\|_{L^\infty}+\| g_m\|_{L^\infty} \big ) 
\nu ^{- \frac 12} 
t ^{ \frac 12}
.
\nonumber \\
\end{eqnarray}


\subsection{Proof of Lemma \ref{lemme_calcul_apriori_ter}}
\label{sec_proof_apriori_ter}

From the usual Duhamel formula around the heat equation,
\begin{eqnarray*}
|\partial_{x} u^{m,\nu}(t,x) |
&\leq& \Big ( t \|\partial_{x}  f_m\|_{L^\infty}+\|\partial_{x} g_m\|_{L^\infty} \Big ) 
+
\Big |
\int_0^t \int_{\R^d} \partial_{x} \tilde {p} (s,t,x,y)  u_m^{m,\nu }(s,y) \partial_{x} u^{m,\nu}(s,y) dy \, ds  \Big | 
\nonumber \\
&\leq& \Big ( t \|\partial_{x}  f_m\|_{L^\infty}+\|\partial_{x} g_m\|_{L^\infty} \Big ) 
\nonumber \\
&&+ 
C \|u^{m,\nu}\|_{L^\infty}
\int_0^t \int_{\R^d} [\nu (t-s)]^{-\frac{1}{2}} \tilde  p (s,t,x,y)  \|\partial_{x} u^{m,\nu}(s,\cdot)\|_{L^\infty} dy \, ds  
.
\end{eqnarray*}
By the well known $L^\infty$ control, see Section \ref{sec_Linfty}, we obtain 
\begin{eqnarray}\label{ineq_D_u_Burgers_apriori1}
&&	|\partial_{x} u^{m,\nu}(t,x) |
\\
&\leq& \Big ( t \|\partial_{x}  f_m\|_{L^\infty}+\|\partial_{x} g_m\|_{L^\infty} \Big ) 
+ C
(T\|f\|_{L^\infty}+ \|g\|_{L^\infty}) 
\nu^{-\frac 12 }	\int_0^t (t-s)^{-\frac 12}  \|\partial_{x} u^{m,\nu}(s,\cdot)\|_{L^\infty} ds .
\nonumber 
\end{eqnarray}
Here, it is not possible to directly use Gr\"onwall's lemma due to the ``$t$" is in the integral.
We have to consider Gr\"onwall-Henry's lemma, cf. \cite{henr:81} chapter 7 Lemma 7.1.1. $\,$. 
\begin{lemma}[Lemma of Gr\"onwall-Henry]\label{lemma_Gronwall}	
Let $T>0$, a positive a constant  $K>0$ and a non-negative function 
$\alpha$ such that the function
$ \varphi : [0,T] \longrightarrow \R_+$ satisfying 
for any  $0<t<Ts$
\begin{equation}\label{ineq_Gronwall_Lemme}
	\varphi(t) \leq \alpha(t)+ K \int_0^t  (t-s)^{-1+\tilde \gamma} \varphi(s) ds,
\end{equation}
then
\begin{equation}\label{ineq_Gronwall_Henry}
	\varphi(t) \leq \alpha(t) + \theta \int_0^t E_{\tilde \gamma}'(\theta (t-s)) \alpha (s) ds ,
\end{equation}
with for any $r \in [0,T]$, 
\begin{eqnarray*}
	\theta &=& (K \Gamma(\tilde \gamma))^{\frac{1}{\tilde \gamma}},
	\nonumber \\
	E_{\tilde \gamma} (r) &=& \sum_{n=0}^{+\infty} \frac {r^{n\tilde \gamma}}{\Gamma(n\tilde \gamma+1)},
\end{eqnarray*}
and if  $\alpha$ is a non-decreasing function, then, for any $0<t<T$, 
\begin{equation}\label{ineq_Gronwall_Henry_FINAL}
	\varphi(t) \leq \alpha(t)  E_{\tilde \gamma}(\theta t)  .
\end{equation}

\end{lemma}
The last inequality  \eqref{ineq_Gronwall_Henry_FINAL} is readily derived from \eqref{ineq_Gronwall_Henry}.
\\

The only needed case, here, is $\tilde \gamma = \frac 12$, for the sake of completeness, we detail the useful proof of  exercise 1 in \cite{henr:81}.
\begin{lemma}\label{lemma_exo1_henry}
For any $t>0$, 
\begin{equation*}
	e^t\leq 	E_{1/2}(t) \leq 2 e^ t.
\end{equation*}
\end{lemma}
\begin{proof}[Proof of Lemma \ref{lemma_exo1_henry}]
Let us recall that 
$\Gamma(\frac{1}{2})=\sqrt{\pi} $,
then we get, for any $r \geq 0$, by differentiating,
\begin{eqnarray*}
	\partial_r 	E_{1/2}(r) &=& r^{-1}\sum_{n=1}^{+\infty} \frac n2 \frac {r^{\frac{n}{2}}}{\Gamma(\frac n2 +1)}
=
	r^{-1}\sum_{n=1}^{+\infty}  \frac {r^{\frac{n}{2}}}{ \Gamma(\frac n2 )}
	\nonumber \\
	&=&
	(\pi r)^{-\frac 12} +	E_{1/2}(r).
\end{eqnarray*}
By integrating this equation,
for any $t \geq 0$, we get
\begin{equation*}	
	E_{1/2}(t)= e^{t}+ e^{t}  \pi ^{-\frac 12} \int_0^t  r^{-\frac 12} e^{-r} dr.
\end{equation*}
The lower bound of the lemma is direct, the upper-bound comes from
\begin{equation}\label{ineq_E_12}	
	E_{1/2}(t) \leq  e^{t}+ e^{t}  \pi ^{-\frac 12} \int_0^{+ \infty}  r^{-\frac 12} e^{-r} dr
	= e^{t}+ e^{t}  \pi ^{-\frac 12}  \Gamma(\frac{1}{2})= 2 e^t .
\end{equation}
\end{proof}
Coming back to inequality \eqref{ineq_D_u_Burgers_apriori1}, we identify the notations in Lemma \ref{lemma_Gronwall}, $\tilde \gamma= \frac 12$, and 
$\theta= C^2 \nu^{-1} \pi 	(T\|f\|_{L^\infty}+ \|g\|_{L^\infty}) ^2 $, then from Lemmas \ref{lemma_Gronwall} and \ref{lemma_exo1_henry}, 
\begin{equation*}
\|\partial_{x} u^{m,\nu}(t,\cdot) \|_{L^\infty}
\leq 2 \Big ( t \|\partial_{x}  f_m\|_{L^\infty}+\|\partial_{x} g_m\|_{L^\infty} \Big ) 
\exp \Big ( C^2 \nu^{-1} \pi 	(T\|f\|_{L^\infty}+ \|g\|_{L^\infty}) ^2  t  \Big ).
\end{equation*}

\appendix

	\mysection{Convergence of the mollified distribution}
\label{sec_conv_molli_Dpsi}
\begin{PROP}\label{Prop_conv_Besov}
	For any $\psi \in C^\gamma(\R^d,\R)$, $\gamma \in (0,1]$, we have for all $\vartheta \in \N_0^d$ and $\theta \in \N_0$ s.t. $|\vartheta|=\theta$,  that $h_{m^{-2}} \star D^\vartheta \psi \in C^\infty_b $ converges towards $D^\vartheta \psi$ in $\dot B_{\infty, \infty}^{-\theta}$ as $m \to + \infty$. More precisely, we have:
	\begin{equation}
	\| h_{m^{-2}}\star D^\vartheta \psi-D^\vartheta \psi\|_{\dot B_{\infty, \infty}^{-\theta}} \leq C [\psi]_\gamma m^{- \gamma},
	\end{equation}
	and $D^\vartheta \psi \in \dot B_{\infty, \infty}^{\gamma-\theta}$ 
	In particular, if $|\vartheta|=0$, $h_{m^{-2}} \star \psi \in C^\infty_b $ converges towards $\psi$ in $L^\infty$ as $m \to + \infty$, and:
	\begin{equation}\label{ineq_psi_p_psi}
	\| h_{m^{-2}}\star\psi-\psi\|_{L^\infty} \leq C [\psi]_\gamma m^{- \gamma}.
	\end{equation}
\end{PROP}
\begin{remark}
	Actually, $\gamma\not \in \{0,1\}$ is not a restrictive condition as changing $\vartheta$ into $\tilde \vartheta \in \N_0^d$ such that $|\tilde \vartheta| = |\vartheta|+1$ yields the same result.
	
	In particular, Proposition \ref{Prop_conv_Besov} is available for the Dirac distribution $\delta \in \dot B_{\infty,\infty}^{-d}$ regarded as the distributional derivative of the sign function (also regarded as the derivative of the absolute value), and for any derivative of the Dirac distribution by the same argument.
%
\end{remark}

\begin{proof}[Proof of Proposition \ref{Prop_conv_Besov}]
	Let us write $\varphi= D^\vartheta \psi$, with $\psi \in C^\gamma$
	\begin{eqnarray*}
		\| h_{m^{-2}}\star\varphi-\varphi\|_{\dot B_{\infty, \infty}^{-\vartheta}} 
		&=& 
		\| D^\vartheta [h_{m^{-2}}\star\psi-\psi]\|_{\dot B_{\infty, \infty}^{-\vartheta}} 
		\nonumber \\
		&=& 
		\sup_{v \in \R_+} v^{1-\frac {-\vartheta} 2}  \|\partial_v h_{v}\star D^\vartheta [h_{m^{-2}}\star\psi-\psi]\|_{L^\infty}
		\nonumber \\
		&=& 
		\sup_{v \in \R_+} v^{1-\frac {-\vartheta} 2}  \|\partial_v D^\vartheta h_{v}\star  [h_{m^{-2}}\star\psi-\psi]\|_{L^\infty},
	\end{eqnarray*}
	by integration by parts in convolutions.
	Next, we can explicitly write,  
	\begin{eqnarray*}
	&&	\| h_{m^{-2}}\star\varphi-\varphi\|_{\dot B_{\infty, \infty}^{-\vartheta}} 
	\nonumber \\
		&=&
		\sup_{v \in \R_+, \ z \in \R^d } v^{1-\frac {-\vartheta} 2} \Big | \int_{\R^ d} \int_{\R^ d}  \partial_v D^\vartheta h_{v}(z-y) h_{m^{-2}}(y-x) [\psi(x)-\psi(y)] dx \, dy \Big |
		\nonumber \\
		&\leq &  C[\psi]_\gamma
		\sup_{v \in \R_+, \ z \in \R^d } \int_{\R^ d} \int_{\R^ d}  h_{c^{-1}v}(z-y) h_{m^{-2}}(y-x) |x-y|^\gamma dx \, dy 
		\nonumber \\
		&\leq &  C[\psi]_\gamma m^{-\gamma}
		\sup_{v \in \R_+, \ z \in \R^d }  \int_{\R^ d} \int_{\R^ d}  h_{c^{-1}v}(z-y) h_{m^{-2}}(y-x)  dx \, dy,
	\end{eqnarray*}
	by exponential absorbing property \eqref{ineq_absorb}.
	Integrating in space finally yields
	\begin{equation*}
	\| h_{m^{-2}}\star\varphi-\varphi\|_{\dot B_{\infty, \infty}^{-\vartheta}} \leq   C[\psi]_\gamma m^{-\gamma}.
	\end{equation*}
	Inequality \eqref{ineq_psi_p_psi} is direct with similar arguments.
\end{proof}

\begin{cor}\label{Corol_conv_Besov}
	For any $\psi \in \dot B_{\infty,\infty}^\gamma(\R^d,\R)$, $\gamma \in (0,1]$ we have,  for all $\vartheta \in \N_0^d$ and $\theta \in \N_0$ s.t. $|\vartheta|=\theta$, that $h_{m^{-2}} \star D^\vartheta \psi \in C^\infty_b(\R^d,\R) $ converges towards $D^\vartheta \psi$ in $\dot B_{\infty, \infty}^{-\theta+\gamma-\varepsilon}(\R^d,\R)$, for any $\varepsilon \in (0,1)$ s.t. $\theta-\gamma+\varepsilon>0$, as $m \to + \infty$. More precisely, we have:
	\begin{equation}
	\| h_{m^{-2}}\star D^\vartheta \psi-D^\vartheta \psi\|_{\dot B_{\infty, \infty}^{-\theta+\gamma-\varepsilon}} \leq C [\psi]_\gamma m^{- \varepsilon}.
	\end{equation}
\end{cor}

\begin{proof}[Proof of Corollary \ref{Corol_conv_Besov}]
	We still use the thermic representation, and by convolution property we have:
	\begin{eqnarray*}
		&&\| h_{m^{-2}}\star\varphi-\varphi\|_{\dot B_{\infty, \infty}^{-\theta+\gamma-\varepsilon}} 
		\nonumber \\
		&=& 
		\sup_{v \in \R_+} v^{1-\frac {-\theta+\gamma-\theta} 2}  \|\partial_v h_{v}\star D^\vartheta [h_{m^{-2}}\star\psi-\psi]\|_{L^\infty(\R^d)}
		\nonumber \\
		&=& 
		\sup_{v \in \R_+} v^{1-\frac {-\theta+\gamma-\varepsilon} 2}  \|\partial_v D^\vartheta h_{v}\star  [h_{m^{-2}}\star\psi-\psi]\|_{L^\infty(\R^d)}
		\nonumber \\
		&=&
		\sup_{v \in \R_+, \ z \in \R^d } v^{1-\frac {-\theta+\gamma-\varepsilon} 2} \Big | \int_{\R^ d} \int_{\R^ d}  \partial_v D^\vartheta h_{v}(z-y) h_{m^{-2}}(y-x) [\psi(x)-\psi(y)] dx \, dy \Big |.
	\end{eqnarray*}
	For a given $v \in (0,+ \infty)$, we compare the regular contribution $v$ with the mollification contribution. 
	In other words, we consider two possibilities.
	\\
	
	$\bullet$ If $m^{-2} < v$, then:
	\begin{eqnarray}\label{ineq_corol_holder_conv1}
	&&v^{1-\frac {-\theta+\gamma-\varepsilon} 2}  \|\partial_v h_{v}\star D^\vartheta [h_{m^{-2}}\star\psi-\psi]\|_{L^\infty(\R^d)}
	\nonumber \\
	&=&
	\sup_{ \ z \in \R^d } v^{1-\frac {-\theta+\gamma-\varepsilon} 2}  \Big | \int_{\R^ d} \int_{\R^ d}  \partial_v D^\vartheta h_{v}(z-y) h_{m^{-2}}(y-x) [\psi(x)-\psi(y)] dx \, dy \Big |
	\nonumber \\
	&\leq &  C[\psi]_\gamma v^{\frac {-\gamma+\varepsilon} 2} 
	\sup_{  z \in \R^d } \int_{\R^ d} \int_{\R^ d}  h_{c^{-1}v}(z-y) h_{m^{-2}}(y-x) |x-y|^\gamma dx \, dy 
	\nonumber \\
	&\leq &  C[\psi]_\gamma v^{\frac {-\gamma+\varepsilon} 2}  m^{-\gamma}
	\sup_{ z \in \R^d }  \int_{\R^ d} \int_{\R^ d}  h_{c^{-1}v}(z-y) h_{m^{-2}}(y-x)  dx \, dy
	\nonumber \\
	&\leq &  C[\psi]_\gamma m^{-\varepsilon}.
	\end{eqnarray}

	$\bullet$ If $m^{-2} \geq  v$, then:
	\begin{eqnarray}\label{ineq_corol_holder_conv2}
	&& v^{1-\frac {-\theta+\gamma-\varepsilon} 2}  \|\partial_v  h_{v}\star D^\vartheta [h_{m^{-2}}\star\psi-\psi]\|_{L^\infty(\R^d)}
	\nonumber \\
	&=&
	\sup_{ z \in \R^d } v^{1-\frac {-\theta+\gamma-\varepsilon} 2} \Big | \int_{\R^ d} \int_{\R^ d}  \partial_v  h_{v}(z-y)D^\vartheta h_{m^{-2}}(y-x) [\psi(x)-\psi(y)] dx \, dy \Big |
	\nonumber \\
	&\leq &  C[\psi]_\gamma m^{\frac{\theta}{2}} v^{\frac {\theta-\gamma+\varepsilon} 2} 
	\sup_{  z \in \R^d } \int_{\R^ d} \int_{\R^ d}  h_{c^{-1}v}(z-y) h_{c^{-1}m^{-1}}(y-x) |x-y|^\gamma dx \, dy 
	\nonumber \\
	&\leq &  C[\psi]_\gamma m^{\theta-\gamma} v^{\frac {\theta-\gamma+\varepsilon} 2} \sup_{ z \in \R^d }  \int_{\R^ d} \int_{\R^ d}  h_{c^{-1}v}(z-y) h_{m^{-2}}(y-x)  dx \, dy
	\nonumber \\
	&\leq &  C[\psi]_\gamma m^{-\varepsilon}.
	\end{eqnarray}
	
	The result follows from \eqref{ineq_corol_holder_conv1} and \eqref{ineq_corol_holder_conv2}.
\end{proof}

Proposition \ref{Prop_conv_Besov} and Corollary \ref{Corol_conv_Besov} are more precise forms of the well known convergence in the distributional sense.
\begin{PROP}\label{Prop_conv_distribu}
	For any $\psi \in \dot B_{\infty,\infty}^\gamma(\R^d,\R)$, $\gamma \in (0,1]$ we have for any $\vartheta \in \N_0^d$ that $h_{m^{-2}} \star D^\vartheta \psi \in C^\infty_b (\R^d,\R)$ converges towards $D^\vartheta \psi$ in distributional sense as $m \to + \infty$. More precisely, we have for any $\eta \in C_0^\infty(\R^d,\R)$:
	\begin{equation}
	\sup_{x \in \R^d}  \Big | \int_{\R^d} \eta(x-y) \big [ h_{m^{-2}}\star D^\vartheta \psi (y)-D^\vartheta \psi (y) \big ] dy \Big | \leq C [\psi]_\gamma m^{- \gamma}.
	\end{equation}
\end{PROP}
\begin{remark}
We precise that $\eta$ is not supposed to be a Gaussian kernel, as in Proposition \ref{Prop_conv_Besov}.
\end{remark}
\begin{proof}
	We directly write by convolution property:
	\begin{eqnarray*}
		&&
		\sup_{x \in \R^d}  \Big | \int_{\R^d} \eta(x-y) \big [ h_{m^{-2}}\star D^\vartheta \psi (y)-D^\vartheta \psi (y) \big ] dy \Big |
	\nonumber \\
		&= &
		\sup_{x \in \R^d}  \Big | \int_{\R^d} D^\vartheta \eta(x-y) \big [ h_{m^{-2}}\star  \psi (y)- \psi (y) \big ] dy \Big |
		\nonumber \\
		&\leq& Cm^{- \gamma}[\psi]_\gamma
		\sup_{x \in \R^d}  \Big | \int_{\R^d} \big | D^\vartheta \eta(x-y) \big |  dy \Big |
		\nonumber \\
		&\leq& Cm^{- \gamma}[\psi]_\gamma
		,
	\end{eqnarray*}
	the penultimate inequality is consequence of inequality \eqref{ineq_psi_p_psi}.
\end{proof}

\mysection{Properties of derivatives of Besov distributions}
\label{sec_interpom}

%

\begin{PROP}\label{Prop_Df_besov}
	For any $\varphi \in \mathcal S '(\R^d)$ such that $\nabla  \varphi \in \dot B_{\infty,\infty}^{-1+\gamma}(\R^d)$, $\gamma \in (0,1)$, there is a constant $C>1$ such that:
	\begin{equation*}
	 \|\nabla  \varphi \|_{\dot B_{\infty,\infty}^{-1+\gamma}}\leq C	\|\varphi \|_{\dot B_{\infty,\infty}^\gamma}
	.
	\end{equation*}
\end{PROP}

\begin{proof}[Proof of Proposition \ref{Prop_Df_besov}]
	
	We first write by the thermic representation of the Besov norm and by integration by parts,
	\begin{eqnarray*}
		\|\nabla \varphi \|_{\dot B_{\infty,\infty}^{-1+\gamma}}
		&=& \sup_{v \in \R_+, \ z\in \R^d }  v^{1-\frac{-1+\gamma}{2}} \Big |\int_{\R^d} \partial_v h_v(z-y) \nabla \varphi(y) dy \Big |
		\nonumber \\
		&= &  \sup_{v \in \R_+, \ z\in \R^d }  v^{\frac{3-\gamma}{2}} \Big |\int_{\R^d} \nabla \cdot \partial_v h_v(z-y) [\varphi(y)-\varphi(z)]dy \Big |,
	\end{eqnarray*}
by absorbing property \eqref{ineq_absorb} we derive
	\begin{eqnarray*}
	\|\nabla \varphi \|_{\dot B_{\infty,\infty}^{-1+\gamma}}
	&\leq&  C [\varphi]_\gamma \sup_{v \in \R_+, \ z\in \R^d }  v^{\frac{3-\gamma}{2}} \int_{\R^d} v^{-\frac{3}{2}} h_{C^{-1}v} (z-y) |y-z|^\gamma dy 
	\nonumber \\
	&\leq&  C [\varphi]_\gamma \sup_{v \in \R_+, \ z\in \R^d }   \int_{\R^d} h_{C^{-1}v} (z-y) dy  
\nonumber \\
&\leq&  C [\varphi]_\gamma.
\end{eqnarray*}
\end{proof}
We also derive the corresponding inequality for the inhomogeneous case.
\begin{cor}\label{coro_Besov_ineq}
		For any $\varphi \in \mathcal S '(\R^d)$ such that $\nabla  \varphi \in  B_{\infty,\infty}^{-1+\gamma}(\R^d)$, $\gamma \in (0,1)$, there is a constant $c>1$ such that:
	\begin{equation*}
	\|\nabla  \varphi \|_{ B_{\infty,\infty}^{-1+\gamma}}\leq c 	\|\varphi \|_{ B_{\infty,\infty}^\gamma}
	.
	\end{equation*}
\end{cor}
\begin{proof}[Proof of Corollary \ref{coro_Besov_ineq}]
	From inequality \eqref{ineq_Besov_homo_inhomo}, we have $	\|\nabla  \varphi \|_{ B_{\infty,\infty}^{-1+\gamma}} \leq \frac{C}{1-\gamma}	\|\nabla  \varphi \|_{ \dot B_{\infty,\infty}^{-1+\gamma}}$.
	Moreover, it is direct that
	\begin{eqnarray*}
\|\varphi\|_{\dot B_{\infty,\infty}^\gamma}= 
\sup_{v \in \R_+} v^{(1-\frac \gamma 2)}    \|\partial_v h_{v}\star \varphi\|_{L^\infty}
&\leq& \|\varphi\|_{\ddot B_{\infty,\infty}^\gamma}+ \sup_{v >1} v^{(1-\frac \alpha 2)}    \|\partial_v h_{v}\star \varphi\|_{L^\infty}
\nonumber \\
&\leq& \|\varphi\|_{\ddot B_{\infty,\infty}^\gamma}+  \|\varphi\|_{L^\infty}
\nonumber \\
&=& \|\varphi\|_{ B_{\infty,\infty}^\gamma}.
\end{eqnarray*}
In other words, we deduce by Proposition \ref{Prop_Df_besov},
\begin{equation*}
	\|\nabla  \varphi \|_{ B_{\infty,\infty}^{-1+\gamma}} \leq \frac{C}{1-\gamma}	\|\nabla  \varphi \|_{ \dot B_{\infty,\infty}^{-1+\gamma}}
	\leq \|\varphi\|_{\dot B_{\infty,\infty}^\gamma} \leq  C \|\varphi\|_{ B_{\infty,\infty}^\gamma}.
\end{equation*}
\end{proof}

\begin{remark}
	In all generality, the reverse inequality of the above results 
	are not true. 
	For example, any constant function lies in $ B_{\infty,\infty}^\gamma$ but its derivative is $0$, hence the  $B_{\infty,\infty}^{-1+\gamma}$ norm is null, and the corresponding Besov norm equivalence obviously fails to be true.
	
	For an equivalence version of this result, we need to consider extra Besov norms, see for instance \cite{kalt:mayb:mitr:07} identity (3.54) for a Triebel-Lizorkin spaces version.
\end{remark}

	\mysection{On the freedom of the mollification choice}
\label{sec_free_mol}

In this section, we  detail why if there is a sequence of smooth function converging toward a $B_{\infty,\infty}^{\tilde \gamma}$ distribution then the mollification procedure \eqref{def_b_epsilon} converges also toward the distribution.
In other words, if there is a sequence $(\bar b_n)_{n \geq 1}$ lying in $L^\infty([0,T];C^\infty_b(\R^d,\R^d))$ such that
\begin{equation}\label{condi_bn}
\lim_{n \to \infty} \| \bar b_n-b\|_{L^\infty(B_{\infty,\infty}^{-\tilde \gamma+ \varepsilon})}=0, 
\end{equation}
for any $0 < \varepsilon$, 
then 
\begin{equation}\label{conv_moll_bm_b}
\lim_{ m \to \infty} \|  b_m-b\|_{L^\infty(B_{\infty,\infty}^{-\tilde \gamma- \varepsilon})}=0,
\end{equation}
where $b_m$ is defined in \eqref{def_b_epsilon} by:
\begin{equation*}
b_m(t,x)= 
b(t,\cdot)\star \rho_m(x), 
\end{equation*}
with for any $z\in \R^{d}$, $\rho_m(z):=m^{d}\rho(z m) $ for $\rho(z)= \frac{1}{(2 \pi )^{\frac{d}{2}}} e^{-\frac{|z|^2}{2}}$. 
Indeed, we readily write by triangular inequality:
\begin{equation}\label{ineq_conv_bm_b}
\|  b_m-b\|_{L^\infty(B_{\infty,\infty}^{-\tilde \gamma+ \varepsilon})}
\leq 
\|  b_m-\rho_m \star \bar b_n\|_{L^\infty(B_{\infty,\infty}^{-\tilde \gamma -\varepsilon})}
+ \|  \rho_m \star \bar b_n- \bar b_n\|_{L^\infty(B_{\infty,\infty}^{-\tilde \gamma- \varepsilon})}
+ \|   \bar b_n-b\|_{L^\infty(B_{\infty,\infty}^{-\tilde \gamma- \varepsilon})}.
\end{equation}
The firs term in the r.h.s. above write:
\begin{equation*}
	\|  b_m-\rho_m \star \bar b_n\|_{L^\infty(B_{\infty,\infty}^{-\tilde \gamma - \varepsilon})}
= \|  \rho_m \star (b-\bar b_n)\|_{L^\infty(B_{\infty,\infty}^{-\tilde \gamma -\varepsilon})}.
\end{equation*}
Hence, we obtain
\begin{equation}\label{ineq_1_appendic_free}
\|  b_m-\rho_m \star \bar b_n\|_{L^\infty(B_{\infty,\infty}^{-\tilde \gamma - \varepsilon})}
= \sup_{v \in [0,1]} v^{\frac{\tilde \gamma - \varepsilon}{2}}\| \rho_m \star  h_v \star (b-\bar b_n)\|_{L^\infty}
\leq  \|   \bar b_n-b\|_{L^\infty(B_{\infty,\infty}^{-\tilde \gamma - \varepsilon})},
\end{equation}
by triangular inequality.

Also, for the second term in \eqref{ineq_conv_bm_b}, let us deal with the corresponding homogeneous norm, 
\begin{equation*}
 \|  \rho_m \star \bar b_n- \bar b_n\|_{L^\infty( B_{\infty,\infty}^{-\tilde \gamma-\varepsilon})}
=\|\varphi(D)(\rho_m \star \bar b_n-\bar b_n)\|_{L^\infty}
+
\|  \rho_m \star \bar b_n- \bar b_n\|_{L^\infty(\ddot B_{\infty,\infty}^{-\tilde \gamma-\varepsilon})}.
\end{equation*}
It is direct that
\begin{equation}\label{ineq_2_appendic_free1}
	\|\varphi(D)(\rho_m \star \bar b_n-\bar b_n)\|_{L^\infty} \leq C \|\rho_m \star \bar b_n-\bar b_n\|_{L^\infty} \leq C m^{-1} \| D\bar {b}_n \|_{L^ \infty}.
\end{equation}
Next,
\begin{eqnarray}\label{ineq_2_appendic_free}
&&\|  \rho_m \star \bar b_n- \bar b_n\|_{L^\infty(\ddot B_{\infty,\infty}^{-\tilde \gamma-\varepsilon})} 
\nonumber \\
&=&\sup_{v \in [0,1], \ t \in [0,T], \ z \in \R^d } v^{\frac{\tilde \gamma +\varepsilon}{2}} \Big | \int_{\R^d} \int_{\R^d }  h_v(z-y) \rho_m (y-x)  [ \bar b_n(t,x)-\bar b_n(t,y)] dx \ dy \Big |
\nonumber \\
&\leq  & 
\| D\bar {b}_n \|_{L^ \infty}
\sup_{v \in [0,1], \ t \in [0,T], \ z \in \R^d } v^{\frac{\tilde \gamma +\varepsilon}{2}} \Big | \int_{\R^d} \int_{\R^d }    h_{C^{-1}v}(z-y) \rho_m (y-x) | x-y|  dx \ dy \Big |
\nonumber \\
&\leq & 
C\| D\bar {b}_n \|_{L^ \infty} 
m^{-1}
\sup_{v \in [0,1], \ t \in [0,T], \ z \in \R^d } v^{\frac{\tilde \gamma +\varepsilon}{2}} \Big | \int_{\R^d} \int_{\R^d }   h_{C^{-1}v}(z-y)  \rho_{C^{-1}m} (y-x) dx \ dy \Big |
\nonumber \\
&= & 
C \| D\bar {b}_n \|_{L^ \infty}
m^{-1}.
\end{eqnarray}
Let us choose $m \gg \| D\bar {b}_n \|_{L^ \infty}$ which yields that
$\lim_{m,n \to \infty}  \|  \rho_m \star \bar b_n- \bar b_n\|_{L^\infty(B_{\infty,\infty}^{-\tilde \gamma-\varepsilon})}=0.
$ 


Finally, gathering identities \eqref{condi_bn}, \eqref{ineq_conv_bm_b}, \eqref{ineq_1_appendic_free}, \eqref{ineq_2_appendic_free1} and \eqref{ineq_2_appendic_free} yields the limit property \eqref{conv_moll_bm_b}.

\mysection{Comments on the strategy of the \textit{cut locus}}
\label{sec_com}

\subsection{Comments on the necessity of using the norm $\|u^{m,\nu}\|_{L^\infty(C^1)}$}
\label{sec_control_diag_u_C1}


Let us rewrite one of the term of $R_{A}^{\tau,\xi,\xi'}(t,x,x')$ in r.h.s. in \eqref{def_R1_R2_A},
\begin{eqnarray*}
	\mathbf R 
	&:=& \Big | \int_0^t \int_{\R^d}  \nabla \hat {p}^{\tau,\xi} (s,t,0,y)
	\nonumber \\
	&&
	\cdot  
	\big (b_m(s,x'+y )-b_m(s,x+y )\big )
	\big (u^{m,\nu}(s,x'+y)- u^{m,\nu}(s,\theta_{s,\tau}^m(\xi)) \big )
	dy \, ds \Big | \Bigg |_{\tau=t,\xi=\xi'=x}
	\nonumber \\
	&\leq &
	C 	|x-x'| 
	\|b_m\|_{L^\infty(C^{1})}
	\int_0^t \|u^{m, \nu}(s,\cdot )\|_{C^{ \gamma}} \int_{\R^d}  [\nu (t-s)]^{-\frac{1}{2}}  \bar  {p}^{\tau,\xi} (s,t,0,y) \times |\theta^m_{s,\tau}(\xi) -x'-y |^\gamma
	dy \, ds .
\end{eqnarray*}
Next, by exponential absorption,
\begin{eqnarray*}
	\mathbf R 
	&\leq &
	C 	|x-x'| 
	\|b_m\|_{L^\infty(C^{1})}
	\int_0^t \|u^{m, \nu}(s,\cdot )\|_{C^{ \gamma}} \Big (  [\nu (t-s)]^{\frac{\gamma-1}{2}} + [\nu (t-s)]^{-\frac{1}{2}}  |x-x'|^\gamma \Big ) ds 
	\nonumber \\
	&\leq &
	C 	  |x-x'|^\gamma \nu^{\alpha_1(1-\gamma)} 
	\|b\|_{L^\infty( B_{\infty,\infty}^{\tilde \gamma})}
	m^{1-\tilde \gamma} 
	\int_0^t \|u^{m, \nu}(s,\cdot )\|_{C^{ \gamma}} 
	\nonumber \\
	&&\times (t-s)^{(1-\gamma)\alpha_2} \Big (  [\nu (t-s)]^{\frac{\gamma-1}{2}} + [\nu (t-s)]^{-\frac{1}{2}} \nu^{\alpha_1\gamma} (t-s)^{\alpha_2\gamma } \Big )ds.
\end{eqnarray*}
Reordering the contributions yields
\begin{eqnarray*}
	\mathbf R 
	&\leq & 
	C 	  |x-x'|^\gamma 
	\|b\|_{L^\infty( B_{\infty,\infty}^{\tilde \gamma})}
	m^{1-\tilde \gamma} 
	\nonumber \\
	&&\times \int_0^t \|u^{m, \nu}(s,\cdot )\|_{C^{ \gamma}} 
	\Big ( (t-s)^{(1-\gamma)\alpha_2+ \frac{\gamma-1}{2}} \nu^{\alpha_1(1-\gamma)+\frac{\gamma-1}{2}}   +\nu^{\alpha_1-\frac{1}{2} } (t-s)^{\alpha_2-\frac{1}{2}  }\Big )ds.
\end{eqnarray*}
Then 
the required, assumption on parameters is for this control (from the second additive term above)
\begin{eqnarray*}
	\alpha_1 >\frac{1}{2}, \ \alpha_2> -\frac{1}{2},
\end{eqnarray*}
which is incompatible with the 
the \textit{off-diagonal} regime.

We could consider the case $\alpha_1=\frac 12$, but this case yields no viscosity contribution and makes the previous upper-bounds blowing up in $m$ (except for the usual framework, i.e. for $b$ Lipschitz continuous); there is no possibility to obtain a \textit{regularisation by turbulence} for such a choice.

If we suppose that $b$ is $\gamma$-H\"older in space, in order to avoid any blowing-up in $m$,  
 we are able to write
\begin{eqnarray*}
	|\mathbf R|
	&\leq &
	C 	|x-x'| ^\gamma
	\|b\|_{L^\infty(C^{\gamma})}
	\int_0^t \|u^{m, \nu}(s,\cdot )\|_{C^{ \gamma}} 
	\nonumber \\
	&& \int_{\R^d}  \mathds 1_{A(x,x',\nu,t)}(s) [\nu (t-s)]^{-\frac{1}{2}}  \bar  {p}^{\tau,\xi} (s,t,0,y) \times |\theta_{s,\tau}^m(x) -x'-y |^\gamma
	dy \, ds 
	\nonumber \\
	&\leq &
	C 	|x-x'| ^\gamma
	\|b\|_{L^\infty(C^{\gamma})}
	\int_0^t \mathds 1_{A(x,x',\nu,t)}(s) \|u^{m, \nu}(s,\cdot )\|_{C^{ \gamma}} \Big (  [\nu (t-s)]^{\frac{\gamma-1}{2}} + [\nu (t-s)]^{-\frac{1}{2}}  |x-x'|^\gamma\Big ) ds .
\end{eqnarray*}
Hence,
\begin{eqnarray*}
	|\mathbf R|
	&\leq& 
	C 	  |x-x'|^\gamma 
	\|b\|_{L^\infty(C^{\gamma})}
	\int_0^t \mathds 1_{A(x,x',\nu,t)}(s) \|u^{m, \nu}(s,\cdot )\|_{C^{ \gamma}} 
	\nonumber \\
	&& \Big ( [\nu (t-s)]^{\frac{\gamma-1}{2}}  +\nu^{-\frac 12}(t-s)^{-\frac 12}\nu^{\alpha_1\gamma} (t-s)^{\alpha_2\gamma} \Big )ds,
\end{eqnarray*}
which goes to $+ \infty$ when $\nu \to 0$, except if $\gamma=1$.
In other words, we need to consider the norms $	\|b_m\|_{L^\infty(C^1)}$ and $\|u^{m,\nu}\|_{L^\infty(C^{1})}$ on the one hand to smoothen the blowing-up in $\nu$ and to get a suitable control by $|x-x'|$ which allows to overwhelm $\nu$ in the \textit{diagonal} regime. 
%

\subsection{Comments on the choice of \textit{freezing} point for the source functions terms}
\label{sec_comm_diag_G_P}

It is crucial to fix the same \textit{freezing} point for the terms associated with source functions. In our context, it may be unavoidable.
To fully explain this choice, let us develop the computations associated with these terms for the same choice of $\xi$ and $\xi'$ as for $A$, the \textit{off-diagonal} regime, i.e. $\xi=x$ and $\xi'=x'$.


To deal with the semi-group, we consider an analysis of the type (or equivalent controls),
\begin{eqnarray*}
	&&
	|\hat P^{\tau,\xi} g_m (t,x) - \hat P^{\tau,\xi'}  g _m(t,x')| \Big |_{\tau=t,\xi=x,\xi'=x'}
	\nonumber \\
	&=& 	\Big | \int_{\R^d} [ \hat {p}^{\tau,\xi} (0,t,x,y)-\hat {p}^{\tau,\xi' } (0,t,x',y)]	g_m(s,y) dy   \Bigg |_{\tau=t,\xi=x,\xi'=x'}
	\nonumber \\
	&= & 	\Big | \int_{\R^d}  \tilde {p}(0,t,0,y)\Big [g_m \big (t,\theta_{0,t}^m(x)+y \big )-g_m \big (t,\theta_{0,t}^m(x')+y \big ) \Big ]	 dy 
	\nonumber \\
	&\leq & [g]_\gamma|\theta_{0,t}^m(x)-\theta_{0,t}^m(x')|^\gamma.
\end{eqnarray*}
%

Similarly, for the Green operator,
\begin{eqnarray*}
	&&|\hat G^{\tau,\xi} f_m(t,x) - \hat G^{\tau,\xi'}  f_m (t,x')| \Big |_{\tau=t,\xi=\xi'=x}
	\nonumber \\
	&=& 	\Big | \int_0^t  \int_{\R^d} [ \hat {p}^{\tau,\xi} (s,t,x,y)-\hat {p}^{\tau,\xi '} (s,t,x',y)]	f_m(s,y) dy \, ds  \Big | \Bigg |_{\tau=t,\xi=x,\xi'=x'}
	\nonumber \\
	&=& 	\Big | \int_0^t \int_{\R^d} [ \tilde  {p}
	(s,t,0,y)[f_m(s,\theta_{s,t}^m(x)+y)-f_m(s,\theta_{s,t}^m(x')+y)] dy \, ds  \Big |
	\nonumber \\
	&\leq & \|f\|_{L^\infty(C^\gamma)} \int_0^t  |\theta_{s,t}^m(x)-\theta^m_{s,t}(x')|^\gamma  ds
	.
\end{eqnarray*}

In other words, we see in the both controls above  that we only upper-bound by the flow associated with $b_m$ which is \textit{a priori} not controlled uniformly on $m$ in a suitable spatial H\"older space, see \eqref{ineq_theta_peruve}.

\bibliographystyle{alpha}
\bibliography{bibli}

\end{document}